\documentclass[a4paper,reqno]{amsart}
\usepackage[utf8]{inputenc}
\usepackage{graphicx}
\usepackage{amsmath}
\usepackage{amssymb,amsthm}
\usepackage{mathtools}
\usepackage{etoolbox}
\usepackage[T1]{fontenc}
\usepackage{pgfplots}
\usepackage{hyperref,enumerate}
\usepackage{enumitem}
\usepackage{mathabx}

\setlist[enumerate]{label=\alph*),leftmargin=2\parindent}

\theoremstyle{plain}
\newtheorem{theorem}{Theorem}
\numberwithin{theorem}{section}
\newtheorem{prop}[theorem]{Proposition}
\newtheorem{lemma}[theorem]{Lemma}
\newtheorem{cor}[theorem]{Corollary}

\theoremstyle{definition}

\newtheorem{rem}[theorem]{Remark}

\allowdisplaybreaks

\newcommand{\N}{\mathbb{N}}
\newcommand{\Z}{\mathbb{Z}}

\newcommand{\R}{\mathbb{R}}
\newcommand{\C}{\mathbb{C}}
\newcommand{\T}{\mathbb{T}}
\newcommand{\RT}{{\R\times\T}}

\newcommand{\vertiii}[1]{{\left\vert\kern-0.25ex\left\vert\kern-0.25ex\left\vert #1 
    \right\vert\kern-0.25ex\right\vert\kern-0.25ex\right\vert}}

\renewcommand{\phi}{\varphi}
\renewcommand{\epsilon}{\varepsilon}
\newcommand{\eps}{\epsilon}

\DeclarePairedDelimiterX\norm[1]\lVert\rVert{\ifblank{#1}{\cdot}{#1}}
\DeclarePairedDelimiterX\abs[1]\lvert\rvert{\ifblank{#1}{\cdot}{#1}}

\DeclarePairedDelimiter\set\{\}
\DeclarePairedDelimiter\br()
\DeclarePairedDelimiterX\intcc[1][]{\ifblank{#1}{0,1}{#1}}
\DeclarePairedDelimiterX\intco[1][){\ifblank{#1}{0,\infty}{#1}}
\DeclarePairedDelimiterX\intoc[1](]{\ifblank{#1}{-\infty,0}{#1}}
\DeclarePairedDelimiterX\intoo[1](){\ifblank{#1}{0,\infty}{#1}}

\newcommand{\absbig}[1]{\abs[\big]{#1}}
\newcommand{\absBig}[1]{\abs[\Big]{#1}}

\newcommand{\setBig}[1]{\set[\Big]{#1}}

\newcommand{\Fcal}{\mathcal{F}}

\newcommand{\Ucal}{\mathcal{U}}

\newcommand{\ew}{\newpage\noindent}

\newcommand{\les}{\lesssim}
\newcommand{\gts}{\gtrsim}

\begin{document}

\title[GWP and polynomial growth for GZK on a cylinder]
{Low-regularity global well-posedness theory for the generalized Zakharov-Kuznetsov equation on $\R \times \T$ and polynomial growth of higher Sobolev norms}
\author{Jakob Nowicki-Koth}

\address{J.~Nowicki-Koth: Mathematisches Institut der
Heinrich-Heine-Universit{\"a}t D{\"u}sseldorf, Universit{\"a}tsstr. 1,
40225 D{\"u}sseldorf, Germany}
\email{jakob.nowicki-koth@hhu.de}

\begin{abstract}
We address the Cauchy problem for the $k$-generalized Zakharov-Kuznetsov equation ($k$-gZK) posed on $\R^2$ and on $\R \times \T$. By applying established and recently developed linear and bilinear Strichartz-type estimates within the framework of the $I$-method, we obtain the following results:

\begin{itemize}
\item The Zakharov-Kuznetsov equation is globally well-posed in $H^s(\R \times \T)$ for every $s>\frac{11}{13}$.
\item The modified Zakharov-Kuznetsov equation is globally well-posed in $H^s(\R^2)$ for every $s>\frac{2}{3}$ and in $H^s(\R \times \T)$ for every $s>\frac{36}{49}$.
\end{itemize}
Moreover, we show that the $H^s(\R \times \T)$-norm of smooth global real-valued solutions of $k$-gZK grows at most polynomially in time for every $k\geq 1$.
\end{abstract}

\maketitle

\section{Introduction}

For $k \in \N$, $s \in \R$, and $X \in \set{\R \times \T, \R^2}$, we consider the Cauchy problem

\begin{equation}
\partial_t u + \partial_x \Delta_{xy}u = \pm \partial_x \left( u^{k+1} \right), \quad u(t=0) = u_0 \in H^s(X) \tag{CP$_{k, X}$}
\end{equation}

for the $k$-generalized Zakharov-Kuznetsov equation ($k$-gZK), where $u=u(t,x,y)$ is a real-valued function. As the name suggests, the $k$-gZK equation is a generalization of the Zakharov-Kuznetsov equation (ZK), which was originally introduced by Zakharov and Kuznetsov \cite{Zakharov1974} to model the propagation of ion-acoustic waves in a magnetized plasma. \\
In the non-periodic setting, the well-posedness theory of $k$-gZK has attracted significant interest for more than two decades and continues to do so today. We provide an overview of the chronological development of well-posedness results for the modified equation ($k=2$), which is of particular interest here: \\
Biagioni and Linares \cite{Biagioni2003} first established local and global well-posedness in $H^1(\R^2)$, with the global result being valid for $H^1_{xy}$-data of sufficiently small $L^2_{xy}$-norm. Linares and Pastor \cite{Linares2009, Linares2011} were later the first to go below the $H^1_{xy}$ level, proving local well-posedness in $H^s(\R^2)$ for $s > \frac{3}{4}$ and global well-posedness for $s> \frac{53}{63}$ under an additional smallness assumption on the $L^2_{xy}$-norm of the initial data in the focusing case (i.e., the case where a negative sign appears before the nonlinearity).
Afterwards, the local result was improved by Ribaud and Vento \cite{Ribaud2012} to the threshold $s> \frac{1}{4}$ by exploiting the interplay between the Kato smoothing effect and a local-in-time $L^2_xL_{yT}^\infty$ maximal function estimate due to Faminskii \cite{Faminskii}.
Subsequently, Bhattacharya, Farah and Roudenko \cite{Bhattacharya2020} improved the global well-posedness result to $s>\frac{3}{4}$ (again under a similar smallness assumption in the focusing case) by means of the $I$-method, and this was followed by Kinoshita \cite{Kinoshita2019mZK}, who further refined the local theory by establishing local well-posedness at the level of $s=\frac{1}{4}$. Moreover, in the same work he proved that the data-to-solution map fails to be $C^3$ for $s< \frac{1}{4}$, thereby highlighting the critical nature of the threshold $s=\frac{1}{4}$ in the perturbative sense. Most recently, Correia and Kinoshita \cite{Correia2025} complemented the global well-posedness result of Bhattacharya, Farah, and Roudenko by establishing global well-posedness at the scaling-critical regularity $s_c=0$ for small initial data. Their result is proved for initial data in function spaces that scale like $H^s(\R^2)$ and are  specifically tailored to interact well with a $L^4_{txy}$-based smoothing estimate for the symmetrized equation. At this point, it is important to note that their argument applies only to non-quantifiably small initial data. Consequently, the question of global well-posedness for arbitrary initial data in $H^s(\R^2)$ for $s \leq \frac{3}{4}$ remains open. \\
For further results concerning $k$-gZK in the non-periodic setting, we refer the reader to \cite{Biagioni2003, Faminskii, Farah2012, GrünrockHerr2014, Kinoshita2021, Linares2009, Linares2011, Pilod2015, Shan2023} in two spatial dimensions, to \cite{Gruenrock2014, Gruenrock2015, Herr, Kato2018, Linares20093D, Pilod2015, Ribaud20123D} in three spatial dimensions, and to \cite{HerrKino2021, Herr, Kato2018} in dimensions $n\geq 4$. We emphasize that these lists are by no means exhaustive. \\
In contrast to the abundance of results in the non-periodic setting, the semiperiodic case has only begun to receive increased attention in recent years. \\
For the original equation ($k=1$), Linares, Pastor, and Saut \cite{Saut2010} initiated the study by proving local well-posedness in $H^s(\R \times \T)$ for all $s > \frac{3}{2}$. By developing a bilinear smoothing estimate for widely separated frequencies, Molinet and Pilod \cite{Pilod2015} later established global well-posedness for all $s\geq 1$, and Osawa \cite{Osawa2022} then further improved the underlying local result to $s > \frac{9}{10}$. Shortly thereafter, Osawa and Takaoka \cite{Osawa2024} proved global well-posedness for all $s > \frac{29}{31}$ by employing the $I$-method, and most recently Cao-Labora \cite{CaoLabora2025} lowered the local threshold to $s>\frac{3}{4}$ ($s > \frac{1}{2}$ under an additional low frequency condition), showing in the same work that these results are optimal up to the endpoint, in the sense that the associated data-to-solution map fails to be $C^2$ for $s< \frac{3}{4}$ and $s<\frac{1}{2}$, respectively. \\
For all higher-order nonlinearities ($k\geq2$), Farah and Molinet \cite{Farah2024} established local well-posedness in $H^s(\R \times \T)$ for all $s>1-$ (and $s > \frac{5}{6}$ for the modified equation), and also proved global well-posedness for all $s \geq 1$ under certain smallness assumptions on the $L^2_{xy}$- or $H^1_{xy}$-norm of the initial data. By developing new linear and bilinear Strichartz-type estimates, the local well-posedness results were recently improved to $s>\frac{11}{24} $ ($k=2$), $s > \frac{1}{2}$ ($k=3$), and $s> 1-\frac{16}{9k}$ ($k \geq 4$) \cite{JNK2026, JNK2025}, while the question of optimality remains open for all $k \geq 2$. \\
A common feature of most of the global well-posedness results mentioned above is the use of the $I$-method developed by Colliander, Keel, Staffilani, Takaoka, and Tao (see, e.g., \cite{Colliander2001, Keel2003}). The core idea of this method is to regularize solutions corresponding to low-regularity data by means of a smoothing operator, thereby lifting them to the regularity level of a conserved quantity, and to establish an almost conservation law for the resulting modified energy. The existence of suitable conservation laws is therefore essential, and for $k$-gZK, we have the conservation of the mass
\[ M[u](t) \coloneqq \int_{X} (u(t,x,y))^2 \ \mathrm{d}(x,y)  \]
at the $L^2_{xy}$ level, while the conservation of the energy
\[ E_{\pm}[u](t) \coloneqq \frac{1}{2} \int_{X} \abs{\nabla u (t,x,y)}_2^2 \pm \frac{2}{k+2} (u(t,x,y))^{k+2} \ \mathrm{d}(x,y) \]
holds at the level of $H^1_{xy}$. \\
The present article is essentially divided into two parts. In the first part, we revisit the $I$-method described above: \\
The global well-posedness results established in \cite{Osawa2024} and \cite{Bhattacharya2020} are based on Strichartz estimates that are suited to obtaining local well-posedness in $H^s_{xy}$ for $s>\frac{9}{10}$ (ZK on $\R \times \T$) and $s> \frac{1}{2}$ (modified ZK on $\R^2$). Since, within the local theory, it has been possible to push these thresholds down to $s> \frac{3}{4}$ \cite{CaoLabora2025} (ZK on $\R \times \T$), $s \geq \frac{1}{4}$ \cite{Kinoshita2019mZK} (modified ZK on $\R^2$), and more recently to $s> \frac{11}{24}$ \cite{JNK2025} (modified ZK on $\R \times \T$), it is natural to expect that the global results for ZK on $\RT$ and modified ZK on $\R^2$ can be improved, and that a global result below $H^1_{xy}$ may be achieved for modified ZK on $\RT$.
To this end, in the case of mZK on $\R^2$, we will make use of a linear $L^4_{txy}$-estimate which allows for a gain of $\frac{1}{4}$ of a derivative in certain frequency regimes (see \cite{Pilod2015}), whereas for the treatment of ZK and modified ZK on $\RT$, we will make use of the Strichartz-type estimates recently developed in \cite{JNK2025}.
More precisely, in the semiperiodic setting, a bilinear refinement of a linear $L^4_{txy}$-estimate (both obtained in \cite{JNK2025}) will be decisive for the strength of the global results, and the targeted activation of this bilinear smoothing estimate within the framework of the $I$-method will account for a substantial part of our efforts in this article. \\
The precise statements of our global well-posedness results are given in

\begin{theorem} \label{GWPZK}
The Cauchy problem $\mathrm{(CP}_{1, \R \times \T} \mathrm{)}$ is globally well-posed for every $s > \frac{11}{13}$. That is, for every $s > \frac{11}{13}$, every $u_0 \in H^s(\R \times \T)$, and every prescribed lifespan $T > 0$, there exists a unique solution
\[ u \in X_{s,\frac{1}{2}+}^T \]
of $\mathrm{(CP}_{1, \R \times \T} \mathrm{)}$. This solution belongs to $C(\intcc{-T,T},H^s(\R \times \T))$, and for every $T' \in \intoo{0,T}$, there exists a neighborhood $\Ucal \subseteq H^s(\R \times \T)$ of $u_0$ such that the data-to-solution map
\[ S: H^s(\R \times \T) \supseteq \Ucal \rightarrow X_{s,\frac{1}{2}+}^{T'}, \quad v_0 \mapsto S(v_0) \coloneqq v \]
is smooth. Moreover, if $s \leq 1$, we have the following bound on the growth of the $H^s(\R \times \T)$-norm of $u$:
\begin{equation} \label{polygrowthZKbelowH^1} \sup_{t \in \intcc{-T,T}} \norm{u(t)}_{H^s(\R \times \T)} \les (1+T)^{\frac{4(1-s^2)}{13s-11}+}. \end{equation}

\end{theorem}

and

\begin{theorem} \label{GWPmZK}
Let $X \in \set{\R \times \T, \R^2}$ and define
\[ s(X) \coloneqq \begin{cases} \frac{36}{49}, & \text{if} \quad X = \R \times \T, \\ \frac{2}{3}, & \text{if} \quad X= \R^2. \end{cases} \]
Then the Cauchy problem $\mathrm{(CP}_{2, X} \mathrm{)}$ is globally well-posed for all regularities $s > s(X)$. More precisely, let $s > s(X)$ and $u_0 \in H^s(X)$ be given. In the focusing case, assume in addition that
\[ \norm{u_0}_{L^2(X)} < \norm{Q_2}_{L^2(\R^2)}, \]
where $Q_2$ denotes the unique positive radial solution of
\[ \Delta_{xy}Q_2 - Q_2 + Q_2^{3} = 0. \]
Then, for every prescribed lifespan $T>0$, there exists a unique solution
\[ u \in X_{s,\frac{1}{2}+}^T \]
of $\mathrm{(CP}_{2, X} \mathrm{)}$. This solution belongs to $C(\intcc{-T,T},H^s(X))$, and for every $T' \in \intoo{0,T}$, there exists a neighborhood $\Ucal \subseteq H^s(X)$ of $u_0$ such that the data-to-solution map
\[ S: H^s(X) \supseteq \Ucal \rightarrow X_{s,\frac{1}{2}+}^{T'}, \quad v_0 \mapsto S(v_0) \coloneqq v \]
is smooth.
Moreover, if $s \leq 1$, the solutions satisfy the following growth bounds for their respective $H^s(X)$-norms:
\begin{equation} \label{polygrowthmZKbelowH^1}
\sup_{t \in \intcc{-T,T}} \norm{u(t)}_{H^s(X)} \les \begin{cases} (1+T)^{\frac{12s(1-s)}{49s-36}+},  & \text{if} \quad X = \R \times \T, \\ (1+T)^{\frac{2s(1-s)}{3(3s-2)}+}, & \text{if} \quad X = \R^2. \end{cases}
\end{equation}

\end{theorem}

In the second part of the article, we change perspective and consider global solutions in the semiperiodic setting corresponding to initial data of higher regularity $s \gg 1$. Since the global existence of the solution essentially relies on the uniform boundedness of the $H^1_{xy}$-norm, it is natural to ask how the $H^s_{xy}$-norm of the solution behaves for $s \gg 1$. It turns out that, for all $k \geq 1$, the associated $H^s_{xy}$-norm grows at most polynomially in time, and this is made precise in 

\begin{theorem} \label{GrowthgZK}
Let $k \in \N$, $s \in 2\N$,
\[ \alpha (k,s) \coloneqq \begin{cases} 4(s-1), & \text{if} \quad k=1, \\ s-1, & \text{if} \quad k\geq 2, \end{cases} \]
and $u_0 \in H^s(\R \times \T)$. Assume in addition that $u_0$ satisfies the smallness assumption specified in Theorem 1.2 of \ \cite{Farah2024}. Then the $H^s(\R \times \T)$-norm of the associated unique global solution $u \in C(\R,H^s(\R \times \T)) \cap L^\infty(\R, H^1(\R \times \T)) \cap \left( \bigcap_{T > 0} X_{s,\frac{1}{2}+}^T \right)$ to $\mathrm{(CP}_{k, \R \times \T} \mathrm{)}$ grows at most polynomially in time. More precisely, for every $\alpha > \alpha (k,s)$, there exists a constant $C=C(k,\alpha,\sup_{t \in \R} \norm{u(t)}_{H^1}) > 0$ such that
\begin{equation} \label{polygrowthgZKaboveH^1}
\norm{u(t)}_{H^s(\R \times \T)} \leq C (1+\abs{t})^{\alpha}(1+\norm{u_0}_{H^s(\R \times \T)})
\end{equation}
holds for all $t \in \R$.
\end{theorem}

Our considerations on the growth of higher $H^s_{xy}$-norms for $k$-gZK in the semiperiodic setting were inspired by an analogous result of Côte and Valet \cite{Valet2021} for the original ZK equation on $\R^2$, and the proof of Theorem \ref{GrowthgZK} will proceed as follows: Following the ideas of Bourgain \cite{Bourgain1993, Bourgain19932} and Staffilani \cite{Staffilani1997}, we fix $t \in \R$ and derive integral representations for differences of the form
\[ \norm{u(t)}_{\dot{H}^s}^2 - \norm{u(0)}_{\dot{H}^s}^2. \]
With the aid of the most recent local well-posedness results from \cite{CaoLabora2025} and \cite{JNK2025}, we then estimate the resulting integral expressions to obtain
\[ \norm{u(t)}_{H^s} \leq \norm{u(0)}_{H^s} + C \norm{u(0)}_{H^s}^{1-\eps}, \]
and a Grönwall-type iteration argument will ultimately provide polynomial growth bounds depending on $\eps \in \intoo{0,1}$. \\
At this point, we state without proof that the same approach allows one to establish polynomial growth bounds in the non-periodic setting for all higher-oder nonlinearities as well. The general arguments presented in Section 5 also apply on $\R^2$ and one only needs to prove an analogue of Proposition \ref{keypropgrowthk>1}, which can be accomplished by exploiting the interplay between the Kato smoothing estimate, Faminskii's \cite{Faminskii} $L_x^2L_{yT}^\infty$ maximal function estimate, and the $L_x^4L_{yT}^\infty$ maximal function estimate proved in \cite{Linares2009}. Proceeding as described, one can achieve any exponent $\alpha > \frac{4(s-1)}{5}$ in the corresponding growth estimate for all $k\geq2$, which, as one would expect, is better than $\alpha > s-1$ in the semiperiodic case.

\subsection*{Acknowledgments}

This work will be part of the author's PhD thesis, written under the supervision of Axel Grünrock. The author gratefully acknowledges his ongoing guidance and insightful feedback throughout the course of this research.

\section{Preliminaries}

We begin by introducing some notational conventions and defining the function spaces that will be used in this article.

\subsection{Notation}

For $x = (x_1,...,x_n) \in \R^n$, its Euclidean norm will be denoted by $\abs{x}_2^2 \coloneqq x_1^2+...+x_n^2$, so that the Japanese brackets take the form $\langle x \rangle \coloneqq \br{1+\abs{x}_2^2}^\frac{1}{2}$.
If $a \in \R$ is an arbitrary real number, we denote its absolute value by $\abs{a}$, and we use the notation $a+$ to denote a number slightly larger than $a$, and $a-$ to denote a number slightly smaller than $a$.
Within this notation, we also permit $\infty -$, to denote a very large, yet still finite, positive real number. 
Given two real numbers $a$ and $b$, we denote their maximum by $a \lor b$ and their minimum by $a \land b$, for brevity. 
Furthermore, if $a$ and $b$ are positive, we will often write $a \les b$ to indicate that there exists some constant $c>0$ such that $a \leq cb$. 
If the constant $c$ in this inequality can be chosen particularly close to $0$, we write $a \ll b$ to express that $b$ is much larger than $a$. 
Moreover, if both $a \les b$ and $b \les a$ hold, we will use the notation $a \sim b$.
In addition, for later use in Section 5, we denote the length of a multi-index $\alpha = (\alpha_1,\alpha_2) \in \N_0^2$ by $\abs{\alpha} = \alpha_1+\alpha_2$, and we use
\[ D^\alpha \coloneqq \frac{\partial^{\alpha_1} \partial^{\alpha_2}}{\partial x^{\alpha_1} \partial y^{\alpha_2}} \]
to simplify the notation for derivatives.
For a parameter $\lambda \in \intoo{0,\infty}$, we define $\T_\lambda \coloneqq \lambda \T = \R/\br{2\pi \lambda \Z}$ and also allow $\lambda = \infty$, in which case we set $\T_\infty \coloneqq \R$ to maintain uniform notation.
If $f: \T_\lambda \rightarrow \C$ and $g: \Z/\lambda \rightarrow \C$ are admissible functions, we define the Fourier transform of a $2\pi \lambda$-periodic function by
\[ \Fcal^\lambda f(q) \coloneqq \int_{-\pi \lambda}^{\pi \lambda} e^{-ixq} f(x) \ \mathrm{d}x, \quad q \in \Z/\lambda, \]
and the corresponding inverse transform by
\[ \Fcal^{-1, \lambda}g(x) \coloneqq \frac{1}{2\pi \lambda} \sum_{q \in \Z/\lambda} e^{ixq} g(q), \quad x \in \T_\lambda. \]
In the special case $\lambda = \infty$, we choose the constants to match those in the periodic setting and define
\[ \Fcal^\infty f (\xi) \coloneqq \int_\R  e^{-ix\xi} f(x) \ \mathrm{d}x, \quad \xi \in \R, \]
as well as
\[ \Fcal^{-1,\infty} f (x) \coloneqq \frac{1}{2\pi} \int_\R  e^{ix\xi} f(\xi) \ \mathrm{d}\xi, \quad x \in \R.  \]
As we will primarily be working with functions $f: \R \times \R \times \T_\lambda \rightarrow \C$, $g: \R \times \R \times \Z/\lambda \rightarrow \C$, we introduce the following shorthand notation for the mixed Fourier transform and its inverse, respectively:
\[ \widehat{\!f}^\lambda (\tau,\xi,q) \coloneqq \int_{\R^2} \int_{-\pi \lambda}^{\pi \lambda} e^{-i(t \tau + x \xi + y q)} f(t,x,y) \ \mathrm{d}(t,x,y), \quad (\tau,\xi,q) \in \R^2 \times \Z/\lambda, \]
and
\[ \widecheck{\!g}^\lambda(t,x,y) \coloneqq \frac{1}{(2\pi)^3\lambda} \int_{\R^2} \sum_{q \in \Z/\lambda} e^{i(t\tau+x\xi+yq)} g(\tau,\xi,q) \ \mathrm{d}(\tau,\xi), \quad (t,x,y) \in \R^2 \times \T_\lambda, \]
with the obvious modifications in the case $\lambda = \infty$.
When employing partial Fourier transforms, we indicate this by using subscripts, as in $\Fcal_x$, $\Fcal_{ty}^\lambda$, or $\Fcal_{xy}^\lambda$. 
The parameter $\lambda$ is included in the notation whenever $\lambda \notin \set{1, \infty}$, and omitted otherwise when clear from the context.
With the (parametrized) Fourier transform in hand, we define the Riesz and Bessel potential operators of order $-s$, where $s$ is an arbitrary real number:
\[ I^s \coloneqq \Fcal_{xy}^{-1,\lambda} \abs{(\xi,q)}_2^s \Fcal_{xy}^{\lambda}, \qquad J^s \coloneqq \Fcal_{xy}^{-1,\lambda} \langle (\xi,q) \rangle^s \Fcal_{xy}^{\lambda}. \]
As we will often apply these operators with respect to just one spatial variable, we will include subscripts when necessary to avoid confusion.
In addition to that, we make use of the Fourier transform to introduce the following family of unitary operators, associated with the linear part of $k$-gZK:
\[ e^{-t\partial_x \Delta_{xy}} \coloneqq \Fcal_{xy}^{-1,\lambda} e^{it\xi(\xi^2+q^2)} \Fcal_{xy}^{\lambda}, \quad t \in \R. \]
As usual, we denote the occurring phase function by $\phi(\xi,q) \coloneqq \xi (\xi^2+q^2)$, and use it to define the resonance functions $R_{\text{ZK}}$ and $R_{\text{mZK}}$, corresponding to the original and modified ZK equations, respectively:
\[ R_{\text{ZK}}(\xi_1,q_1,\xi,q) \coloneqq \phi(\xi,q) - \phi(\xi_1,q_1) - \phi(\xi-\xi_1,q-q_1), \]
\[ R_{\text{mZK}}(\xi_1,q_1,\xi_2,q_2,\xi,q) \coloneqq \phi(\xi,q) - \phi(\xi_1,q_1) - \phi(\xi_2,q_2) - \phi(\xi-\xi_1-\xi_2,q-q_1-q_2). \]
Lastly, we fix the notation $\abs{(\xi,q)}^2 \coloneqq 3\xi^2+q^2$, as this dilated quantity arises naturally from differentiating $R_{\text{ZK}}$ in the context of a bilinear smoothing estimate, which will be used frequently throughout this article.

\subsection{Function spaces}

Let $\lambda \in \intoc{0,\infty}$ and $X \in \set{ \T_\lambda, \R \times \T_\lambda}$.
For each $s \in \R$, we equip the Sobolev space $H^s(X)$ with the norm 
\[ \norm{f}_{H^s(X)} \coloneqq \norm{J^s f}_{L^2(X)}, \]
while we adopt the convention 
\[ \norm{f}_{\dot{H}^s(X)} \coloneqq \norm{I^s f}_{L^2(X)} \]
for its homogeneous variant.
To reduce notational overhead, we omit the underlying domain in the notation or use subscripts, as in $H_x^s = H^s(\R)$, $H_{\lambda}^s = H^s(\R \times \T_\lambda)$, or $H_{y,\lambda}^s = H^s(\T_\lambda)$.
Moreover, in Section 5, we will dispense entirely with the parameter $\lambda =1$ in the notation and simply write $H^s$ for $H^s(\R \times \T)$ and $\dot{H}^s$ for $\dot{H}^s(\R \times \T)$.
For $\lambda \neq \infty$, the convolution theorem and Parseval's identity take the forms
\begin{align*} \Fcal_{xy}^\lambda (fg)(\xi,q) &= (\Fcal_{xy}^\lambda f \ast \Fcal_{xy}^\lambda g)(\xi,q) \\ & \coloneqq \frac{1}{(2\pi)^2\lambda} \int_\R \sum_{q_1 \in \Z/\lambda} \Fcal_{xy}^\lambda f(\xi-\xi_1,q-q_1) \Fcal_{xy}^\lambda g (\xi_1,q_1) \ \mathrm{d}\xi_1 \end{align*}
and
\[ \int_\R \int_{-\pi \lambda}^{\pi \lambda} f(x,y) \overline{g(x,y)} \ \mathrm{d}(x,y) = \frac{1}{(2\pi)^2\lambda} \int_\R \sum_{q \in \Z/\lambda} \Fcal_{xy}^\lambda f(\xi,q) \overline{\Fcal_{xy}^\lambda g(\xi,q)} \ \mathrm{d}\xi, \]
respectively, allowing us to express the Sobolev norms in terms of Fourier transforms, as in
\[ \norm{f}_{H^s(\R \times \T_\lambda)}^2 = \frac{1}{(2\pi)^2\lambda} \int_\R \sum_{q \in \Z/\lambda} \langle (\xi,q) \rangle^{2s} \abs{\Fcal_{xy}^\lambda f(\xi,q)}^2 \ \mathrm{d} \xi. \]
Now let $T >0$ and $p,q \in \intcc{1,\infty}$. We equip the mixed $L^pL^q$ spaces in time and space with norms defined, for instance, by
\[ \norm{f}_{L_{t}^pL_{xy,\lambda}^q} \coloneqq \left( \int_\R \norm{f(t,\cdot,\cdot)}_{L^q(\R \times \T_\lambda)}^p \ \mathrm{d}t \right)^\frac{1}{p} \]
and
\[ \norm{f}_{L_{Tx}^pL_{y,\lambda}^q} \coloneqq \left( \int_{\intcc{-T,T} \times \R} \norm{f(t,x,\cdot)}_{L^q(\T_\lambda)}^p \ \mathrm{d}(t,x) \right)^\frac{1}{p}, \]
with the usual modifications in the cases $p= \infty$ or $q = \infty$.
We want to emphasize that this list of examples in not intended to be exhaustive, but merely to clarify the notation we adopt for such mixed-norm spaces.
Moreover, in analogy with the notation for the Fourier transform, we often omit the parameter $\lambda$ in the notation of the norm whenever $\lambda \in \set{1,\infty}$, provided there is no possibility of confusion.
To conclude this section, we introduce solution spaces that are specifically adapted to the phase function $\phi$:
Given $s,b \in \R$ and $\lambda \in \intoo{0,\infty}$, we endow the rescaled semiperiodic Bourgain space $X_{s,b,\lambda}$ with the norm
\[ \norm{f}_{X_{s,b,\lambda}}^2 \coloneqq \frac{1}{(2\pi)^3\lambda} \int_{\R^2} \sum_{q \in \Z/\lambda} \langle (\xi,q) \rangle^{2s} \langle \tau - \phi(\xi,q) \rangle^{2b} \abs{\widehat{\! f}^\lambda(\tau,\xi,q)}^2 \ \mathrm{d}(\tau,\xi), \]
while in the non-periodic case, we define the norm via
\[ \norm{f}_{X_{s,b,\infty}}^2 \coloneqq \frac{1}{(2\pi)^3} \int_{\R^3} \langle (\xi,\eta) \rangle^{2s} \langle \tau - \phi(\xi,\eta) \rangle^{2b} \abs{\widehat{\! f}^\infty ( \tau,\xi,\eta)}^2 \ \mathrm{d}(\tau,\xi,\eta). \]
For any $T > 0$, we further define $X_{s,b,\lambda}^T$ to be the space of all restrictions of $X_{s,b,\lambda}$-functions to $\intcc{-T,T} \times \R \times \T_\lambda$, and equip it with the norm
\[ \norm{f}_{X_{s,b,\lambda}^T} \coloneqq \inf \setBig{\norm{\widetilde{f}}_{X_{s,b,\lambda}} \ \Big| \ \widetilde{f} |_{\intcc{-T,T} \times \R \times \T_\lambda} = f \ \text{and} \ \widetilde{f} \in X_{s,b,\lambda}}. \]
Provided there is no risk of ambiguity, we will also suppress the parameter $\lambda$ in the notation for $X_{s,b,\lambda}$ and the associated time-restricted spaces in the special cases $\lambda \in \set{1,\infty}$.

\section{Rescaled Strichartz-type estimates}

As a final preparatory step, we revisit several well-established linear and bilinear estimates of Strichartz type and reformulate them within the framework of the $\lambda$-dependent function spaces introduced above. In doing so, we follow the approach of Osawa and Takaoka \cite{Osawa2024} in spirit, and our starting point is a slightly refined version of a bilinear smoothing estimate, originally due to Molinet and Pilod (see Proposition 3.6 in \cite{Pilod2015}).

\begin{prop} \label{PropMPlambda}
Let $\epsilon > 0$ and $b_1,b_2 > \frac{1}{2}$ be arbitrary. For $\lambda \in \intcc{1,\infty}$, we define the bilinear convolution operator $MP_\lambda(\cdot,\cdot)$ via its Fourier transform by
\begin{align*} &\widehat{\! MP_\lambda(u,v)}^\lambda \left(\tau,\xi,\begin{cases} q, & \text{if} \ \lambda \neq \infty \\ \eta, & \text{if} \ \lambda = \infty \end{cases} \right) \coloneqq \\  & \begin{cases} \frac{1}{\lambda} \displaystyle\int_{\R^2} \displaystyle\sum_{\substack{q_1 \in \Z/\lambda \\ \ast}} \abs{\abs{(\xi_1,q_1)}^2 - \abs{(\xi_2,q_2)}^2}^\frac{1}{2} \widehat{\! u}^\lambda(\tau_1,\xi_1,q_1) \widehat{\! v}^\lambda(\tau_2,\xi_2,q_2) \ \mathrm{d}(\tau_1,\xi_1), & \text{if} \ \lambda \neq \infty \\ \displaystyle\int_{\substack{\R^3 \\ \ast}} \abs{\abs{(\xi_1,\eta_1)}^2 - \abs{(\xi_2,\eta_2)}^2}^\frac{1}{2} \widehat{\! u}^\infty(\tau_1,\xi_1,\eta_1) \widehat{\! v}^\infty(\tau_2,\xi_2,\eta_2) \ \mathrm{d}(\tau_1,\xi_1,\eta_1), & \text{if} \ \lambda = \infty   \end{cases} \end{align*}
where $\ast$ refers to the convolution constraint $(\tau,\xi,q) = (\tau_1+\tau_2,\xi_1+\xi_2,q_1+q_2)$, with an analogous formulation for the real Fourier variables $\eta$, $\eta_1$, and $\eta_2$, in place of $q$, $q_1$, and $q_2$.
Then, the estimate

\begin{equation} \label{MPlambda}
\norm{MP_\lambda(u,v)}_{L^2_{txy,\lambda}} \les_{\epsilon,b_1,b_2} \norm{J_y^{\frac{1}{2}+\epsilon}u}_{X_{0,b_1,\lambda}} \norm{v}_{X_{0,b_2,\lambda}} 
\end{equation}

holds for all $u,v$ with $J_y^{\frac{1}{2}+\epsilon}u \in X_{0,b_1,\lambda}$ and $v \in X_{0,b_2,\lambda}$. Moreover, the implicit constant is independent of the parameter $\lambda$.

\end{prop}

\begin{proof}
For $i \in \set{1,2}$ and arbitrary $\epsilon_i > 0$, we write $b_i = \frac{1}{2} + \epsilon_i$, and fix $\lambda \in \intco{1,\infty}$.
An application of Parseval's identity, followed by the Cauchy-Schwarz inequality and a subsequent application of Fubini's theorem, yields the estimate
\[ \norm{MP_\lambda(u,v)}_{L_{txy,\lambda}^2} \les \left( \sup_{(\tau,\xi,q) \in \R \times \R \times \Z/\lambda} \frac{I_{\tau,\xi,q,\lambda}}{\lambda} \right)^\frac{1}{2} \norm{J_y^{\frac{1}{2}+\epsilon}u}_{X_{0,b_1,\lambda}} \norm{v}_{X_{0,b_2,\lambda}} \]
with
\[ I_{\tau,\xi,q,\lambda} = \int_{\R^2} \sum_{\substack{q_1 \in \Z/\lambda \\ \ast}} \langle q_1 \rangle^{-1-2\epsilon} \abs{\abs{(\xi_1,q_1)}^2 - \abs{(\xi_2,q_2)}^2} \prod_{i=1}^2 \langle \tau_i - \phi(\xi_i,q_i) \rangle^{-1-2\epsilon_i} \ \mathrm{d}(\tau_1,\xi_1). \]
It therefore suffices to show that there exists some constant $C = C(\epsilon,\epsilon_1,\epsilon_2) > 0$ such that
\[ \frac{I_{\tau,\xi,q,\lambda}}{\lambda} \leq C \]
holds uniformly with respect to $\tau$, $\xi$, $q$, and $\lambda$.
We proceed from the innermost to the outermost integration, beginning with the integration with respect to $\tau_1$: 
A direct application of Lemma 4.2 in \cite{Ginibre1997} yields the bound
\begin{align*} &\int_\R \langle \tau_1-\phi(\xi_1,q_1) \rangle^{-1-2\epsilon_1} \langle \tau-\tau_1 - \phi(\xi-\xi_1,q-q_1) \rangle^{-1-2\epsilon_2} \ \mathrm{d}\tau_1 \\ &\les \langle \tau - \phi(\xi_1,q_1) - \phi(\xi-\xi_1,q-q_1) \rangle^{-1-2(\epsilon_1 \land \epsilon_2)}, \end{align*}
and consequently
\begin{align*} \frac{I_{\tau,\xi,q,\lambda}}{\lambda} &\les \frac{1}{\lambda} \sum_{q_1 \in \Z/\lambda} \langle q_1 \rangle^{-1-2\epsilon} \int_\R \abs{\abs{(\xi_1,q_1)}^2-\abs{(\xi-\xi_1,q-q_1)}^2} \\ & \ \ \ \cdot \langle \tau - \phi(\xi_1,q_1) - \phi(\xi-\xi_1,q-q_1) \rangle^{-1-2(\epsilon_1 \land \epsilon_2)} \ \mathrm{d}\xi_1. \end{align*}
The integration with respect to $\xi_1$ can be handled by substituting
\[ \widetilde{\xi_1} = \tau - \phi(\xi_1,q_1) - \phi(\xi-\xi_1,q-q_1) \curvearrowright "\mathrm{d}\widetilde{\xi_1} = (\abs{(\xi-\xi_1,q-q_1)}^2 - \abs{(\xi_1,q_1)}^2) \ \mathrm{d}\xi_1", \]
which results in the cancellation of the Fourier multiplier appearing in the definition of the operator $MP_\lambda(\cdot,\cdot)$.
We thus obtain
\begin{align*} \frac{I_{\tau,\xi,q,\lambda}}{\lambda} &\les \frac{1}{\lambda} \sum_{q_1 \in \Z/\lambda} \langle q_1 \rangle^{-1-2\epsilon} \int_\R \langle \widetilde{\xi_1} \rangle^{-1-2(\epsilon_1 \land \epsilon_2)} \ \mathrm{d}\widetilde{\xi_1} \\ & \les_{\epsilon_1,\epsilon_2} \frac{1}{\lambda} \sum_{\widetilde{q_1} \in \Z} \langle \frac{\widetilde{q_1}}{\lambda} \rangle^{-1-2\epsilon} \end{align*}
and an integral comparison leads us to
\begin{align*} &\leq \frac{1}{\lambda}\left( 1 + \int_\R \langle \frac{y}{\lambda} \rangle^{-1-2\epsilon} \ \mathrm{d}y \right) = \frac{1}{\lambda} \left( 1 + \lambda \int_\R \langle \widetilde{y} \rangle^{-1-2\epsilon} \ \mathrm{d}\widetilde{y} \right) \les_{\epsilon} 1,
\end{align*}
where, in the final step, we made use of the restriction $\lambda \geq 1$ to ensure the uniformity of the estimate with respect to $\lambda$.
This concludes the proof for $\lambda \in \intco{1,\infty}$. The case $\lambda = \infty$ can be treated in a completely analogous manner, with the parameter $\lambda$ not appearing in the estimates at all, and no final comparison with an integral being required.
\end{proof}

\begin{rem} \label{RemPropMP}
\begin{enumerate} \item[(i)] Estimate \eqref{MPlambda} can be understood as a smoothing estimate in the sense that it yields a gain of nearly half a derivative for widely separated wave numbers $(\xi_1,q_1)$ and $(\xi_2,q_2)$.

\item[(ii)] The proof shows that the loss of derivatives on the right-hand side of the inequality can always be shifted onto the low-frequency factor.

\item[(iii)] An alternative application of Lemma 4.2 from \cite{Ginibre1997}, as described in Remark 3.2 of \cite{JNK2025}, also yields the estimate
\[ \norm{MP_\lambda(u,v)}_{L_{txy,\lambda}^2} \les \norm{u}_{X_{1+,\frac{1}{4}+,\lambda}} \norm{v}_{X_{1+,\frac{1}{4}+,\lambda}} \]
in the rescaled setting. This can be combined with \eqref{MPlambda} by means of bilinear interpolation to obtain the following statement, which is intended for use in estimates involving duality: For every $\epsilon > 0$, there exists an $\widetilde{\epsilon} > 0$ such that
\begin{equation} \label{MPlambdadual}
\norm{MP_\lambda(u,v)}_{L_{txy,\lambda}^2} \les \norm{u}_{X_{\frac{1}{2}+\epsilon,\frac{1}{2}-\widetilde{\epsilon},\lambda}} \norm{v}_{X_{\epsilon,\frac{1}{2}-\widetilde{\epsilon},\lambda}}.
\end{equation}
\end{enumerate}

\end{rem}

As key ingredients for multilinear estimates in the case of closely spaced wave numbers, we also want to recall the (nearly optimal) linear $L^4$-estimate from \cite{JNK2025}, along with an associated family of bilinear refinements likewise established there.

\begin{prop} \label{PropL^4lambda}
Let $\lambda \in \intco{1,\infty}$, $\epsilon > 0$, and $b > \frac{1}{2}$ be arbitrary. Moreover, for $\alpha \in \intcc{0,1}$, let $P_\lambda^\alpha$ denote the Fourier projector defined by
\[ P_\lambda^\alpha u \coloneqq \Fcal_{xy}^{-1,\lambda} \chi_{\set{\abs{3\xi^2-q^2} \gts \abs{\xi}^\alpha, \ \abs{\xi} \gts 1}} \Fcal_{xy}^{\lambda}u. \]
Under these assumptions, the estimates
\begin{equation} \label{L^4lambda}
\norm{u}_{L_{txy,\lambda}^4} \les_{\epsilon,b} \norm{u}_{X_{\epsilon,b,\lambda}}
\end{equation}
and
\begin{equation} \label{Bilinreflambda}
\norm{I_x^\frac{\alpha}{4}P_\lambda^\alpha(uv)}_{L_{txy,\lambda}^2} \les_{\epsilon,b} \norm{u}_{X_{\epsilon,b,\lambda}} \norm{v}_{X_{\epsilon,b,\lambda}}
\end{equation}
hold for all $u,v \in X_{\epsilon,b,\lambda}$, with implicit constants independent of the scaling parameter $\lambda$.
\end{prop}

\begin{proof}
The proof of the estimates proceeds in direct analogy with the arguments given in the proofs of Lemma 3.8, Theorem 1.1, and Proposition 3.11 in \cite{JNK2025}, where the estimates to be established were proven in the special case $\lambda = 1$. Only two points require particular attention, both of which have already been illustrated in detail in the proof of Proposition \ref{PropMPlambda}: First, it must be ensured that individual summands arising in the context of integral comparisons do not cause any issues - this is guaranteed by the assumption $\lambda \geq 1$, since it implies $\frac{1}{\lambda} \leq 1$. Second, after carrying out the integral comparisons, an additional substitution $\widetilde{y} = \frac{y}{\lambda}$ must be performed. This change of variables has the effect of cancelling the remaining $\lambda$-dependent parameters that appear in the estimates.
\end{proof}

\begin{rem}
\begin{enumerate} \item[(i)] In addition to the trivial estimate
\begin{equation} \label{trivialL^2}
\norm{u}_{L_{txy,\lambda}^2} = \norm{u}_{X_{0,0,\lambda}},
\end{equation}
the Sobolev embedding theorem\footnote{The Sobolev embedding theorem also holds in the rescaled setting, with the constants in the estimates remaining independent of $\lambda$, provided that $\lambda \geq 1$.} also provides us with the estimate
\begin{equation} \label{trivialLinfty}
\norm{u}_{L_{txy,\lambda}^\infty} \les_{\epsilon,b} \norm{u}_{X_{1+\epsilon,b,\lambda}}, \quad \epsilon > 0, \ b > \frac{1}{2}.
\end{equation}
These two estimates will serve as recurring anchor points in the context of interpolation, as they allow us to obtain versions of known Strichartz estimates that are suitable for use in multilinear estimates by duality.
As a first example, the two estimates
\begin{equation} \label{L^4+lambda}
\norm{u}_{L_{txy,\lambda}^{4+}} \les \norm{u}_{X_{0+,\frac{1}{2}+,\lambda}}
\end{equation}
and
\begin{equation} \label{L^4-lambda}
\norm{u}_{L_{txy,\lambda}^{4-}} \les \norm{u}_{X_{0+,\frac{1}{2}-,\lambda}}
\end{equation}
arise from interpolating \eqref{L^4lambda} with \eqref{trivialLinfty}, and \eqref{L^4lambda} with \eqref{trivialL^2}, respectively.
\item[(ii)] Following a similar approach as in Remark \ref{RemPropMP} (iii), and making use of Lemma 4.2 from \cite{Ginibre1997}, one obtains the trivial estimate
\[ \norm{I_x^\frac{\alpha}{4}P_\lambda^\alpha(uv)}_{L_{txy,\lambda}^2} \les \norm{J_x^\frac{\alpha}{4}u}_{X_{1+,\frac{1}{4}+,\lambda}} \norm{J_x^\frac{\alpha}{4}v}_{X_{0,\frac{1}{4}+,\lambda}}, \]
which - after bilinear interpolation with \eqref{Bilinreflambda} - leads us to the alternate version
\begin{equation} \label{Bilinreflambdadual}
\norm{I_x^\frac{\alpha}{4}P_\lambda^\alpha(uv)}_{L_{txy,\lambda}^2} \les \norm{u}_{X_{0+,\frac{1}{2}-,\lambda}} \norm{v}_{X_{0+,\frac{1}{2}-,\lambda}}
\end{equation}
of \eqref{Bilinreflambda}, suitable for further use.
\item[(iii)] In the context of the linear $L^4$-estimate, the situation is considerably more favorable in the non-periodic setting ($\lambda = \infty$).
In particular, one unconditionally obtains the estimate
\begin{equation} \label{L^4non-periodic}
\norm{u}_{L_{txy}^4} \les \norm{u}_{X_{0,\frac{5}{12}+}}
\end{equation}
without any loss of derivatives, which follows from the Strichartz estimates for free solutions of $k$-gZK established by Linares and Pastor (see Proposition 2.3 in \cite{Linares2009}), in conjunction with the transfer principle (see Lemma 2.3 in \cite{Ginibre1997}) and a subsequent interpolation with \eqref{trivialL^2}.
Moreover, as a consequence of the optimal $L^4$-restriction theorem for polynomial hypersurfaces in $\R^3$ established by Carbery, Kenig, and Ziesler (see Theorem 1.1 in \cite{Carbery2012}), one obtains the estimate
\begin{equation} \label{L^4derivativegain}
\norm{\abs{K(D)}^\frac{1}{8}u}_{L_{txy}^4} \les \norm{u}_{X_{0,\frac{1}{2}+}},
\end{equation}
which can be found in Proposition 3.5 of \cite{Pilod2015}.
Here, the Fourier multiplier $\abs{K(D)}^\frac{1}{8}$ is defined by
\[ \abs{K(D)}^\frac{1}{8} \coloneqq \Fcal_{xy}^{-1} \abs{3\xi^2-\eta^2}^\frac{1}{8} \Fcal_{xy}, \]
so that \eqref{L^4derivativegain} yields a gain of up to $\frac{1}{4}$ of a derivative, provided one stays sufficiently far from the lines $\abs{\eta} = \sqrt{3} \abs{\xi}$ in Fourier space.
Once again, interpolation with \eqref{trivialLinfty} and \eqref{trivialL^2}, respectively, yields useful variants of \eqref{L^4derivativegain}, namely
\begin{equation} \label{L^4derivativegain+}
\norm{\abs{K(D)}^\frac{1}{8}u}_{L_{txy}^{4+}} \les \norm{u}_{X_{0+,\frac{1}{2}+}}
\end{equation}
and
\begin{equation}\label{L^4derivativegain-}
\norm{\abs{K(D)}^\frac{1}{8}u}_{L_{txy}^{4-}} \les \norm{u}_{X_{0+,\frac{1}{2}-}}.
\end{equation}
\end{enumerate}
\end{rem}

To wrap up this preparatory section, we want to transfer some other supplementary linear estimates from \cite{JNK2025} into the rescaled framework. Given that these estimates are derived from well-known one-dimensional results for the non-periodic Airy and periodic Schrödinger propagators, it becomes essential to analyze the behaviour of the Schrödinger estimate under rescaling. It turns out that the rescaled estimate necessarily involves a (harmless) dependence of the constant on the scaling parameter $\lambda$, which we make precise in the proposition that follows (cf. Proposition 2.36 in \cite{Bourgain1993} and Lemma 3.5 in \cite{Megretski2025}).

\begin{prop} \label{Prop1drescaledSchrödingerL^6}
Let $\lambda \in \intco{1,\infty}$ and $\epsilon > 0$ be arbitrary. Then, the following estimate holds for all $u_0 \in H^\epsilon_x(\T_\lambda)$:
\begin{equation} \label{1drescaledSchrödingerL^6}
\norm{e^{it\partial_x^2}u_0}_{L_{tx}^6(\T_{\lambda^2} \times \T_\lambda)} \les_\epsilon \lambda^\epsilon \norm{u_0}_{H^\epsilon_x(\T_\lambda)}.
\end{equation}
\end{prop}

\begin{proof} For dyadic numbers $N \in 2^{\N_0}$, we define the Fourier projectors $P_N$ by setting
\[ P_N \coloneqq \begin{cases} \Fcal_x^{-1,\lambda} \chi_{\set{\abs{\xi} \les 1}} \Fcal_x^{\lambda}, & \text{if} \ N=1 \\ \Fcal_x^{-1,\lambda} \chi_{\set{\abs{\xi} \sim N}} \Fcal_x^{\lambda}, & \textrm{if} \ N \neq 1. \end{cases} \]
It is then sufficient to show that
\[ \norm{e^{it\partial_x^2}P_Nu_0}_{L_{tx}^6(\T_{\lambda^2} \times \T_\lambda)} \les_\epsilon \lambda^\epsilon N^\epsilon \norm{P_Nu_0}_{L_x^2(\T_\lambda)}, \]
as estimate \eqref{1drescaledSchrödingerL^6} follows from this by a subsequent dyadic summation.
We have
\begin{align*} \norm{e^{it\partial_x^2}P_Nu_0}_{L_{tx}^6(\T_{\lambda^2} \times \T_\lambda)}^6 &= \norm{(e^{it\partial_x^2}P_Nu_0)^3}_{L_{tx}^2(\T_{\lambda^2} \times \T_\lambda)}^2 \\ & \sim \frac{1}{\lambda^3} \norm{\Fcal_{tx}^{\lambda^2, \lambda}(e^{it\partial_x^2}P_Nu_0)^3}_{\ell_{\tau \xi}^2(\Z/\lambda^2 \times \Z/\lambda)}^2 \eqqcolon (\ast), \end{align*}
where
\[ \Fcal_{tx}^{\lambda^2, \lambda}f(\tau,\xi) \coloneqq \int_{-\pi \lambda^2}^{\pi \lambda^2} \int_{-\pi \lambda}^{\pi \lambda} e^{-i(t \tau + x \xi)} f(t,x) \ \mathrm{d}(t,x), \]
and taking into account that
\begin{align*} &\Fcal_{tx}^{\lambda^2, \lambda}(e^{it\partial_x^2}P_Nu_0)^3(\tau,\xi) = \\ & \frac{1}{(2\pi\lambda)^2} \sum_{\substack{\xi_1,\xi_2 \in \Z/\lambda \\ \xi_1+\xi_2+\xi_3=\xi}} 2\pi\lambda^2 \delta_{-\tau,\xi_1^2+\xi_2^2+\xi_3^2} \Fcal_x^\lambda(P_Nu_0)(\xi_1)\Fcal_x^\lambda(P_Nu_0)(\xi_2)\Fcal_x^\lambda(P_Nu_0)(\xi_3),   \end{align*}
allows us to further conclude
\[ (\ast) \les \sup_{\substack{(\tau,\xi) \in \Z/\lambda^2 \times \Z/\lambda \\ \abs{\tau} \les N^2, \ \abs{\xi} \les N}} \left( \sum_{\substack{\xi_1,\xi_2 \in \Z/\lambda \\ \xi_1^2+\xi_2^2+(\xi-\xi_1-\xi_2)^2=\tau}} 1 \right) \norm{P_Nu_0}_{L_{x}^2(\T_\lambda)}^6, \]
for which we have invoked the Cauchy-Schwarz inequality and Fubini's theorem.
At this point, we note that
\begin{align*} \sup_{\substack{(\tau,\xi) \in \Z/\lambda^2 \times \Z/\lambda \\ \abs{\tau} \les N^2, \ \abs{\xi} \les N}} \sum_{\substack{\xi_1,\xi_2 \in \Z/\lambda \\ \xi_1^2+\xi_2^2+(\xi-\xi_1-\xi_2)^2=\tau}} 1 &= \sup_{\substack{(\widetilde{\tau},\widetilde{\xi}) \in \Z \times \Z \\ \abs{\widetilde{\tau}} \les \lambda^2 N^2, \ \abs{\widetilde{\xi}} \les \lambda N}} \sum_{\substack{\widetilde{\xi_1},\widetilde{\xi_2} \in \Z \\ \widetilde{\xi_1}^2+\widetilde{\xi_2}^2+(\widetilde{\xi}-\widetilde{\xi_1}-\widetilde{\xi_2})^2=\widetilde{\tau}}} 1 \\ & \les_\epsilon \lambda^\epsilon N^\epsilon, \end{align*}
where the final inequality has been established in the proof of Proposition 2.36 in \cite{Bourgain1993}. This completes the proof.
\end{proof}
With this, we are now in a position to formulate the corresponding estimates for the semiperiodic ZK propagator in the rescaled setting.
\begin{prop} \label{PropASlambda}
Let $\epsilon, T > 0$ be arbitrary, and let $\lambda \in \intco{1,\infty}$. Then, for all $u_0,v_0$ with $J_y^\frac{1}{3}u_0 \in L_{xy,\lambda}^2$ and $J_x^\frac{1}{3}J_y^\epsilon v_0 \in L_{xy,\lambda}^2$, the following two estimates hold:
\begin{equation} \label{freeAiryL^6lambda}
\norm{I_x^\frac{1}{6} e^{-t\partial_x \Delta_{xy}}u_0}_{L_{txy,\lambda}^6} \les \norm{J_y^\frac{1}{3}u_0}_{L_{xy,\lambda}^2}
\end{equation}
and
\begin{equation} \label{freeSchrödingerL^6lambda}
\norm{e^{-t\partial_x\Delta_{xy}}v_0}_{L_{Txy,\lambda}^6} \les_{\epsilon, T} \lambda^\epsilon \norm{J_x^\frac{1}{3}J_y^\epsilon v_0}_{L_{xy,\lambda}^2}.
\end{equation}
Furthermore, let $p \in \intcc{2,6}$ and $b > \frac{1}{2}$ be arbitrary. Then, for all $u$ with $J_y^{\frac{1}{2}-\frac{1}{p}}u \in X_{0,(\frac{3}{2}-\frac{3}{p})b,\lambda}$, the estimate
\begin{equation} \label{AiryL^plambda}
\norm{I_x^{\frac{1}{4}-\frac{1}{2p}}u}_{L_{txy,\lambda}^p} \les_{b} \norm{J_y^{\frac{1}{2}-\frac{1}{p}}u}_{X_{0,(\frac{3}{2}-\frac{3}{p})b,\lambda}}
\end{equation}
holds, and for all time-localized functions $v \in X_{\frac{1}{3}-\frac{2}{3p}+(\frac{1}{2}-\frac{1}{p})\epsilon,(\frac{3}{2}-\frac{3}{p})b,\lambda}$, we find
\begin{equation} \label{optimizedL^plambda}
\norm{v}_{L_{Txy,\lambda}^p} \les_{\epsilon, T, b} \lambda^{(\frac{3}{2}-\frac{3}{p})\epsilon} \norm{v}_{X_{\frac{1}{3}-\frac{2}{3p}+(\frac{1}{2}-\frac{1}{p})\epsilon,(\frac{3}{2}-\frac{3}{p})b,\lambda}}
\end{equation}
to be true.
\end{prop}

\begin{proof}
The proof of \eqref{freeAiryL^6lambda} proceeds in complete analogy to the proof of the corresponding statement in the special case $\lambda = 1$ (see Proposition 3.4 in \cite{JNK2025}), taking into account the fact that the constant in the Sobolev embedding theorem, in the rescaled setting, is independent of $\lambda$ for $\lambda \geq 1$.
The proof of \eqref{freeSchrödingerL^6lambda} follows similarly, by replicating the argument used in the case $\lambda = 1$ (see Proposition 3.3 in \cite{JNK2025}), except that the one-dimensional Schrödinger $L^6$-estimate on $\T$ is replaced by the estimate established in Proposition \ref{Prop1drescaledSchrödingerL^6}.
Estimate \eqref{AiryL^plambda} then follows from \eqref{freeAiryL^6lambda} after an application of the transfer principle and a subsequent interpolation with \eqref{trivialL^2}, while \eqref{optimizedL^plambda} is obtained by minimizing the resulting derivative loss between the transferred $X_{s,b,\lambda}$-versions of \eqref{freeAiryL^6lambda} and \eqref{freeSchrödingerL^6lambda}, yielding an optimized $L^6$-estimate, which is then interpolated with \eqref{trivialL^2} (see Corollary 3.6 in \cite{JNK2025}).
\end{proof}

\begin{rem}

\begin{enumerate} \item[(i)] For later estimates involving duality, we fix
\begin{equation} \label{AiryL^6+lambda}
\norm{I_x^\frac{1}{6}u}_{L_{txy,\lambda}^{6+}} \les \norm{u}_{X_{\frac{1}{3}+,\frac{1}{2}+,\lambda}}
\end{equation}
and
\begin{equation} \label{AiryL^6-lambda}
\norm{I_x^\frac{1}{6}u}_{L_{txy,\lambda}^{6-}} \les \norm{u}_{X_{\frac{1}{3},\frac{1}{2}-,\lambda}},
\end{equation}
as well as
\begin{equation} \label{optimizedL^6+}
\norm{v}_{L_{Txy,\lambda}^{6+}} \les \lambda^{0+} \norm{v}_{X_{\frac{2}{9}+,\frac{1}{2}+,\lambda}}
\end{equation}
and
\begin{equation} \label{optimizedL^6-}
\norm{v}_{L_{Txy,\lambda}^{6-}} \les \lambda^{0+} \norm{v}_{X_{\frac{2}{9}+,\frac{1}{2}-,\lambda}},
\end{equation}
all of which follow from \eqref{AiryL^plambda} and \eqref{optimizedL^plambda} via interpolation with suitably adapted variants of \eqref{trivialLinfty} and \eqref{trivialL^2}.

\item[(ii)] As described in Remark 3.5 of \cite{JNK2025}, the estimate
\begin{equation} \label{AiryL^4L^2}
\norm{I_x^\frac{1}{8}u}_{L_{tx}^4L_{y,\lambda}^2} \les \norm{u}_{X_{0,\frac{3}{8}+,\lambda}}
\end{equation}
emerges naturally as a byproduct of the proof of Proposition \ref{PropASlambda}.
As this variant of \eqref{AiryL^plambda} found several niche applications in the proof of the improved local well-posedness result for mZK in \cite{JNK2025}, it is important that we also have this estimate available in the rescaled framework.
\end{enumerate}

\end{rem}

\section{global well-posedness below $H^1$}

\subsection{The $I$-method: setup and general arguments}

This subsection is devoted to setting up the $I$-method and some associated preliminary considerations that can be carried out simultaneously for both the ZK and mZK equations.
We begin by outlining the approach in its most basic form, as developed by Colliander, Keel, Staffilani, Takaoka, and Tao \cite{Colliander2001}. This blueprint will serve as a guiding framework throughout the remainder of Section 4: \\
Since we do not have direct access to the (conserved) energy $E_{\pm}[u](t)$ for rough local solutions $u(t) \in H^s$ with $s<1$, we introduce a smoothing operator $I_N$, which acts like the identity on low frequencies $\les N$ while damping high frequencies $\gg N$.
For $I_Nu(t)$, the modified energy $E_{\pm}[I_Nu](t)$ is then well-defined, and although it is not conserved in general, we will be able to show that its increment over each time step remains sufficiently small to allow for repeated iteration of the local theory, enabling us to reach  any prescribed lifespan $T > 0$.
We thus aim to establish global existence based on an almost conservation law for the modified energy. \\
The quality of the global result obtained through the method essentially depends on the degree of nonlinear smoothing provided by the (currently) best known thresholds for local well-posedness, while the computational effort associated with the method is closely linked to the question of whether the scaling symmetry of $k$-gZK can be meaningfully exploited or not - more on this will follow later. \\
For $\lambda \in \intcc{1,\infty}$, $N \gg 1$, and $s < 1$, let $m: \R^2 \rightarrow \R^{>0}$ be a smooth, radially decreasing function with
\[ m(\xi,\eta) = \begin{cases} 1, & \text{if} \ \abs{(\xi,\eta)} \leq N \\ \frac{\abs{(\xi,\eta)}^{s-1}}{N^{s-1}}, & \text{if} \ \abs{(\xi,\eta)} \geq 2N. \end{cases} \]
Then, the previously introduced smoothing operator $I_N$ is defined by
\[ I_Nf \coloneqq \Fcal_{xy}^{-1,\lambda} m(\xi,q) \Fcal_{xy}^{\lambda}f, \]
and a straightforward computation yields
\[ \norm{I_Nf}_{H_{\lambda}^1} \les N^{1-s} \norm{f}_{H_{\lambda}^s}, \]
which shows that $I_N: H^s_\lambda \rightarrow H^1_\lambda$ indeed lives up to its designation as a smoothing operator.
At this point, we note the following drawback of the $I_N$ operator: As the cutoff point $\sim N$ moves farther from the origin, the constant appearing in the continuity estimate grows without bound, which would lead to technical difficulties at a later stage in the argument. To circumvent these issues, we now make use of the scaling symmetry of $k$-gZK. \\
For $\lambda \in \intco{1,\infty}$, $k \in \N$, $s < 1$, and $\widetilde{X} \in \set{\R \times \T_\lambda, \R^2}$, we consider the Cauchy problem
\begin{equation} \tag{$\mathrm{CP}_{k,\widetilde{X}}$}
\partial_t u_\lambda + \partial_x \Delta_{xy} u_\lambda = \pm \partial_x \left( u_\lambda^{k+1} \right), \quad u_\lambda(t=0) = u_{0,\lambda} \in H^s(\widetilde{X}),
\end{equation}
where solutions $u$ of $\mathrm{(CP}_{k,X} \mathrm{)}$ on $\intcc{-T,T}$ correspond one-to-one, via the transformation 
\[ u \mapsto u_\lambda, \quad u_\lambda(t,x,y) \coloneqq \lambda^{-\frac{2}{k}} u(\frac{t}{\lambda^3},\frac{x}{\lambda},\frac{y}{\lambda}), \] 
to solutions $u_\lambda$ of $\mathrm{(CP}_{k,\widetilde{X}} \mathrm{)}$ on the interval $\intcc{-\lambda^3T,\lambda^3T}$. The precise advantage of passing to this rescaled problem is the subject of the following

\begin{lemma} \label{lemmarescaledH^1lambda} For $\lambda \in \intco{1,\infty}$, $k \in \N$, $0\leq s <1$, and $u_0 \in H^s_1$, we define $u_{0,\lambda}(x,y) \coloneqq \lambda^{-\frac{2}{k}} u_0(\frac{x}{\lambda},\frac{y}{\lambda})$. Then, the estimates
\begin{equation} \label{lambdaL^2estimate}
\norm{I_Nu_{0,\lambda}}_{L^2_\lambda} \les \lambda^{1-\frac{2}{k}} \norm{u_0}_{L^2_1}
\end{equation}
and
\begin{equation} \label{lambdagradientestimate}
\norm{\nabla I_N u_{0,\lambda}}_{L^2_\lambda} \les N^{1-s} \lambda^{1-\frac{2}{k}-s} \norm{u_0}_{H^s_1}
\end{equation}
hold.
\end{lemma}

\begin{proof} We begin by computing the Fourier transform of $u_{0,\lambda}$. The substitution
\[ (\widetilde{x},\widetilde{y}) = \left( \frac{x}{\lambda},\frac{y}{\lambda} \right) \curvearrowright "\mathrm{d}(\widetilde{x},\widetilde{y}) = \frac{1}{\lambda^2} \mathrm{d}(x,y)" \]
yields
\begin{align*} \Fcal_{xy}^\lambda u_{0,\lambda}(\xi,q) & = \lambda^{-\frac{2}{k}} \int_\R \int_{-\pi \lambda}^{\pi \lambda} e^{-i(\xi x + q y)} u_0 \left( \frac{x}{\lambda},\frac{y}{\lambda} \right) \ \mathrm{d}(x,y) \\ & = \lambda^{2-\frac{2}{k}} \int_\R \int_{-\pi}^{\pi} e^{-i(\lambda \xi \widetilde{x} + \lambda q \widetilde{y})} u_0(\widetilde{x},\widetilde{y}) \ \mathrm{d}(\widetilde{x},\widetilde{y}) \\ &= \lambda^{2-\frac{2}{k}} \Fcal_{xy}^1u_0(\lambda \xi,\lambda q),
\end{align*}
which allows us to conclude
\begin{align*} \norm{I_Nu_{0,\lambda}}_{L^2_\lambda}^2 & \leq \norm{u_{0,\lambda}}_{L^2_\lambda}^2 \\ &\sim \frac{1}{\lambda} \int_\R \sum_{q \in \Z/\lambda} \abs{\Fcal_{xy}^\lambda u_{0,\lambda}(\xi,q)}^2 \ \mathrm{d}\xi \\ & = \lambda^{3-\frac{4}{k}} \int_\R \sum_{q \in \Z/\lambda} \abs{\Fcal_{xy}^1u_{0}(\lambda \xi, \lambda q)}^2 \ \mathrm{d}\xi \\ & = \lambda^{2-\frac{4}{k}} \int_\R \sum_{\widetilde{q} \in \Z} \abs{\Fcal_{xy}^1u_0(\widetilde{\xi},\widetilde{q})}^2 \ \mathrm{d}\widetilde{\xi} \sim \lambda^{2-\frac{4}{k}} \norm{u_0}_{L^2_1}^2,
\end{align*}
where in the penultimate step, we made the substitution $(\widetilde{\xi},\widetilde{q}) = (\lambda \xi, \lambda q)$. This completes the proof of \eqref{lambdaL^2estimate}.
To prove \eqref{lambdagradientestimate}, we invoke Parseval's identity, taking into account the definition of $I_N$, and subsequently insert the expression we previously computed for the Fourier transform of $u_{0,\lambda}$. This leads us to
\begin{align*} \norm{\nabla I_N u_{0,\lambda}}_{L^2_\lambda}^2 & \les \frac{1}{\lambda} \int_\R \sum_{q \in \Z/\lambda} \chi_{\set{\abs{(\xi,q)} \leq 2N}} \abs{(\xi,q)}^2 \abs{\Fcal_{xy}^\lambda u_{0,\lambda}(\xi,q)}^2  \\ & \ \ \ + \chi_{\set{\abs{(\xi,q)} > 2N}} N^{2-2s} \abs{(\xi,q)}^{2s-2} \abs{(\xi,q)}^2 \abs{\Fcal_{xy}^\lambda u_{0,\lambda}(\xi,q)}^2 \ \mathrm{d}\xi \\ & \les \frac{1}{\lambda} N^{2-2s} \int_\R \sum_{q \in \Z/\lambda} \abs{(\xi,q)}^{2s} \abs{\Fcal_{xy}^\lambda u_{0,\lambda}(\xi,q)}^2 \ \mathrm{d}\xi \\ & = \lambda^{3-\frac{4}{k}} N^{2-2s} \int_\R \sum_{q \in \Z/\lambda} \abs{(\xi,q)}^{2s} \abs{\Fcal_{xy}^1u_0(\lambda \xi,\lambda q)}^2 \ \mathrm{d}\xi \\ & = \lambda^{2-\frac{4}{k}-2s} N^{2-2s} \int_\R \sum_{\widetilde{q} \in \Z} \abs{(\widetilde{\xi},\widetilde{q})}^{2s} \abs{\Fcal_{xy}^1u_0(\widetilde{\xi},\widetilde{q})}^2 \ \mathrm{d}\widetilde{\xi} \\ & \les \lambda^{2-\frac{4}{k}-2s} N^{2-2s} \norm{u_0}_{H^s_1}^2,
\end{align*}
and taking the square root on both sides of the inequality finally yields \eqref{lambdagradientestimate}.
\end{proof}

\begin{rem} \label{RemarklambdaZKandlambdamZK} \begin{enumerate} \item[(i)] An analogous proof yields both \eqref{lambdaL^2estimate} and \eqref{lambdagradientestimate} in the non-periodic setting as well. Thus, for $k \in \N$, $\lambda \in \intco{1,\infty}$, and $0 \leq s <1$, we may also rely on the estimates
\[ \norm{I_Nu_{0,\lambda}}_{L_{\infty}^2} \les \lambda^{1-\frac{2}{k}} \norm{u_0}_{L_{\infty}^2} \]
and
\[ \norm{\nabla I_Nu_{0, \lambda}}_{L^2_{\infty}} \les N^{1-s} \lambda^{1-\frac{2}{k}-s} \norm{u_0}_{H^s_{\infty}} \]
in what follows.

\item[(ii)] The estimates obtained in the cases $k=1$ and $k=2$ are of particular relevance, as they correspond to the scaling symmetries of ZK and mZK, respectively.
In the case $k=1$, choosing $\lambda_{\text{ZK}} \coloneqq C_{\text{ZK}} N^{\frac{1-s}{1+s}}$ ($C_{\text{ZK}} \gg 1$) and taking into account \eqref{lambdaL^2estimate} and \eqref{lambdagradientestimate}, we obtain the bound
\begin{equation*} 
\norm{I_Nu_{0,\lambda_{\text{ZK}}}}_{H^1_{\lambda_{\text{ZK}}}} \ll 1,
\end{equation*}
which holds uniformly in $N$.
An analogous conclusion can be enforced in the case $k=2$, where the choice $\lambda_{\text{mZK}} \coloneqq C_{\text{mZK}} N^{\frac{1-s}{s}}$ ($C_{\text{mZK}} \gg 1$) yields both
\begin{equation*} 
\norm{I_Nu_{0,\lambda_{\text{mZK}}}}_{H^1_{\lambda_{\text{mZK}}}} \les \norm{u_0}_{L^2_1}
\end{equation*}
and
\begin{equation*} 
\norm{I_Nu_{0,\lambda_{\text{mZK}}}}_{H^1_\infty} \les \norm{u_0}_{L^2_\infty},
\end{equation*}
both of which also hold uniformly in $N$.

\end{enumerate}
\end{rem}

Next, we aim to transfer the Gagliardo-Nirenberg-type estimate proven by Farah and Molinet on $\R \times \T$ to the rescaled $\R \times \T_\lambda$ setting, as this will help in estimating the potential part of the modified energy.

\begin{lemma} Let $\lambda \in \intco{1,\infty}$, $k \in \N$, $u_0 \in H^1_1$, and set $u_{0,\lambda}(x,y) \coloneqq \lambda^{-\frac{2}{k}} u_0(\frac{x}{\lambda},\frac{y}{\lambda})$. Then, the estimate
\begin{equation} \label{GNRTlambda}
\norm{u_{0,\lambda}}_{L^{k+2}_\lambda}^{k+2} \leq C_{k,\R} \norm{u_{0,\lambda}}_{L^2_\lambda}^2\left(\norm{\nabla u_{0,\lambda}}_{L^2_\lambda}^2 + C_{k,\T}\lambda^{-2} \norm{u_{0,\lambda}}_{L^2_\lambda}^2\right)^{\frac{k}{2}}
\end{equation}
holds true. In this context, $C_{k,\T}$ is a universal positive constant, while $C_{k,\R}$ represents the sharp constant for the inequality
\begin{equation} \label{GNR^2}
\norm{f}_{L^{k+2}_\infty} \leq C_{k,\R} \norm{f}_{L^2_\infty}^2 \norm{\nabla f}_{L^2_\infty}^k, \quad f \in H^1_\infty.
\end{equation}

\end{lemma}

\begin{proof} We have
\begin{align*}
\norm{u_{0,\lambda}}_{L_\lambda^{k+2}}^{k+2} & = \lambda^{-2-\frac{4}{k}} \int_\R \int_{-\pi \lambda}^{\pi \lambda} \absBig{u_0\left(\frac{x}{\lambda}, \frac{y}{\lambda} \right)}^{k+2} \ \mathrm{d}(x,y) \\ & = \lambda^{-\frac{4}{k}} \int_\R \int_{- \pi}^{\pi} \abs{u_0(\widetilde{x},\widetilde{y})}^{k+2} \ \mathrm{d}(\widetilde{x},\widetilde{y}) = \lambda^{-\frac{4}{k}} \norm{u_0}_{L_1^{k+2}}^{k+2},
\end{align*}
and at this point we can employ the estimate developed by Farah and Molinet (see Lemma 1.1 in \cite{Farah2024}) to further obtain
\[ \leq \lambda^{-\frac{4}{k}} C_{k,\R} \norm{u_0}_{L_1^2}^2(\norm{\nabla u_0}_{L_1^2}^2 + C_{k,\T} \norm{u_0}_{L_1^2}^2)^\frac{k}{2}. \]
Now, observing that
\[ \norm{u_{0,\lambda}}_{L_\lambda^2}^2 = \lambda^{2-\frac{4}{k}} \norm{u_0}_{L_1^2}^2 \quad \text{and} \quad \norm{\nabla u_{0,\lambda}}_{L_\lambda^2}^2 = \lambda^{-\frac{4}{k}} \norm{\nabla u_0}_{L_1^2}^2, \]
we conclude that
\[ \norm{u_{0,\lambda}}_{L_\lambda^{k+2}}^{k+2} \leq C_{k,\R} \norm{u_{0,\lambda}}_{L^2_\lambda}^2\left(\norm{\nabla u_{0,\lambda}}_{L^2_\lambda}^2 + C_{k,\T}\lambda^{-2} \norm{u_{0,\lambda}}_{L^2_\lambda}^2\right)^{\frac{k}{2}}, \]
which is exactly what we wanted to show.
\end{proof}

\begin{rem} \label{RemarkonmodifiedH^1normforZKandmZK} In what follows, we denote by $u_{\lambda_\text{ZK}}(t)$ the solution to $\mathrm{(CP}_{1,\R \times \T_{\lambda_{\text{ZK}}}} \mathrm{)}$, and by $u_{\lambda_{\text{mZK}}}(t)$ the solution to $\mathrm{(CP}_{2,\widetilde{X}} \mathrm{)}$ on $\R \times \T_{\lambda_{\text{mZK}}}$ or on $\R^2$, respectively.
Throughout the subsequent computations, constants may change from line to line.

\begin{enumerate} \item[(i)] The sharp constant $C_{k,\R}$ in \eqref{GNR^2} was determined by Weinstein \cite{Weinstein1983} to be
\[ C_{k,\R} = \frac{2^{\frac{k-2}{2}}(k+2)}{k^{\frac{k}{2}} \norm{Q_k}_{L_{\infty}^2}^k}, \]
wherein $Q_k = Q_k(x,y)$, $(x,y) \in \R^2$, denotes the unique positive radial solution of
\[ \Delta_{xy} Q_k - Q_k + Q_k^{k+1} = 0. \]

\item[(ii)] An application of \eqref{GNRTlambda} in the case $k=1$ leads us to
\begin{align*} \norm{\nabla I_N u_{\lambda_{\text{ZK}}}(t)}_{L_{\lambda_\text{ZK}}^2}^2  &\leq 2 E_{\pm}[I_N u_{\lambda_\text{ZK}}](t) + \frac{2}{3} \norm{I_N u_{\lambda_\text{ZK}}(t)}_{L_{\lambda_\text{ZK}}^3}^3 \\ & \leq 2E_{\pm}[I_N u_{\lambda_\text{ZK}}](t) + \frac{2}{3} C_{1,\R} \norm{I_N u_{\lambda_\text{ZK}}(t)}_{L_{\lambda_\text{ZK}}^2}^2 \\ & \ \ \ \cdot \left(\norm{\nabla I_N u_{\lambda_\text{ZK}}(t)}_{L_{\lambda_\text{ZK}}^2}^2 + C_{1,\T} \lambda_\text{ZK}^{-2} \norm{I_N u_{\lambda_\text{ZK}}(t)}_{L_{\lambda_\text{ZK}}^2}^2\right)^\frac{1}{2} \\ & \leq 2E_{\pm}[I_N u_{\lambda_\text{ZK}}](t) + \frac{2}{3} C_{1,\R} \norm{u_{0,\lambda_\text{ZK}}}_{L_{\lambda_{\text{ZK}}}^2}^2 \norm{ \nabla I_N u_{\lambda_\text{ZK}}(t)}_{L_{\lambda_\text{ZK}}^2} \\ & \ \ \ + \frac{2}{3} C_{1,\R}C_{1,\T}\lambda_\text{ZK}^{-2} \norm{u_{0,\lambda_{\text{ZK}}}}_{L_{\lambda_\text{ZK}}^2}^3,
\end{align*}
where in the last step we made use of the conservation of the $L^2_{\lambda_\text{ZK}}$-norm.
Since $\norm{u_{0,\lambda_{\text{ZK}}}}_{L_{\lambda_{\text{ZK}}}^2} = \lambda_{\text{ZK}}^{-1} \norm{u_{0}}_{L_1^2} \ll 1$ and, without loss of generality, \\
$\norm{ \nabla I_N u_{\lambda_{\text{ZK}}}(t)}_{L_{\lambda_\text{ZK}}^2} \gtrsim 1$, we further obtain
\[ \norm{\nabla I_N u_{\lambda_{\text{ZK}}}(t)}_{L_{\lambda_\text{ZK}}^2}^2 \leq C_0\left(\norm{u_{0,\lambda_{\text{ZK}}}}_{L_{\lambda_{\text{ZK}}}^2}\right) + C_1E_{\pm}[I_N u_{\lambda_{\text{ZK}}}](t), \]
and another application of the mass conservation law gives
\begin{equation} \label{controloverH^1ZK}
\norm{I_N u_{\lambda_{\text{ZK}}}(t)}_{H^1_{\lambda_{\text{ZK}}}}^2 \leq C_0 + C_1 \absbig{E_{\pm}[I_N u_{\lambda_{\text{ZK}}}](t) - E_{\pm}[I_N u_{\lambda_{\text{ZK}}}](0)},
\end{equation}
with $C_0 = C_0\left(\norm{u_{0,\lambda_{\text{ZK}}}}_{L_{\lambda_\text{ZK}}^2}, E_{\pm}[I_N u_{\lambda_{\text{ZK}}}](0)\right)$. At this stage, we observe that the constant $C_0 > 0$ can be made arbitrarily small by enlarging $C_{\text{ZK}}$ ($E_\pm[I_Nu_{\lambda_{\text{ZK}}}](0)$ can essentially be controlled by $\norm{I_Nu_{0,\lambda_{\text{ZK}}}}_{H^1_{\lambda_{\text{ZK}}}}$, which becomes small for large $C_{\text{ZK}}$).

\item[(iii)] For the further analysis of mZK, we employ \eqref{GNRTlambda} in the case $k=2$ to obtain
\begin{align*} \norm{\nabla I_N u_{\lambda_{\text{mZK}}}(t)}_{L_{\lambda_{\text{mZK}}}^2}^2 & = 2 E_{\pm}[I_N u_{\lambda_{\text{mZK}}}](t) \mp \frac{1}{2} \norm{I_N u_{\lambda_{\text{mZK}}}(t)}_{L_{\lambda_{\text{mZK}}}^4}^4 \\ & \leq 2 E_{\pm}[I_Nu_{\lambda_{\text{mZK}}}](t) + \frac{1}{2}C_{2,\R} \norm{I_Nu_{\lambda_{\text{mZK}}}(t)}_{L_{\lambda_{\text{mZK}}}^2}^2 \\ & \ \ \ \cdot \left( \norm{\nabla I_N u_{\lambda_{\text{mZK}}}(t)}_{L_{\lambda_{\text{mZK}}}^2}^2 + C_{2,\T} \lambda_{\text{mZK}}^{-2} \norm{I_N u_{\lambda_{\text{mZK}}}(t)}_{L_{\lambda_{\text{mZK}}}^2}^2 \right) \\ & \leq 2 E_{\pm}[I_Nu_{\lambda_{\text{mZK}}}](t) + \frac{1}{2} C_{2,\R} \norm{u_{0}}_{L^2_1}^2 \norm{\nabla I_N u_{\lambda_{\text{mZK}}}(t)}_{L_{\lambda_{\text{mZK}}}^2}^2 \\ & \ \ \ + \frac{1}{2}C_{2,\R}C_{2,\T} \lambda_{\text{mZK}}^{-2} \norm{u_{0}}_{L^2_1}^4,
\end{align*}
where in the final step we used the $L^2_{\lambda_{\text{mZK}}}$-norm conservation together with the relation $\norm{u_{0,\lambda_{\text{mZK}}}}_{L^2_{\lambda_{\text{mZK}}}} = \norm{u_0}_{L_1^2}$.
We now need\footnote{In the defocusing case, the smallness assumption can be omitted, as the modified energy already controls the $L^2_{\lambda_\text{mZK}}$-norm of the gradient.} precisely
\[ 1 - \frac{C_{2,\R}}{2} \norm{u_0}_{L_1^2}^2 > 0 \Leftrightarrow \norm{u_0}_{L_1^2} < \sqrt{\frac{2}{C_{2,\R}}} = \norm{Q_2}_{L_\infty^2} \]
in order to deduce
\[ \norm{\nabla I_N u_{\lambda_{\text{mZK}}}(t)}_{L_{\lambda_{\text{mZK}}}^2}^2 \leq C_0\left(\norm{u_0}_{L_1^2}\right) + C_1 E_{\pm}[I_Nu_{\lambda_{\text{mZK}}}](t), \]
and by once again invoking the mass conservation law, we finally arrive at
\begin{equation} \label{controloverH^1mZKsp}
\norm{I_Nu_{\lambda_{\text{mZK}}}(t)}_{H^1_{\lambda_{\text{mZK}}}}^2 \leq C_0 + C_1\absbig{E_{\pm}[I_N u_{\lambda_{\text{mZK}}}](t) - E_{\pm}[I_N u_{\lambda_{\text{mZK}}}](0)},
\end{equation}
with $C_0 = C_0\left(\norm{u_0}_{L_1^2}, E_\pm[I_N u_{\lambda_{\text{mZK}}}](0)\right).$
In the non-periodic case, the computation proceeds in exactly the same way, and using \eqref{GNR^2} instead of \eqref{GNRTlambda}, one obtains
\begin{equation} \label{controloverH^1mZKnp}
\norm{I_Nu_{\lambda_{\text{mZK}}}(t)}_{H^1_\infty}^2 \leq C_0 + C_1\absbig{E_{\pm}[I_N u_{\lambda_{\text{mZK}}}](t) - E_{\pm}[I_N u_{\lambda_{\text{mZK}}}](0)},
\end{equation}
with $C_0 = C_0\left(\norm{u_0}_{L_\infty^2}, E_\pm[I_N u_{\lambda_{\text{mZK}}}](0)\right).$
\end{enumerate}

\end{rem}

We now turn to the difference $E_\pm[I_Nu_\lambda](t) - E_\pm[I_N u_\lambda](0)$ and derive an identity that renders it more amenable to further estimates (cf. \cite{Osawa2024} for the special case $k=1$); this is the content of the following

\begin{lemma} \label{modifiedenergylemma}
Let $\lambda \in \intco{1,\infty}$, $k \in \N$, $\widetilde{X} \in \set{\R \times \T_\lambda, \R^2}$, and let $u_\lambda = u_\lambda(t,x,y)$ denote the solution of $\mathrm{(CP}_{k,\widetilde{X}} \mathrm{)}$ on $\widetilde{X}$ with initial data $u_{0,\lambda} \in H^s(\widetilde{X})$, where $s<1$ is such that the solution exists.
Then, for every $\delta \in \R$ for which the solution exists, it holds that
\begin{align}
&E_{\pm}[I_Nu_\lambda](\delta) - E_{\pm}[I_Nu_\lambda](0) = \notag \\ & \mp \int_{0}^{\delta} \int_{\widetilde{X}} \Delta I_N u_\lambda \partial_x \left( I_N(u_\lambda^{k+1}) - (I_Nu_\lambda)^{k+1} \right) \ \mathrm{d}(x,y) \mathrm{d}t \label{modifiedenergydifference} \\ & + \int_{0}^{\delta} \int_{\widetilde{X}} I_N(u_\lambda^{k+1}) \partial_x \left( I_N(u_\lambda^{k+1}) - (I_Nu_\lambda)^{k+1} \right) \ \mathrm{d}(x,y) \mathrm{d}t. \notag
\end{align}
\end{lemma}

\begin{proof} For the following calculations, we assume that $u_\lambda$ is smooth for each fixed $t \in \intcc{-\abs{\delta}, \abs{\delta}}$ and rapidly decreasing in the non-periodic directions of propagation. The general case then follows by approximation, taking into account the continuous dependence of the solution on the initial data. \\
An application of the fundamental theorem of calculus yields
\begin{equation} \label{fundamentalcalc}
E_\pm[I_Nu_\lambda](\delta) - E_\pm[I_N u_\lambda](0) = \int_{0}^{\delta} \frac{\mathrm{d}}{\mathrm{d}t} E_\pm[I_N u_\lambda](t) \ \mathrm{d}t,
\end{equation}
where
\[ \frac{\mathrm{d}}{\mathrm{d}t} E_\pm[I_Nu_\lambda](t) = \int_{\widetilde{X}} \langle \nabla I_N \frac{\mathrm{d}}{\mathrm{d}t} u_\lambda, \nabla I_N u_\lambda \rangle_2 \pm (I_N u_\lambda)^{k+1} I_N \frac{\mathrm{d}}{\mathrm{d}t}u_\lambda \ \mathrm{d}(x,y), \]
and an integration by parts applied to the first term gives
\[ = \int_{\widetilde{X}} \left( \pm (I_N u_\lambda)^{k+1} - \Delta I_N u_\lambda \right) I_N \frac{\mathrm{d}}{\mathrm{d}t} u_\lambda \ \mathrm{d}(x,y). \]
Using that $u_\lambda$ solves $\mathrm{(CP}_{k,\widetilde{X}} \mathrm{)}$, we further obtain
\begin{align*}
& = \int_{\widetilde{X}} \left( \Delta I_N u_\lambda \mp (I_N u_\lambda)^{k+1} \right) \partial_x \left( \Delta I_N u_\lambda \mp I_N (u_\lambda^{k+1}) \right) \ \mathrm{d}(x,y) \\ & = \int_{\widetilde{X}} \partial_x \Delta I_N u_\lambda \Delta I_N u_\lambda \ \mathrm{d}(x,y) \\ & \ \ \ + \int_{\widetilde{X}} \mp(I_N u_\lambda)^{k+1} \partial_x \Delta I_N u_\lambda \mp \partial_x I_N (u_\lambda^{k+1}) \Delta I_N u_\lambda \ \mathrm{d}(x,y) \\ & \ \ \ + \int_{\widetilde{X}} (I_N u_\lambda)^{k+1} \partial_x I_N(u_\lambda^{k+1}) \ \mathrm{d}(x,y) \\ & \eqqcolon (I) + (II) + (III),
\end{align*}
and these three integrals are now to be treated individually.
For the first integral, the fundamental theorem of calculus immediately gives
\[ (I) = \int_{\widetilde{X}} \frac{1}{2} \partial_x(\Delta I_N u_\lambda)^2 \ \mathrm{d}(x,y) = 0,  \]
and the second integral takes the form
\[ (II) = \mp \int_{\widetilde{X}} \Delta I_N u_\lambda \partial_x \left( I_N (u_\lambda^{k+1}) - (I_Nu_\lambda)^{k+1} \right) \ \mathrm{d}(x,y) \]
after an integration by parts.
For the third integral, another integration by parts, followed by the addition of a harmless zero term, yields
\begin{align*} (III) &= \int_{\widetilde{X}} I_N(u_\lambda^{k+1}) \partial_x \left( I_N(u_\lambda^{k+1}) - (I_Nu_\lambda)^{k+1} \right) \ \mathrm{d}(x,y) \\ & \ \ \ - \frac{1}{2}  \int_{\widetilde{X}} \partial_x \left(I_N(u_\lambda^{k+1})\right)^2 \ \mathrm{d}(x,y),
\end{align*}
and the last term vanishes yet again by the fundamental theorem of calculus.
Inserting these partial results into \eqref{fundamentalcalc} finally completes the proof.
\end{proof}

\begin{rem}
The proof of Lemma \ref{modifiedenergylemma} implicitly relies on the assumption that $\mathrm{(CP}_{k,\widetilde{X}} \mathrm{)}$ is locally well-posed for each $k \in \N$ at least for some $s = s(k) < 1$. On $\R^2$, this follows immediately from the well-posedness results in the non-rescaled setting. On $\R \times \T_\lambda$, however, a more careful analysis is required, since the underlying function spaces change under the rescaling in the periodic component. In the cases $k=1$ and $k=2$, we will explicitly work out the local well-posedness in the rescaled framework below (see Propositions \ref{bilinestimateZKlambda} and \ref{trilinestimatemZKlambda}, as well as Remark \ref{RemarkonLWPformZK}); for $k \geq 3$, we limit ourselves to the claim that the local well-posedness results of Theorem 1.2 in \cite{JNK2025} and Theorem 1.1 in \cite{JNK2026} can be transferred to the rescaled setting.
\end{rem}

With this, we now have all the necessary general tools at our disposal and can proceed to the equation-specific arguments.

\subsection{Improved GWP for ZK on $\mathbb{R} \times \mathbb{T}$}

We start by presenting an alternative proof of the bilinear $X_{s,b,1}$-estimate by Cao-Labora (see the proof of Theorem 1.3 in \cite{CaoLabora2025}), which yields local well-posedness in $H^s_1$ for all $s > \frac{3}{4}$. This approach allows us to directly reformulate it within the rescaled framework.

\begin{prop} \label{bilinestimateZKlambda} Let $\lambda \in \intco{1,\infty}$. Then, for every $s > \frac{3}{4}$, there exists an $\epsilon > 0$ such that
\begin{equation} \label{lwpZKinH^slambda}
\norm{\partial_x(u_1 u_2)}_{X_{s,-\frac{1}{2}+2\epsilon,\lambda}} \les \norm{u_1}_{X_{s,\frac{1}{2}+\epsilon,\lambda}} \norm{u_2}_{X_{s,\frac{1}{2}+\epsilon,\lambda}}
\end{equation}
holds for all $u_1,u_2 \in X_{s,\frac{1}{2}+\epsilon,\lambda}$.
\end{prop}

\begin{proof}
By the definition of the $X_{s,b,\lambda}$-norms, we may, without loss of generality, assume that $\widehat{\! u_i}^\lambda \geq 0$, and, by symmetry, restrict our considerations to the frequency configuration $\abs{(\xi_1,q_1)} \geq \abs{(\xi_2,q_2)}$. Here, $(\xi_i,q_i)$ denotes the wave number associated with the factor $u_i$, and we use $\ast$ to indicate the underlying convolution constraint $(\tau_0,\xi_0,q_0) = (\tau_1+\tau_2,\xi_1+\xi_2,q_1+q_2)$.
Furthermore, by duality and an application of Parseval's identity, it suffices to verify
\begin{align*} I_f &\coloneqq \frac{1}{\lambda^2} \int_{\R^2} \sum_{q_0 \in \Z/\lambda} \int_{\R^2} \sum_{\substack{q_1 \in \Z/\lambda \\ \ast}} \abs{\xi_0} \langle (\xi_0,q_0) \rangle^{s} \widehat{\! f}^\lambda(\tau_0,\xi_0,q_0) \\ & \ \ \ \cdot  \prod_{i=1}^{2} \widehat{\! u_i}^\lambda(\tau_i,\xi_i,q_i) \ \mathrm{d}(\tau_1,\xi_1) \mathrm{d}(\tau_0,\xi_0) \\ & \les \norm{f}_{X_{0,\frac{1}{2}-2\epsilon,\lambda}} \prod_{i=1}^{2} \norm{u_i}_{X_{s,\frac{1}{2}+\epsilon,\lambda}}
\end{align*}
for every $f \in X_{0,\frac{1}{2}-2\epsilon,\lambda}$ with $\widehat{\! f}^\lambda \geq 0$ and $\norm{f}_{X_{0, \frac{1}{2}-2\epsilon,\lambda}} \leq 1$, which we will now do by means of a case-by-case analysis. \\ \\
(i) \underline{$\abs{(\xi_1,q_1)} \gg \abs{(\xi_2,q_2)}$}: \\
In this case, we have
\[ \abs{\xi_0}  \langle (\xi_0,q_0) \rangle^s \les \abs{\abs{(\xi_1,q_1)}^2-\abs{(\xi_2,q_2)}^2}^\frac{1}{2} \langle (\xi_1,q_1) \rangle^s,  \]
so that, after undoing Plancherel and applying Hölder's inequality, it follows that
\[ I_f \les \norm{MP_\lambda(J^s u_1, u_2)}_{L_{txy,\lambda}^2} \norm{f}_{L_{txy,\lambda}^2}. \]
By using \eqref{MPlambda} for the first factor, we further obtain
\[ \les \norm{u_1}_{X_{s,\frac{1}{2}+\epsilon,\lambda}} \norm{u_2}_{X_{\frac{1}{2}+,\frac{1}{2}+\epsilon,\lambda}} \norm{f}_{X_{0,\frac{1}{2}-2\epsilon,\lambda}} \les \norm{f}_{X_{0,\frac{1}{2}-2\epsilon,\lambda}} \prod_{i=1}^{2} \norm{u_i}_{X_{s,\frac{1}{2}+\epsilon,\lambda}}, \]
and this works for every $s > \frac{1}{2}$, provided that $0 < \epsilon \ll 1$ is chosen small enough to ensure $ 0 \leq \frac{1}{2}-2\epsilon$. \\
(ii) \underline{$\abs{(\xi_1,q_1)} \sim \abs{(\xi_2,q_2)}$}: \\
(ii.1) \underline{$\abs{\abs{(\xi_i,q_i)}^2-\abs{(\xi_j,q_j)}^2} \gtrsim \abs{\xi_0}^{\frac{3}{2}}$ for some tuple $(i,j) \in \set{(0,1),(0,2),(1,2)}$}: \\
In the case $(i,j) = (1,2)$, one argues exactly as in case (i), taking into account assumption (ii) and the resulting pointwise estimate
\[ \abs{\xi_0} \langle (\xi_0,q_0) \rangle^s \lesssim \abs{\abs{(\xi_1,q_1)}^2 - \abs{(\xi_2,q_2)}^2}^{\frac{1}{2}} \langle (\xi_1,q_1) \rangle^s \langle (\xi_2,q_2) \rangle^{\frac{1}{4}}, \] which leads to $s > \frac{3}{4}$ instead of $s > \frac{1}{2}$. For symmetry reasons, it remains only to consider the case $(i,j) = (0,1)$. By assumption (ii), it follows that
\[ \abs{\xi_0} \langle (\xi_0,q_0) \rangle^s \les \abs{\abs{(\xi_0,q_0)}^2-\abs{(\xi_1,q_1)}^2}^\frac{1}{2} \langle (\xi_1,q_1) \rangle^{s-} \langle (\xi_0,q_0) \rangle^{-\frac{1}{2}-} \langle (\xi_2,q_2) \rangle^{\frac{3}{4}+},  \]
and an application of Plancherel's theorem, followed by Hölder's inequality, then gives
\[ I_f \les \norm{MP_\lambda(J^{s-}\widetilde{u}_1,J^{-\frac{1}{2}-}f)}_{L_{txy,\lambda}^2} \norm{J^{\frac{3}{4}+}u_2}_{L_{txy,\lambda}^2}, \]
with $\widetilde{u}_1(t,x,y) \coloneqq u_1(-t,-x,-y)$. For the first factor, we can now employ \eqref{MPlambdadual}, which yields
\[ \les \norm{u_1}_{X_{s,\frac{1}{2}+\epsilon,\lambda}} \norm{f}_{X_{0,\frac{1}{2}-2\epsilon,\lambda}} \norm{u_2}_{X_{\frac{3}{4}+,\frac{1}{2}+\epsilon,\lambda}} \les \norm{f}_{X_{0,\frac{1}{2}-2\epsilon,\lambda}} \prod_{i=1}^{2} \norm{u_i}_{X_{s,\frac{1}{2}+\epsilon,\lambda}}, \]
and the last inequality works for every $s>\frac{3}{4}$, provided that $0< \epsilon \ll 1$ is chosen sufficiently small. \\
(ii.2) \underline{$\abs{\abs{(\xi_i,q_i)}^2-\abs{(\xi_j,q_j)}^2} \ll \abs{\xi_0}^{\frac{3}{2}}$ for all tuples $(i,j) \in \set{(0,1),(0,2),(1,2)}$}: \\
(ii.2.1) \underline{$\abs{\xi_0} \les 1$}:
Due to the active assumptions (ii) and (ii.2.1), we can infer the pointwise bound
\[ \abs{\xi_0} \langle (\xi_0,q_0) \rangle^{s} \les \langle (\xi_1,q_1) \rangle^{\frac{s}{2}} \langle (\xi_2,q_2) \rangle^{\frac{s}{2}}, \]
which, after an application of Plancherel's theorem and Hölder's inequality, leads us to
\[ I_f \les \norm{J^{\frac{s}{2}}u_1J^{\frac{s}{2}}u_2}_{L_{txy,\lambda}^2} \norm{f}_{L_{txy,\lambda}^2} \leq \norm{J^{\frac{s}{2}}u_1}_{L_{txy,\lambda}^4} \norm{J^{\frac{s}{2}}u_2}_{L_{txy,\lambda}^4} \norm{f}_{L_{txy,\lambda}^2}. \]
For the first two factors, we can now invoke \eqref{L^4lambda}, which yields
\[ \les \norm{u_1}_{X_{\frac{s}{2}+,\frac{1}{2}+\epsilon,\lambda}} \norm{u_2}_{X_{\frac{s}{2}+,\frac{1}{2}+\epsilon,\lambda}} \norm{f}_{X_{0,\frac{1}{2}-2\epsilon,\lambda}} \les \norm{f}_{X_{0,\frac{1}{2}-2\epsilon,\lambda}} \prod_{i=1}^{2} \norm{u_i}_{X_{s,\frac{1}{2}+\epsilon,\lambda}}, \]
and this works for all $s > 0$, provided that $0 \leq \frac{1}{2}-2\epsilon$. \\
(ii.2.2) \underline{$\abs{\xi_0} \gg 1$}: \\
(ii.2.2.1) \underline{$\abs{\xi_{1}} \ll \abs{\xi_2}$ or $\abs{\xi_{2}} \ll \abs{\xi_1}$ }: \\
We may, without loss of generality, assume that $\abs{\xi_1} \ll \abs{\xi_2}$. Then we have $1 \ll \abs{\xi_0} = \abs{\xi_1+\xi_2}$ and $1 \ll \abs{\xi_0} \les \abs{\xi_0+(-\xi_1)}$, so that, if atleast one of the two inequalities
\begin{equation} \label{bilinworksZK1}
\begin{aligned}
\abs{3\xi_0^2-q_0^2} &\gtrsim \abs{\xi_0}, \\
\text{or} \quad \abs{3\xi_2^2-q_2^2} &\gtrsim \abs{\xi_2}
\end{aligned}
\end{equation}
holds, the bilinear refinement of the linear $L^4_{txy,\lambda}$-estimate \eqref{L^4lambda} can be used in the form \eqref{Bilinreflambda} or \eqref{Bilinreflambdadual} to obtain the desired estimate for all $s > \frac{3}{4}$. We illustrate this by working out the first case; the second case follows analogously to the first, taking into account the pointwise estimate
\[ \abs{\xi_0} \langle (\xi_0,q_0) \rangle^s \lesssim \abs{\xi_0+(-\xi_1)}^{\frac{1}{4}} \langle (\xi_0,q_0) \rangle^{0-} \langle (-\xi_1,-q_1) \rangle^{\frac{3}{4}+} \langle (\xi_2,q_2) \rangle^s \]
and then applying \eqref{Bilinreflambdadual} in place of \eqref{Bilinreflambda}.
Thus, let us assume that $\abs{3\xi_0^2-q_0^2} \gtrsim \abs{\xi_0}$. Then we have
\[ \abs{\xi_0} \langle (\xi_0,q_0) \rangle^s \les \abs{\xi_1+\xi_2}^\frac{1}{4} \langle (\xi_1,q_1) \rangle^{\frac{s}{2}+\frac{3}{8}} \langle (\xi_2,q_2) \rangle^{\frac{s}{2}+\frac{3}{8}}, \]
and undoing Plancherel, followed by Hölder's inequality, yields
\[ I_f \les \norm{I_x^\frac{1}{4}P_\lambda^{1}(J^{\frac{s}{2}+\frac{3}{8}}u_1 J^{\frac{s}{2}+\frac{3}{8}}u_2)}_{L_{txy,\lambda}^2} \norm{f}_{L_{txy,\lambda}^2}. \]
For the first factor, we can now invoke \eqref{Bilinreflambda} in the case $\alpha = 1$ to further obtain
\[ \les \norm{u_1}_{X_{\frac{s}{2}+\frac{3}{8}+,\frac{1}{2}+\epsilon,\lambda}} \norm{u_2}_{X_{\frac{s}{2}+\frac{3}{8}+,\frac{1}{2}+\epsilon,\lambda}} \norm{f}_{X_{0,\frac{1}{2}-2\epsilon,\lambda}} \les \norm{f}_{X_{0,\frac{1}{2}-2\epsilon,\lambda}} \prod_{i=1}^{2} \norm{u_i}_{X_{s,\frac{1}{2}+\epsilon,\lambda}}, \]
with the last estimate being valid for all $s > \frac{3}{4}$ and $0 \leq \frac{1}{2} - 2\epsilon$.
If, on the other hand, "$\ll$" held in both relations listed in \eqref{bilinworksZK1}, then using (ii.2) together with $\abs{\xi_0-\xi_2} = \abs{\xi_1} \ll \abs{\xi_2}$ and $\abs{q_1} \sim \abs{(\xi_1,q_1)}$ would give $\abs{2q_2+q_1} \ll \abs{\xi_0}$, and hence
\[ \abs{\xi_0}^{\frac{3}{2}} \gg \abs{\abs{(\xi_1,q_1)}^2 - \abs{(\xi_2,q_2)}^2} \sim \abs{\sqrt{2}q_2+q_1} \abs{\sqrt{2}q_2-q_1} \sim \abs{(\xi_1,q_1)}^2 \]
- a contradiction. Thus, the discussion of this subcase is concluded. \\
(ii.2.2.2) \underline{$\abs{\xi_1} \sim \abs{\xi_2}$}: \\
First, we note that the assumption $\abs{\xi_2} \ll \abs{(\xi_1,q_1)}$, together with (ii.2) would lead to
\[ \abs{q_1} \sim \abs{q_2} \sim \abs{q_0} \sim \abs{(\xi_1,q_1)} \ \text{with} \ \abs{\abs{q_i}-\abs{q_j}} \ll \abs{(\xi_1,q_1)} \ \forall i,j \in \set{0,1,2}, \]
which is incompatible with the convolution constraint $q_0 = q_1+q_2$. Hence, we must have $\abs{\xi_1} \sim \abs{\xi_2} \sim \abs{(\xi_1,q_1)}$. We now consider the resonance function $R_{\text{ZK}}$: A straightforward computation shows that
\[ R_{\text{ZK}} = -6\xi_0\xi_1\xi_2 + \sum_{i=1}^{2} \xi_i (\abs{(\xi_0,q_0)}^2-\abs{(\xi_i,q_i)}^2), \]
and taking (ii.2), (ii.2.2), and (ii.2.2.2) into account, we obtain
\[ \abs{R_{\text{ZK}}} \gtrsim \abs{\xi_0} \abs{\xi_1} \abs{\xi_2}. \]
Furthermore, we have
\[ \abs{R_{\text{ZK}}} \les \max_{i=0}^{2} \langle \tau_i - \phi(\xi_i,q_i) \rangle \eqqcolon \max_{i=0}^{3} \langle \sigma_i \rangle, \]
so overall
\[ \abs{\xi_0} \abs{\xi_1} \abs{\xi_2} \les \max_{i=0}^{2} \langle \sigma_i \rangle,  \]
and we are now led to distinguish two cases:
In the case $\max_{i=0}^{2} \langle \sigma_i \rangle = \langle \sigma_0 \rangle$, we obtain
\[ \abs{\xi_0} \langle (\xi_0,q_0) \rangle^s \langle \sigma_0 \rangle^{-\frac{1}{2}+2\epsilon} \les \langle (\xi_1,q_1) \rangle^{\frac{s}{2}-\frac{1}{4}+} \langle (\xi_2,q_2) \rangle^{\frac{s}{2}-\frac{1}{4}+}, \]
and undoing Plancherel, followed by Hölder's inequality gives
\begin{align*} I_f &\les \norm{J^{\frac{s}{2}-\frac{1}{4}+}u_1J^{\frac{s}{2}-\frac{1}{4}+}u_2}_{L_{txy,\lambda}^2} \norm{f}_{X_{0,\frac{1}{2}-2\epsilon}} \\ & \leq \norm{J^{\frac{s}{2}-\frac{1}{4}+}u_1}_{L_{txy,\lambda}^4} \norm{J^{\frac{s}{2}-\frac{1}{4}+}u_2}_{L_{txy,\lambda}^4} \norm{f}_{X_{0,\frac{1}{2}-2\epsilon,\lambda}}. \end{align*}
An application of \eqref{L^4lambda} to each of the first two factors then yields
\begin{align*} ... & \les \norm{u_1}_{X_{\frac{s}{2}-\frac{1}{4}+,\frac{1}{2}+\epsilon,\lambda}} \norm{u_2}_{X_{\frac{s}{2}-\frac{1}{4}+,\frac{1}{2}+\epsilon,\lambda}} \norm{f}_{X_{0,\frac{1}{2}-2\epsilon,\lambda}} \\ & \les \norm{f}_{X_{0,\frac{1}{2}-2\epsilon,\lambda}} \prod_{i=1}^2 \norm{u_i}_{X_{s,\frac{1}{2}+\epsilon,\lambda}},
\end{align*}
and the last step works as long as
\[ \frac{s}{2}-\frac{1}{4}+ \leq s \Leftrightarrow -\frac{1}{2} + \leq s \]
is satisfied. By choosing $0 < \epsilon \ll 1$ sufficiently small, any $s > -\frac{1}{2}$ can be realized within this estimate, and this concludes the discussion of this first case.
On the other hand, if, without loss of generality, $\max_{i=0}^{2} \langle \sigma_i \rangle = \langle \sigma_1 \rangle$, we obtain pointwise
\[ \abs{\xi_0} \langle (\xi_0,q_0) \rangle^s \langle \sigma_1 \rangle^{-\frac{1}{2}-\epsilon} \les \langle (\xi_1,q_1) \rangle^s \langle (\xi_2,q_2) \rangle^{-\frac{1}{2}+} \langle (\xi_0,q_0) \rangle^{0-}, \]
and applying Plancherel's theorem and Hölder's inequality leads to
\begin{align*} I_f &\les \norm{J^{0-}fJ^{-\frac{1}{2}+}\widetilde{u}_2}_{L_{txy,\lambda}^2} \norm{u_1}_{X_{s,\frac{1}{2}+\epsilon,\lambda}} \\ & \leq \norm{J^{0-}f}_{L_{txy,\lambda}^{4-}} \norm{J^{-\frac{1}{2}+}\widetilde{u}_2}_{L_{txy,\lambda}^{4+}} \norm{u_1}_{X_{s,\frac{1}{2}+\epsilon,\lambda}}.  \end{align*}
For the first factor, we now use \eqref{L^4-lambda}, and for the second, \eqref{L^4+lambda}, which together give
\begin{align*}
... & \les \norm{f}_{X_{0,\frac{1}{2}-2\epsilon,\lambda}} \norm{u_2}_{X_{-\frac{1}{2}+,\frac{1}{2}+\epsilon,\lambda}} \norm{u_1}_{X_{s,\frac{1}{2}+\epsilon, \lambda}} \\ & \les \norm{f}_{X_{0,\frac{1}{2}-2\epsilon,\lambda}} \prod_{i=1}^2 \norm{u_i}_{X_{s,\frac{1}{2}+\epsilon, \lambda}},
\end{align*}
and the last step again works for any $s > -\frac{1}{2}$, provided $\epsilon > 0$ is chosen sufficiently small. Thus, all cases have now been adressed, and the proof is complete.
\end{proof}

\begin{rem}
The proof shows that estimate \eqref{lwpZKinH^slambda} remains valid, up to changes in the constants, for every $0< \epsilon' \leq \epsilon$.
\end{rem}

Using Proposition \ref{bilinestimateZKlambda}, we can now establish the following variant of local well-posedness for $\mathrm{(CP}_{1,\R \times \T_\lambda} \mathrm{)}$ in $I_N^{-1}H^1_\lambda$.

\begin{prop} \label{variantLWPZK}
Let $\lambda \in \intco{1,\infty}$. Then $\mathrm{(CP}_{1,\R \times \T_\lambda} \mathrm{)}$ is locally well-posed in \\ $I_N^{-1}H^1_\lambda \coloneqq \set{f \in H^s_\lambda \ | \ \norm{I_Nf}_{H^1_\lambda} < \infty}$ for every $s \in \intoo{\frac{3}{4},1}$. Moreover, for the lifespan $\delta > 0$ of the solution $u_\lambda$ with initial data $u_{0,\lambda}$, one has

\[ \delta \gtrsim \norm{I_Nu_{0,\lambda}}_{H^1_\lambda}^{-\gamma} \]

for some $\gamma > 0$, and the estimate

\begin{equation} \label{contdepZKlambda} \norm{I_Nu_\lambda}_{X_{1,\frac{1}{2}+,\lambda}^\delta} \lesssim \norm{I_Nu_{0,\lambda}}_{H^1_\lambda} \end{equation}

holds.

\end{prop}

\begin{proof}

It is well known that, for any given $s \in \intoo{\frac{3}{4},1}$, it suffices to find $0< \epsilon \ll 1$ such that the bilinear estimate

\begin{equation} \label{variantbilinZK}
\norm{\partial_xI_N(u_1 u_2)}_{X_{1,-\frac{1}{2}+2\epsilon,\lambda}} \lesssim \norm{I_Nu_1}_{X_{1,\frac{1}{2}+\epsilon,\lambda}} \norm{I_Nu_2}_{X_{1,\frac{1}{2}+\epsilon,\lambda}}
\end{equation}

holds for all functions $u_1,u_2$ with $I_Nu_1,I_Nu_2 \in X_{1,\frac{1}{2}+\epsilon,\lambda}$. To this end, we write

\begin{align*} \norm{\partial_xI_N(u_1 u_2)}_{X_{1,-\frac{1}{2}+2\epsilon,\lambda}} &\leq \norm{\partial_x(I_Nu_1 I_Nu_2)}_{X_{1,-\frac{1}{2}+2\epsilon,\lambda}} \\ & \ \ \ + \norm{\partial_x(I_N(u_1 u_2) - I_Nu_1 I_Nu_2)}_{X_{1,-\frac{1}{2}+2\epsilon,\lambda}}, \end{align*}

and for the first term, Proposition \ref{bilinestimateZKlambda} immediately gives

\[ \norm{\partial_x(I_Nu_1 I_Nu_2)}_{X_{1,-\frac{1}{2}+2\epsilon,\lambda}} \les \norm{I_Nu_1}_{X_{1,\frac{1}{2}+\epsilon,\lambda}} \norm{I_Nu_2}_{X_{1,\frac{1}{2}+\epsilon,\lambda}}  \]

for all $s<1$, provided $\epsilon > 0$ is chosen sufficiently small. We are therefore left to consider the second term. By duality and an application of Parseval's identity, we obtain

\begin{align*} &\norm{\partial_x(I_N(u_1 u_2) - I_Nu_1 I_Nu_2)}_{X_{1,-\frac{1}{2}+2\epsilon,\lambda}} \sim \\ & \sup_{\norm{f}_{X_{0,\frac{1}{2}-2\epsilon,\lambda}} \leq 1} \frac{1}{\lambda^2} \left| \int_{\R^2} \sum_{q_0 \in \Z/\lambda} \int_{\R^2} \sum_{\substack{q_1 \in \Z/\lambda \\ \ast}} \xi_0 \langle (\xi_0,q_0) \rangle \frac{m(\xi_0,q_0) - m(\xi_1,q_1)m(\xi_2,q_2)}{m(\xi_1,q_1)m(\xi_2,q_2)} \right. \\ & \cdot \left. \hat{\! f}^\lambda(\tau_0,\xi_0,q_0) \prod_{i=1}^{2} \widehat{\! I_Nu_i}^\lambda (\tau_i,\xi_i,q_i) \ \mathrm{d}(\tau_1,\xi_1) \mathrm{d}(\tau_0, \xi_0) \right| \eqqcolon \sup_{\norm{f}_{X_{0,\frac{1}{2}-2\epsilon,\lambda}} \leq 1} I_f \end{align*}

with the convolution constraint $(\tau_0,\xi_0,q_0) = (\tau_1+\tau_2,\xi_1+\xi_2,q_1+q_2)$, and we now proceed to a detailed case-by-case analysis: Due to the definition of the $X_{s,b,\lambda}$-spaces, we may assume without loss of generality that $\widehat{\! f}^\lambda, \widehat{\! I_Nu_1}^\lambda, \widehat{\! I_Nu_1}^\lambda \geq 0$, and, for symmetry reasons, we can further restrict our attention to the frequency configuration $\abs{(\xi_1,q_1)} \geq \abs{(\xi_2,q_2)}$. \\ \\
(i) \underline{$\abs{(\xi_1,q_1)} \ll N$}: \\
Since
\[ m(\xi_0,q_0) - m(\xi_1,q_1)m(\xi_2,q_2) = 1 - 1 \cdot 1 = 0 \]
in this case, it follows that
\[ I_f = 0, \]
and thus nothing more needs to be done. \\
(ii) \underline{$\abs{(\xi_1,q_1)} \gtrsim N$}: \\
(ii.1) \underline{$\abs{(\xi_1,q_1)} \gg \abs{(\xi_2,q_2)}$}: \\
(ii.1.1) \underline{$\abs{(\xi_2,q_2)} \ll N$}: \\
In this situation, the mean value theorem yields
\begin{align*} Diff_{\text{ZK}_1} &\coloneqq \frac{\abs{m(\xi_0,q_0) - m(\xi_1,q_1)m(\xi_2,q_2)}}{m(\xi_1,q_1)m(\xi_2,q_2)} = \frac{\abs{m(\xi_0,q_0) - m(\xi_1,q_1)}}{m(\xi_1,q_1)} \\  &\les \abs{(\xi_1,q_1)}^{-1} \abs{(\xi_2,q_2)}, \end{align*}
which allows us to pass pointwise to
\[ \abs{\xi_0} \langle (\xi_0,q_0) \rangle Diff_{\text{ZK}_1} \les \abs{(\xi_1,q_1)}^{-\frac{1}{2}+} \abs{\abs{(\xi_1,q_1)}^2-\abs{(\xi_2,q_2)}^2}^\frac{1}{2} \langle (\xi_1,q_1) \rangle \langle (\xi_2,q_2) \rangle^{\frac{1}{2}-}.  \]
If we also take $\abs{(\xi_1,q_1)}^{-\frac{1}{2}+} \les N^{-\frac{1}{2}+}$ into account, then after undoing Plancherel and applying Hölder's inequality we obtain
\[ I_f \les N^{-\frac{1}{2}+} \norm{MP_\lambda(J^1I_Nu_1,J^{\frac{1}{2}-}I_Nu_2)}_{L_{txy,\lambda}^2} \norm{f}_{L_{txy,\lambda}^2}, \]
and a subsequent application of \eqref{MPlambda} to the first norm finally leads to
\begin{align*} ... &\les N^{-\frac{1}{2}+} \norm{I_Nu_1}_{X_{1,\frac{1}{2}+\epsilon,\lambda}} \norm{I_Nu_2}_{X_{1,\frac{1}{2}+\epsilon,\lambda}} \norm{f}_{L_{txy,\lambda}^2} \\ & \les N^{-\frac{1}{2}+} \norm{I_Nu_1}_{X_{1,\frac{1}{2}+\epsilon,\lambda}} \norm{I_Nu_2}_{X_{1,\frac{1}{2}+\epsilon,\lambda}} \norm{f}_{X_{0,\frac{1}{2}-2\epsilon,\lambda}}. \end{align*}
In the last step, we made use of $\norm{f}_{L_{txy,\lambda}^2} \les \norm{f}_{X_{0,\frac{1}{2}-2\epsilon,\lambda}}$, which is justified as long as $ 0\leq \frac{1}{2}-2\epsilon$ is satisfied. \\
(ii.1.2) \underline{$\abs{(\xi_2,q_2)} \gtrsim N$}: \\
Due to the active assumption (ii.1), we still have $\abs{(\xi_0,q_0)} \sim \abs{(\xi_1,q_1)}$, which allows us to conclude
\[ Diff_{\text{ZK}_1} \les \frac{1}{m(\xi_2,q_2)} \sim N^{s-1} \abs{(\xi_2,q_2)}^{1-s}. \]
Therefore we obtain the pointwise bound
\begin{align*} \abs{\xi_0} \langle (\xi_0,q_0) \rangle Diff_{\text{ZK}_1} &\les N^{s-1} \abs{(\xi_2,q_2)}^{\frac{1}{2}-s+} \abs{\abs{(\xi_1,q_1)}^2 - \abs{(\xi_2,q_2)}^2}^\frac{1}{2} \langle (\xi_1,q_1) \rangle \\ & \ \ \ \cdot \langle (\xi_2,q_2) \rangle^{\frac{1}{2}-}, \end{align*}
and in the case $s > \frac{1}{2}$, taking (ii.1.2) into account, we have $N^{s-1} \abs{(\xi_2,q_2)}^{\frac{1}{2}-s+} \les N^{-\frac{1}{2}+}$. Thus, as far as the pointwise estimates are concerned (with the additional constraint $s > \frac{1}{2}$), we are now precisely in the situation of case (ii.1.1). Consequently, we may again conclude in complete analogy that
\[ I_f \les N^{-\frac{1}{2}+} \norm{I_Nu_1}_{X_{1,\frac{1}{2}+\epsilon,\lambda}} \norm{I_Nu_2}_{X_{1,\frac{1}{2}+\epsilon,\lambda}} \norm{f}_{X_{0,\frac{1}{2}-2\epsilon,\lambda}},   \]
as long as $0 \leq \frac{1}{2}-2\epsilon$ is satisfied. \\
(ii.2) \underline{$\abs{(\xi_1,q_1)} \sim \abs{(\xi_2,q_2)}$}: \\
(ii.2.1) \underline{$\abs{(\xi_1,q_1)} \gg \abs{(\xi_0,q_0)}$}: \\
In this case,
\[ Diff_{\text{ZK}_1} \les \frac{1}{m(\xi_1,q_1)m(\xi_2,q_2)} \sim N^{2s-2} \abs{(\xi_2,q_2)}^{2-2s}, \]
so we can pass pointwise to
\begin{align*} \abs{\xi_0} \langle (\xi_0,q_0) \rangle Diff_{\text{ZK}_1} &\les N^{2s-2} \abs{(\xi_2,q_2)}^{\frac{3}{2}-2s+} \abs{\abs{(\xi_1,q_1)}^2-\abs{(\xi_0,q_0)}^2}^\frac{1}{2} \langle (\xi_0,q_0) \rangle^{0-} \\ & \ \ \ \cdot \langle (\xi_1,q_1) \rangle^{\frac{1}{2}-} \langle (\xi_2,q_2) \rangle.  \end{align*}
Now, for every $s > \frac{3}{4}$, taking into account (ii) and (ii.2), we have \\ $N^{2s-2}\abs{(\xi_2,q_2)}^{\frac{3}{2}-2s+} \les N^{-\frac{1}{2}+}$, and after undoing Plancherel, followed by Hölder's inequality, we obtain
\[ I_f \les N^{-\frac{1}{2}+} \norm{MP_\lambda(J^{0-}f,J^{\frac{1}{2}-}I_N\widetilde{u}_1)}_{L_{txy,\lambda}^2} \norm{J^1I_Nu_2}_{L_{txy,\lambda}^2}. \]
A subsequent application of variant \eqref{MPlambdadual} of the bilinear smoothing estimate \eqref{MPlambda} to the first norm ultimately gives
\[ ... \les N^{-\frac{1}{2}+} \norm{f}_{X_{0,\frac{1}{2}-2\epsilon,\lambda}} \norm{I_Nu_1}_{X_{1,\frac{1}{2}+\epsilon,\lambda}} \norm{I_Nu_2}_{X_{1,\frac{1}{2}+\epsilon,\lambda}}, \]
and all of this works provided that $0<\epsilon \ll 1$ is chosen sufficiently small. \\
(ii.2.2) \underline{$\abs{(\xi_1,q_1)} \sim \abs{(\xi_0,q_0)}$}: \\
We are now in the situation $\abs{(\xi_1,q_1)} \sim \abs{(\xi_2,q_2)} \sim \abs{(\xi_0,q_0)}$, so that we can conclude
\[ Diff_{\text{ZK}_1} \les \frac{1}{m(\xi_2,q_2)} \sim N^{s-1} \abs{(\xi_2,q_2)}^{1-s}, \]
which in turn allows us to infer the bound
\[ \abs{\xi_0} \langle (\xi_0,q_0) \rangle Diff_{\text{ZK}_1} \les N^{s-1} \abs{(\xi_2,q_2)}^{\frac{3}{4}-s+} \abs{\xi_0} \langle (\xi_0,q_0) \rangle^{\frac{3}{4}+} \langle (\xi_1,q_1) \rangle^{\frac{1}{4}-} \langle (\xi_2,q_2) \rangle^{\frac{1}{4}-}. \]
For every $s > \frac{3}{4}$, we can further pass to $N^{s-1} \abs{(\xi_2,q_2)}^{\frac{3}{4}-s+} \les N^{-\frac{1}{4}+}$, and after applying Parseval's identity, we obtain
\[ \sup_{\norm{f}_{X_{0,\frac{1}{2}-2\epsilon,\lambda}} \leq 1} I_f \les N^{-\frac{1}{4}+} \norm{I_x^1(J^{\frac{1}{4}-}I_Nu_1 J^{\frac{1}{4}-}I_Nu_2)}_{X_{\frac{3}{4}+,-\frac{1}{2}+2\epsilon,\lambda}}, \]
which, using Proposition \ref{bilinestimateZKlambda} for sufficiently small $\epsilon > 0$, can be further estimated, leading to
\[ ... \les N^{-\frac{1}{4}+} \norm{I_Nu_1}_{X_{1,\frac{1}{2}+\epsilon,\lambda}} \norm{I_Nu_2}_{X_{1,\frac{1}{2}+\epsilon,\lambda}}. \] \\ \\
Collecting all partial results from this case-by-case analysis, we arrive at
\begin{equation} \label{nonlinearsmoothingZKrough}
\norm{\partial_x(I_N(u_1u_2) - I_Nu_1I_Nu_2)}_{X_{1,-\frac{1}{2}+2\epsilon,\lambda}} \les N^{-\frac{1}{4}+} \norm{I_Nu_1}_{X_{1,\frac{1}{2}+\epsilon,\lambda}} \norm{I_Nu_1}_{X_{1,\frac{1}{2}+\epsilon,\lambda}},
\end{equation}
and thus the proof of \eqref{variantbilinZK} is complete.
\end{proof}

\begin{rem}
A careful examination of the proof of \eqref{variantbilinZK} shows that we in fact have \eqref{nonlinearsmoothingZKrough} in the following form at our disposal: Let $\lambda \in \intco{1,\infty}$ and $s \in \intoo{\frac{3}{4},1}$ be arbitrary. Then, for every $\widetilde{\epsilon} > 0$, there exists $\epsilon > 0$ such that
\begin{equation} \label{nonlinearsmoothingZKprecise}
\norm{\partial_x(I_N(u_1u_2) - I_Nu_1I_Nu_2)}_{X_{1,-\frac{1}{2}+2\epsilon',\lambda}} \les_{s, \widetilde{\epsilon},\epsilon'} N^{-\frac{1}{4}+\widetilde{\epsilon}} \prod_{i=1}^{2} \norm{I_Nu_i}_{X_{1,\frac{1}{2}+\epsilon',\lambda}}
\end{equation}
holds for all $0<\epsilon' \leq \epsilon$ and $u_1,u_2$ with $I_Nu_1, I_Nu_2 \in X_{1,\frac{1}{2}+\epsilon',\lambda}$.
\end{rem}

The estimate \eqref{nonlinearsmoothingZKprecise}, obtained as a byproduct in the proof of Proposition \ref{variantLWPZK}, can now be used to further bound the cubic term that appears in the difference of the modified energy.

\begin{lemma}

Let $\lambda \in \intco{1,\infty}$ be arbitrary. Given $\delta \in \R$, $s \in \intoo{\frac{3}{4},1}$, $\widetilde{\epsilon} > 0$, and $\epsilon' > 0$, one has

\begin{equation} \label{cubictermZK} \begin{aligned} &\left| \int_{0}^{\delta} \int_{\R \times \T_\lambda} \Delta I_Nu_1 \partial_x(I_N(u_2 u_3) - I_Nu_2 I_Nu_3) \ \mathrm{d}(x,y) \mathrm{d}t \right| \les_{s, \widetilde{\epsilon}, \epsilon ', \delta} \\ & N^{-\frac{1}{4}+\widetilde{\epsilon}} \prod_{i=1}^{3} \norm{I_Nu_i}_{X_{1,\frac{1}{2}+\epsilon ',\lambda}^{\abs{\delta}}} \end{aligned} \end{equation}

for all $u_1,u_2,u_3$ with $I_Nu_i \in X_{1,\frac{1}{2}+\epsilon ',\lambda}^{\abs{\delta}}$ ($i \in \set{1,2,3}$).

\end{lemma}

\begin{proof}

Since we are working in a Hilbert space framework, for each $i \in \set{1,2,3}$ there exists an extension $v_i \in X_{1,\frac{1}{2}+\epsilon',\lambda}$ of $I_Nu_i$, such that $v_i|_{\intcc{-\abs{\delta},\abs{\delta}} \times \R \times \T_\lambda} = I_Nu_i$ and $\norm{v_i}_{X_{1,\frac{1}{2}+\epsilon',\lambda}} = \norm{I_Nu_i}_{X_{1,\frac{1}{2}+\epsilon',\lambda}^{\abs{\delta}}}$. Substituting $u_i$ by $I_N^{-1}v_i$ on the left-hand side of the estimate in question, the desired bound follows from \eqref{nonlinearsmoothingZKprecise}, upon applying the Cauchy-Schwarz inequality and using the fact that 
\[ 1 = \langle (\xi,q) \rangle^{-1} \langle (\xi,q) \rangle \langle \tau - \phi(\xi,q) \rangle^{\frac{1}{2}-\epsilon'} \langle \tau- \phi(\xi,q) \rangle^{-\frac{1}{2}+\epsilon'} \]
and that
\[ M: H_t^{\frac{1}{2}-\epsilon''} \longrightarrow H_t^{\frac{1}{2}-\epsilon''}, \quad f \mapsto M(f) \coloneqq \chi_{\intcc{\min({0,\delta}),\max({0,\delta})}} f \]
is a continuous linear operator for every $\epsilon'' \in \intoo{0,1}$.
\end{proof}

The final remaining task is to control the quartic contribution to the difference of the modified energy in an analogous manner. This estimate is carried out in the subsequent lemma, and it completes the preparations required for the iteration argument leading to global well-posedness.

\begin{lemma} Let $\lambda \in \intco{1,\infty}$, $\delta \in \R$, $s \in \intoo{\frac{1}{2},1}$, $\widetilde{\epsilon} > 0$, and $\epsilon' > 0$ be given. Then the estimate

\begin{equation} \label{quartictermZK}
\begin{aligned}
&\left| \int_{0}^{\delta} \int_{\R \times \T_\lambda} I_N(u_1 u_2) \partial_x(I_N(u_3 u_4) - I_Nu_3 I_Nu_4) \ \mathrm{d}(x,y) \mathrm{d}t \right| \les_{s, \widetilde{\epsilon}, \epsilon', \delta} \\ & N^{-1+\widetilde{\epsilon}} \prod_{i=1}^{4} \norm{I_Nu_i}_{X_{1,\frac{1}{2}+\epsilon', \lambda}^{\abs{\delta}}}
\end{aligned}
\end{equation}
holds for all $u_1,u_2,u_3,u_4$ with $I_Nu_i \in X_{1,\frac{1}{2}+\epsilon',\lambda}^{\abs{\delta}}$  ($i \in \set{1,2,3,4}$).
\end{lemma}

\begin{proof}
It suffices to verify the estimate
\begin{equation} \label{quartictermZKnodelta}
\begin{aligned}
&\left| \int_\R \int_{\R \times \T_\lambda} I_N(v_1 v_2) \partial_x(I_N(v_3 v_4) - I_Nv_3 I_Nv_4) \ \mathrm{d}(x,y) \mathrm{d}t \right| \les_{s, \widetilde{\epsilon}, \epsilon'} \\ & N^{-1+\widetilde{\epsilon}} \prod_{i=1}^{4} \norm{I_Nv_i}_{X_{1,\frac{1}{2}+\epsilon', \lambda}}
\end{aligned}
\end{equation}

for arbitrary functions $v_1,v_2,v_3,v_4$ with $I_Nv_i \in X_{1,\frac{1}{2}+\epsilon',\lambda}$ and $\widehat{\! v_i}^\lambda \geq 0$ for all $i \in \set{1,2,3,4}$.
The desired estimate \eqref{quartictermZK} then follows by proceeding, as in the proof of \eqref{cubictermZK}, to an appropriate extension for each $I_Nu_i$, incorporating the argument presented in \cite{CollianderSchroedi2001} to pass from the time interval $\intcc{\min(0,\delta),\max(0,\delta)}$ to all of $\R$, and finally applying \eqref{quartictermZKnodelta}.
We now turn to the proof of \eqref{quartictermZKnodelta}. An application of Parseval's identity transforms the left-hand side of \eqref{quartictermZKnodelta} into the expression

\begin{align*} & \frac{1}{\lambda^3} \left| \int_{\R^2} \sum_{q_0 \in \Z/\lambda} \int_{\R^4} \sum_{\substack{q_1,q_3 \in \Z/\lambda \\ \ast}} \xi_0 \frac{m(\xi_0,q_0)(m(\xi_0,q_0) - m(\xi_3,q_3)m(\xi_4,q_4))}{m(\xi_1,q_1)m(\xi_2,q_2)m(\xi_3,q_3)m(\xi_4,q_4)} \right. \\ & \cdot \left. \prod_{i=1}^{4} \widehat{\! I_Nv_i}^\lambda(\tau_i,\xi_i,q_i) \ \mathrm{d}(\tau_1,\tau_3,\xi_1,\xi_3) \mathrm{d}(\tau_0, \xi_0) \right| \eqqcolon Int \end{align*}

subject to the convolution constraint $(\tau_0,\xi_0,q_0) = (\tau_1+\tau_2,\xi_1+\xi_2,q_1+q_2) = (\tau_3+\tau_4,\xi_3+\xi_4,q_3+q_4)$, and it is this integral that we must estimate in a case-by-case analysis: Exploiting symmetry, we may assume without loss of generality that $\abs{(\xi_1,q_1)} \geq \abs{(\xi_2,q_2)}$ and $\abs{(\xi_3,q_3)} \geq \abs{(\xi_4,q_4)}$, and we proceed with case \\ \\
(i) \underline{$\abs{(\xi_3,q_3)} \ll N$}: \\
In this situation, we have
\[ m(\xi_0,q_0) - m(\xi_3,q_3)m(\xi_4,q_4) = 1- 1 \cdot 1 = 0, \]
so
\[ Int = 0, \]
and no further steps are required. \\
(ii) \underline{$\abs{(\xi_3,q_3)} \gtrsim N$}: \\
(ii.1) \underline{$\abs{(\xi_1,q_1)} \ll N$}: \\
In this case, the convolution constraint gives $\abs{\xi_0} \les \abs{(\xi_1,q_1)} \land \abs{(\xi_3,q_3)} = \abs{(\xi_1,q_1)}$, $\abs{(\xi_3,q_3)} \sim \abs{(\xi_4,q_4)}$, and thus
\[ Diff_{\text{ZK}_2} \coloneqq \frac{m(\xi_0,q_0)(m(\xi_0,q_0)-m(\xi_3,q_3)m(\xi_4,q_4))}{m(\xi_1,q_1)m(\xi_2,q_2)m(\xi_3,q_3)m(\xi_4,q_4)} \les N^{2s-2} \abs{(\xi_3,q_3)}^{2-2s}, \]
allowing us to pass pointwise to
\[ \abs{\xi_0} Diff_{\text{ZK}_2} \les N^{2s-2} \abs{(\xi_3,q_3)}^{-2s+} \langle (\xi_2,q_2) \rangle^{-1+} \prod_{i=1}^{4} \langle (\xi_i,q_i) \rangle^{1-}. \]
Taking into account that $N^{2s-2} \abs{(\xi_3,q_3)}^{-2s+} \les N^{-2+\widetilde{\epsilon}}$ holds for $s>0$, we then obtain, after undoing Plancherel and subsequently applying Hölder's inequality, that
\begin{align*}
Int &\les N^{-2+\widetilde{\epsilon}} \norm{J^{1-}I_Nv_1J^{1-}I_Nv_2}_{L_{txy,\lambda}^2} \norm{J^{1-}I_Nv_3J^{1-}I_Nv_4}_{L_{txy,\lambda}^2} \\ &\leq N^{-2+\widetilde{\eps}} \prod_{i=1}^{4} \norm{J^{1-}I_Nv_i}_{L_{txy,\lambda}^4} \\ & \les_{s,\widetilde{\epsilon},\epsilon'} N^{-2+\widetilde{\epsilon}} \prod_{i=1}^{4} \norm{I_Nv_i}_{X_{1,\frac{1}{2}+\epsilon',\lambda}},
\end{align*}
where in the last step we applied \eqref{L^4lambda} a total of four times.
This completes the discussion of this subcase. \\
(ii.2) \underline{$\abs{(\xi_1,q_1)} \gtrsim N$}: \\
In this case, we have
\begin{align*} Diff_{\text{ZK}_2} &\les \frac{1}{m(\xi_1,q_1)m(\xi_2,q_2)m(\xi_3,q_3)m(\xi_4,q_4)} \\ & \sim N^{2s-2} \abs{(\xi_1,q_1)}^{1-s} \abs{(\xi_3,q_3)}^{1-s} \frac{1}{m(\xi_2,q_2)m(\xi_4,q_4)}, \end{align*}
so that, taking into account $\abs{\xi_0} \les \abs{(\xi_1,q_1)} \land \abs{(\xi_3,q_3)}$, we may pass to
\begin{align*} \abs{\xi_0} Diff_{\text{ZK}_2} &\les N^{2s-2} \abs{(\xi_1,q_1)}^{\frac{1}{2}-s+} \abs{(\xi_3,q_3)}^{\frac{1}{2}-s+}  \frac{\langle (\xi_2,q_2) \rangle^{-1+} \langle (\xi_4,q_4) \rangle^{-1+}}{m(\xi_2,q_2)m(\xi_4,q_4)} \\ & \ \ \ \cdot \prod_{i=1}^{4} \langle (\xi_i,q_i) \rangle^{1-}.  \end{align*}
Since $s > \frac{1}{2}$ (and in particular $s > 0$), we can crudely estimate the contribution of
\[ \frac{\langle (\xi_2,q_2) \rangle^{-1+} \langle (\xi_4,q_4) \rangle^{-1+}}{m(\xi_2,q_2)m(\xi_4,q_4)} \]
from above - up to a constant - by $1$. Moreover, since $N^{2s-2} \abs{(\xi_1,q_1)}^{\frac{1}{2}-s+} \\ \abs{(\xi_3,q_3)}^{\frac{1}{2}-s+} \les N^{-1+\widetilde{\epsilon}}$ for $s > \frac{1}{2}$, we obtain, after undoing Plancherel and applying Hölder's inequality, that
\begin{align*}
Int &\les N^{-1+\widetilde{\epsilon}} \norm{J^{1-}I_Nv_1J^{1-}I_Nv_2}_{L_{txy,\lambda}^2} \norm{J^{1-}I_Nv_3J^{1-}I_Nv_4}_{L_{txy,\lambda}^2} \\ &\leq \prod_{i=1}^{4} \norm{J^{1-}I_Nv_i}_{L_{txy,\lambda}^4} \\ & \les_{s,\widetilde{\epsilon},\epsilon'} N^{-1+\widetilde{\epsilon}} \prod_{i=1}^{4} \norm{I_Nv_i}_{X_{1,\frac{1}{2}+\epsilon',\lambda}},
\end{align*}
with the last step being justified by a fourfold application of \eqref{L^4lambda}.
That concludes the treatment of this final subcase.
\end{proof}

With this, we now have all the requisite ingredients in place for the

\begin{proof}[Proof of Theorem \ref{GWPZK}]
Without loss of generality, we may assume $s \in \intoo{\frac{11}{13},1}$, i.e. $s = \frac{11}{13}+ \widetilde{\epsilon}$ for some $\widetilde{\epsilon} \in \intoo{0,\frac{2}{13}}$. 
Let $u_0 \in H^s(\R \times \T)$ and $T > 0$ be arbitrary. Our goal is to prove that the (local) solution $u$ of $\mathrm{(CP}_{1,\R \times \T} \mathrm{)}$ exists on the entire time interval $\intcc{-T,T}$. By the scaling symmetry considerations carried out in Section 4.1, this is equivalent to proving that the rescaled solution $u_\lambda$ of $\mathrm{(CP}_{1,\R \times \T_\lambda} \mathrm{)}$ exists on $\intcc{-\lambda^3T,\lambda^3T}$.
We now choose
\[ \lambda = \lambda_{\text{ZK}} = C_{\text{ZK}} N^{\frac{1-s}{1+s}} \]
as in Remark \ref{RemarklambdaZKandlambdamZK} (ii). Proposition \ref{variantLWPZK} then yields the existence of the associated solution $u_{\lambda_{\text{ZK}}}$ on the interval $\intcc{-\delta_0,\delta_0}$, where the initial lifespan $\delta_0 > 0$ satisfies
\[ \delta_0 \geq c \norm{I_Nu_{0,\lambda_{\text{ZK}}}}_{H^{1}_{\lambda_{\text{ZK}}}}^{-\gamma} \]
for some $\gamma > 0$. Combining \eqref{controloverH^1ZK}, Lemma \ref{modifiedenergylemma}, \eqref{cubictermZK}, \eqref{quartictermZK}, and the continuity estimate \eqref{contdepZKlambda} from Proposition \ref{variantLWPZK}, we obtain for this solution $u_{\lambda_{\text{ZK}}}$ the estimate
\begin{align*} \norm{I_Nu_{\lambda_{\text{ZK}}}(\delta_0)}_{H^1_{\lambda_{\text{ZK}}}}^2 &\leq C_0\left(\norm{u_{0,\lambda_{\text{ZK}}}}_{L_{\lambda_\text{ZK}}^2}, E_{\pm}[I_N u_{\lambda_{\text{ZK}}}](0)\right) \\ & \ \ \ + \widetilde{C} \left( N^{-\frac{1}{4}+\frac{\widetilde{\epsilon}}{2}} \norm{I_Nu_{0,\lambda_{\text{ZK}}}}_{H^1_{\lambda_{\text{ZK}}}}^3 + N^{-1+\frac{\widetilde{\epsilon}}{2}}  \norm{I_Nu_{0,\lambda_{\text{ZK}}}}_{H^1_{\lambda_{\text{ZK}}}}^4 \right),  \end{align*}
and we now wish to ensure that the right-hand side of this inequality remains small throughout as many forthcoming iterations of the fixed-point argument as needed. To this end, we choose $0 < \epsilon_0 \ll 1$ such that
\[ c \epsilon_0^{-\frac{\gamma}{2}} \geq 1, \]
and by enlarging the constant $C_{\text{ZK}}$ in the definition of $\lambda_{\text{ZK}}$, Lemma \ref{lemmarescaledH^1lambda} guarantees that we may enforce
\[ C_0\left(\norm{u_{0,\lambda_{\text{ZK}}}}_{L_{\lambda_\text{ZK}}^2}, E_{\pm}[I_N u_{\lambda_{\text{ZK}}}](0)\right) \leq \frac{\epsilon_0}{2}. \]
(see also Remark \ref{RemarkonmodifiedH^1normforZKandmZK} (ii)).
Next, choosing $N \gg 1$ sufficiently large so that
\begin{align*} &\widetilde{C} \left( N^{-\frac{1}{4}+\frac{\widetilde{\epsilon}}{2}} \left( N^{\frac{1}{4}-\widetilde{\epsilon}} + \norm{I_Nu_{0,\lambda_{\text{ZK}}}}_{H^1_{\lambda_{\text{ZK}}}}^3 \right) + N^{-1+\frac{\widetilde{\epsilon}}{2}} \left( N^{\frac{1}{4}-\widetilde{\epsilon}} + \norm{I_Nu_{0,\lambda_{\text{ZK}}}}_{H^1_{\lambda_{\text{ZK}}}}^4 \right) \right) \\ & \leq \frac{\epsilon_0}{2}, \end{align*}
the fixed-point argument from Proposition \ref{variantLWPZK} can be applied $N^{\frac{1}{4}-\widetilde{\epsilon}}$ times, with each iteration yielding an extension of the lifespan by at least $1$ (both forward and backward in time). Consequently, the solution $u_{\lambda_{\text{ZK}}}$ exists at least on the interval $\intcc{-N^{\frac{1}{4}-\widetilde{\epsilon}}, N^{\frac{1}{4}-\widetilde{\epsilon}}}$.
We now require
\[ N^{\frac{1}{4}-\widetilde{\epsilon}} > \lambda_{\text{ZK}}^3 T = C_{\text{ZK}}^3 N^{\frac{3(1-s)}{1+s}} T, \]
which holds for sufficiently large $N$, provided that
\[ \frac{1}{4}-\widetilde{\epsilon} - \frac{3(1-s)}{1+s} > 0 \Leftrightarrow s > \frac{11+4 \widetilde{\epsilon}}{13-4 \widetilde{\epsilon}}. \]
Since our fixed $s = \frac{11}{13}+\widetilde{\epsilon}$ satisfies this condition, the solution $u_{\lambda_{\text{ZK}}}$ persists on the time interval $\intcc{-\lambda_{\text{ZK}}^3T,\lambda_{\text{ZK}}^3T}$. Scaling back to the original problem then completes the argument. Moreover, the choice
\[ N \sim (1+T)^{\frac{4(1+s)}{13s-11}+}, \]
together with $\lambda_{\text{ZK}} = C_{\text{ZK}} N^{\frac{1-s}{1+s}}$ and the estimate
\[ \norm{u(\lambda_{\text{ZK}}^{-3}t)}_{H^s_1} \les \lambda_{\text{ZK}}^{1+s} \norm{I_Nu_{\lambda_{\text{ZK}}}(t)}_{H^1_{\lambda_{\text{ZK}}}}, \]
finally yields
\[ \norm{u(\lambda_{\text{ZK}}^{-3}t)}_{H^s_1} \les (1+T)^{\frac{4(1-s^2)}{13s-11}+} \]
for all $t \in \intcc{-\lambda_{\text{ZK}}^3T,\lambda_{\text{ZK}}^3T}$. Taking the supremum over all such $t$ then gives precisely the polynomial bound \eqref{polygrowthZKbelowH^1}. This completes the proof of Theorem \ref{GWPZK}.
\end{proof}

\begin{rem}
The proof reveals that the decay rate $N^{-\frac{1}{4}+}$ of the cubic term in the modified energy is the decisive factor governing the quality of the resulting threshold for global well-posedness. It is worth noting that the extracted exponent $-\frac{1}{4}+$ exactly reflects the full amount of nonlinear smoothing encoded in the bilinear estimate \eqref{lwpZKinH^slambda}, which is optimal up to the endpoint.
\end{rem}

\subsection{(Improved) GWP for mZK on $\mathbb{R}^2$ and $\mathbb{R} \times \mathbb{T}$}

Following the structure of the previous section, we start by transferring the most recent local well-posedness result for mZK on $\R \times \T$ to the rescaled setting.

\begin{prop} \label{trilinestimatemZKlambda} Let $\lambda \in \intco{1,\infty}$ and $\widetilde{\epsilon} > 0$ be given. Then, for every $s > \frac{11}{24}$, there exists an $\epsilon > 0$ such that
\begin{equation} \label{LWPmZKinH^slambda11/24}
\norm{\partial_x(u_1u_2u_3)}_{X_{s,-\frac{1}{2}+2\epsilon,\lambda}} \les_{\widetilde{\epsilon}} \lambda^{\widetilde{\epsilon}} \prod_{i=1}^{3}  \norm{u_i}_{X_{s,\frac{1}{2}+\epsilon,\lambda}}
\end{equation}
holds for all time-localized functions $u_1,u_2,u_3 \in X_{s,\frac{1}{2}+\epsilon,\lambda}$.
\end{prop}

\begin{proof}
The proof of \eqref{LWPmZKinH^slambda11/24} proceeds in complete analogy with the proof of Proposition 4.3 in \cite{JNK2025}, taking into account the fact that all Strichartz-type estimates used there carry over to the rescaled framework; these are essentially \eqref{MPlambda}, \eqref{L^4lambda}, \eqref{Bilinreflambda}, \eqref{AiryL^6+lambda}, \eqref{AiryL^6-lambda}, \eqref{optimizedL^6+}, \eqref{optimizedL^6-}, and \eqref{AiryL^4L^2}, as established in Section 3. The appearance of the factor $\lambda^{\widetilde{\epsilon}}$ on the right-hand side of the inequality reflects the fact that the Schrödinger $L^6$-estimate \eqref{1drescaledSchrödingerL^6} cannot be transferred to the rescaled spaces in a lossless manner.
\end{proof}

\begin{rem} \label{RemarkonLWPformZK} \begin{enumerate} \item[(i)] If one restricts to $s > \frac{1}{2}$, estimates \eqref{MPlambda} and \eqref{L^4lambda} are sufficient to obtain
\begin{equation} \label{LWPmZKinH^slambda1/2}
\norm{\partial_x(u_1u_2u_3)}_{X_{s,-\frac{1}{2}+2\epsilon,\lambda}} \les \prod_{i=1}^{3}  \norm{u_i}_{X_{s,\frac{1}{2}+\epsilon,\lambda}}
\end{equation} (see the proof of Proposition 4.1 in \cite{JNK2025}), and the assumption of time localization of the functions $u_i$, $i \in \set{1,2,3}$ can be dropped. Moreover, a careful inspection of the underlying local well-posedness proofs in \cite{JNK2025} shows that in both \eqref{LWPmZKinH^slambda11/24} and \eqref{LWPmZKinH^slambda1/2} the parameter $\epsilon$ may be replaced by any $0<\epsilon' \leq \epsilon$, upon adjusting the implicit constants in "$\les$".

\item[(ii)] For the non-periodic case, we note that the local well-posedness result for mZK due to Ribaud and Vento \cite{Ribaud2012} can be easily reformulated in an $X_{s,b}$-framework, which is tied to the following trilinear estimate: for every $s > \frac{1}{4}$, there exists an $\epsilon > 0$ such that
\begin{equation} \label{LWPmZKinH^s1/4}
\norm{\partial_x(u_1u_2u_3)}_{X_{s,-\frac{1}{2}+2\epsilon,\infty}} \les \prod_{i=1}^{3}  \norm{u_i}_{X_{s,\frac{1}{2}+\epsilon,\infty}}
\end{equation}
holds for all time-localized functions $u_1,u_2,u_3 \in X_{s,\frac{1}{2}+\epsilon,\infty}$. Moreover, as in the previous cases, $\epsilon$ may be replaced by any $0<\epsilon' \leq \epsilon$, with a corresponding adjustment of the implicit constants.  \end{enumerate}
\end{rem}

We are now in a position to prove the following local well-posedness result for 
$\mathrm{(CP}_{2, \R \times \T_\lambda} \mathrm{)}$ in $I_N^{-1}H^1_{\lambda}$.

\begin{prop} \label{variantLWPmZK}
Let $\lambda \in \intcc{1,N^{100}} \cup \set{\infty}$. Then $\mathrm{(CP}_{2,\R \times \T_\lambda} \mathrm{)}$ is locally well-posed in $I_N^{-1}H^1_{\lambda} = \set{f \in H^s_\lambda \ | \ \norm{I_Nf}_{H^1_\lambda} < \infty}$ for every $s \in \intoo{\frac{1}{2},1}$. Moreover, for the lifespan $\delta > 0$ of the solution $u_\lambda$ with initial data $u_{0,\lambda}$\footnote{We stress that the subscript $\infty$ in $u_\infty$ and $u_{0,\infty}$ does not refer to a rescaling of the type introduced in Section 4.1.}, one has
\[ \delta \gts \norm{I_Nu_{0,\lambda}}_{H^1_\lambda}^{-\gamma} \]
for some $\gamma > 0$, and the estimate
\begin{equation} \label{contdepmZKlambda}
\norm{I_Nu_\lambda}_{X_{1,\frac{1}{2}+,\lambda}^\delta} \les \norm{I_Nu_{0,\lambda}}_{H^1_\lambda}
\end{equation}
holds.
\end{prop}

\begin{proof}
As in the proof of Proposition \ref{variantLWPZK}, we restrict ourselves to proving the following statement: For every $s \in \intoo{\frac{1}{2},1}$, there exists an $\epsilon > 0$ such that the trilinear estimate
\begin{equation} \label{varianttrilinmZK}
\norm{\partial_xI_N(u_1u_2u_3)}_{X_{1,-\frac{1}{2}+2\epsilon,\lambda}} \les \prod_{i=1}^{3} \norm{I_Nu_i}_{X_{1,\frac{1}{2}+\epsilon,\lambda}}
\end{equation}
holds for all time-localized functions $u_1,u_2,u_3$ with $I_Nu_i \in X_{1,\frac{1}{2}+\epsilon,\lambda}$, $i \in \set{1,2,3}$. The claimed well-posedness result then follows from a fixed-point argument that is standard in the literature. \\
We begin by applying the triangle-inequality to split the left-hand side of \eqref{varianttrilinmZK} into two contributions, yielding
\begin{align*} \norm{\partial_xI_N(u_1u_2u_3)}_{X_{1,-\frac{1}{2}+2\epsilon,\lambda}} &\leq \norm{\partial_x(I_Nu_1I_Nu_2I_Nu_3)}_{X_{1,-\frac{1}{2}+2\epsilon,\lambda}} \\ & \ \ \ + \norm{\partial_x(I_N(u_1u_2u_3) - I_Nu_1I_Nu_2I_Nu_3)}_{X_{1,-\frac{1}{2}+2\epsilon,\lambda}}.  \end{align*}
For the first term, we can now immediately apply \eqref{LWPmZKinH^slambda1/2} in the case $\lambda \neq \infty$ or \eqref{LWPmZKinH^s1/4} in the case $\lambda = \infty$, which, for a sufficiently small choice of $0<\epsilon \ll 1$, leads to the desired estimate
\[ \norm{\partial_x(I_Nu_1I_Nu_2I_Nu_3)}_{X_{1,-\frac{1}{2}+2\epsilon,\lambda}} \les \prod_{i=1}^{3} \norm{I_Nu_i}_{X_{1,\frac{1}{2}+\epsilon,\lambda}}  \]
for every $s<1$.
In order to estimate the second term, we proceed by duality and invoke Parseval's identity to obtain
\begin{align*}
&\norm{\partial_x(I_N(u_1u_2u_3) - I_Nu_1I_Nu_2I_Nu_3)}_{X_{1,-\frac{1}{2}+2\epsilon,\lambda}} \sim \\ & \sup_{\norm{f}_{X_{0,\frac{1}{2}-2\epsilon,\lambda}}\leq 1} \frac{1}{\lambda^3} \left|\int_{\R^2} \sum_{q_0 \in \Z/\lambda} \int_{\R^4} \sum_{\substack{q_1,q_2 \in \Z/\lambda \\ \ast}} \widehat{\! f}^\lambda(\tau_0,\xi_0,q_0) \prod_{i=1}^{3} \widehat{\! I_Nu_i}^\lambda(\tau_i,\xi_i,q_i) \right. \\ & \cdot \left. \xi_0 \langle (\xi_0,q_0) \rangle \frac{m(\xi_0,q_0)-m(\xi_1,q_1)m(\xi_2,q_2)m(\xi_3,q_3)}{m(\xi_1,q_1)m(\xi_2,q_2)m(\xi_3,q_3)} \ \mathrm{d}(\tau_1,\tau_2,\xi_1,\xi_2) \mathrm{d}(\tau_0,\xi_0) \right|
\end{align*}
for $\lambda \in \intcc{1,N^{100}}$ and
\begin{align*}
&\norm{\partial_x(I_N(u_1u_2u_3) - I_Nu_1I_Nu_2I_Nu_3)}_{X_{1,-\frac{1}{2}+2\epsilon,\infty}} \sim \\ & \sup_{\norm{f}_{X_{0,\frac{1}{2}-2\epsilon,\infty}}\leq 1} \left|\int_{\R^3} \int_{\substack{\R^6 \\ \ast}} \widehat{\! f}^\lambda(\tau_0,\xi_0,\eta_0) \left( \prod_{i=1}^{3} \widehat{\! I_Nu_i}^\lambda(\tau_i,\xi_i,\eta_i) \right) \xi_0 \langle (\xi_0,\eta_0) \rangle \right. \\ & \cdot \left.  \frac{m(\xi_0,\eta_0)-m(\xi_1,\eta_1)m(\xi_2,\eta_2)m(\xi_3,\eta_3)}{m(\xi_1,\eta_1)m(\xi_2,\eta_2)m(\xi_3,\eta_3)} \ \mathrm{d}(\tau_1,\tau_2,\xi_1,\xi_2,\eta_1,\eta_2) \mathrm{d}(\tau_0,\xi_0,\eta_0) \right|
\end{align*}
for $\lambda = \infty$. Here, $\ast$ denotes the convolution constraints $(\tau_0,\xi_0,q_0) = (\tau_1+\tau_2+\tau_3,\xi_1+\xi_2+\xi_3,q_1+q_2+q_3)$ and $(\tau_0,\xi_0,\eta_0) = (\tau_1+\tau_2+\tau_3,\xi_1+\xi_2+\xi_3,\eta_1+\eta_2+\eta_3)$, and we denote the corresponding integral expressions by $I_f$, thereby suppressing the $\lambda$-dependence in the notation.
The further estimation of $I_f$ will now be the subject of a detailed case-by-case analysis, which we aim to carry out simultaneously for the rescaled semiperiodic and the non-periodic cases as far as it is possible. For simplicity, we write $q_i$ in place of $\eta_i$, $i \in \set{0,1,2,3}$, reverting to the correct variable names only in the case-specific arguments. Moreover, for everything that follows, we may, by the definition of the $X_{s,b,\lambda}$-spaces and by symmetry, assume that $\widehat{\! f}^\lambda,\widehat{\! u_i}^\lambda \geq 0$, $i \in \set{1,2,3}$, and restrict our considerations to the frequency configuration $\abs{(\xi_1,q_1)} \geq \abs{(\xi_2,q_2)} \geq \abs{(\xi_3,q_3)}$. This concludes the preliminary remarks, and we now begin with the discussion of case \\ \\
(i) \underline{$\abs{(\xi_1,q_1)} \ll N$}: \\
In this case, we have
\[ m(\xi_0,q_0) - m(\xi_1,q_1)m(\xi_2,q_2)m(\xi_3,q_3) = 1- 1\cdot 1 \cdot 1 = 0, \]
and therefore
\[ I_f = 0, \]
so that there is nothing more to prove. \\
(ii) \underline{$\abs{(\xi_1,q_1)} \gtrsim N$}: \\
(ii.1) \underline{$\abs{(\xi_1,q_1)} \gg \abs{(\xi_3,q_3)}$}: \\
(ii.1.1) \underline{$\abs{\abs{(\xi_0,q_0)}^2 - \abs{(\xi_2,q_2)}^2} \gtrsim \abs{\xi_0}^2$}: \\
(ii.1.1.1) \underline{$\abs{(\xi_2,q_2)} \ll N$}: \\
In this situation, the mean value theorem is applicable and yields
\begin{align*}
Diff_{\text{mZK}_1} &\coloneqq \frac{\abs{m(\xi_0,q_0) - m(\xi_1,q_1)m(\xi_2,q_2)m(\xi_3,q_3)}}{m(\xi_1,q_1)m(\xi_2,q_2)m(\xi_3,q_3)} = \frac{\abs{m(\xi_0,q_0) - m(\xi_1,q_1)}}{m(\xi_1,q_1)} \\ &\les \abs{(\xi_1,q_1)}^{-1} \abs{(\xi_2,q_2)}.
\end{align*}
We thus obtain the pointwise estimate
\begin{align*}
\abs{\xi_0} \langle (\xi_0,q_0) \rangle Diff_{\text{mZK}_1} &\les \abs{(\xi_1,q_1)}^{-\frac{3}{2}+} \langle (\xi_3,q_3) \rangle^{-\frac{1}{2}+} \\ & \ \ \ \cdot \abs{\abs{(\xi_1,q_1)}^2 - \abs{(\xi_3,q_3)}^2}^\frac{1}{2} \langle (\xi_1,q_1) \rangle \langle (\xi_3,q_3) \rangle^{\frac{1}{2}-} \\ & \ \ \ \cdot \abs{\abs{(\xi_0,q_0)}^2-\abs{(\xi_2,q_2)}^2}^\frac{1}{2} \langle (\xi_0,q_0) \rangle^{0-} \langle (\xi_2,q_2) \rangle^{\frac{1}{2}-},
\end{align*}
and since $\abs{(\xi_1,q_1)} \gtrsim N$, we may pass to $\abs{(\xi_1,q_1)}^{-\frac{3}{2}+} \langle (\xi_3,q_3) \rangle^{-\frac{1}{2}+} \les N^{-\frac{3}{2}+}$. Undoing Plancherel, followed by an application of Hölder's inequality, then gives
\[ I_f \les N^{-\frac{3}{2}+} \norm{MP_\lambda(J^1I_Nu_1,J^{\frac{1}{2}-}I_Nu_3)}_{L_{txy,\lambda}^2} \norm{MP_\lambda(J^{0-}f,J^{\frac{1}{2}-}I_N\widetilde{u}_2)}_{L_{txy,\lambda}^2},  \]
and for the first factor, we may now apply \eqref{MPlambda}, while the second factor can be estimated using \eqref{MPlambdadual}. This leads us to
\begin{equation*}
...\les N^{-\frac{3}{2}+} \norm{f}_{X_{0,\frac{1}{2}-2\epsilon,\lambda}} \prod_{i=1}^{3} \norm{I_Nu_i}_{X_{1,\frac{1}{2}+\epsilon,\lambda}},
\end{equation*}
which is valid as long as $0<\epsilon \ll 1 $ is chosen sufficiently small. \\
(ii.1.1.2) \underline{$\abs{(\xi_2,q_2)} \gtrsim N$}: \\
Under the given assumptions, we can estimate
\[ Diff_{\text{mZK}_1} \les N^{2s-2} \abs{(\xi_1,q_1)}^{1-s} \abs{(\xi_2,q_2)}^{1-s} \frac{1}{m(\xi_3,q_3)}, \]
which yields the pointwise bound
\begin{align*}
\abs{\xi_0} \langle (\xi_0,q_0) \rangle Diff_{\text{mZK}_1} &\les N^{2s-2} \abs{(\xi_1,q_1)}^{\frac{1}{4}-s+} \abs{(\xi_2,q_2)}^{\frac{1}{4}-s+} \frac{\langle (\xi_3,q_3) \rangle^{-\frac{1}{2}+}}{m(\xi_3,q_3)} \\ & \ \ \ \cdot \abs{\abs{(\xi_1,q_1)}^2 - \abs{(\xi_3,q_3)}^2}^\frac{1}{2} \langle (\xi_1,q_1) \rangle \langle (\xi_3,q_3) \rangle^{\frac{1}{2}-} \\ & \ \ \ \cdot \abs{\abs{(\xi_0,q_0)}^2-\abs{(\xi_2,q_2)}^2}^\frac{1}{2} \langle (\xi_0,q_0) \rangle^{0-} \langle (\xi_2,q_2) \rangle^{\frac{1}{2}-}.
\end{align*}
Now, taking into account that $N^{2s-2} \abs{(\xi_1,q_1)}^{\frac{1}{4}-s+} \abs{(\xi_2,q_2)}^{\frac{1}{4}-s+} \les N^{-\frac{3}{2}+}$ for $s > \frac{1}{4}$ and that $ \frac{\langle (\xi_3,q_3) \rangle^{-\frac{1}{2}+}}{m(\xi_3,q_3)} \les 1$ for $s > \frac{1}{2}$, we may proceed exactly as in case (ii.1.1.1) to obtain
\[ I_f \les N^{-\frac{3}{2}+} \norm{f}_{X_{0,\frac{1}{2}-2\epsilon,\lambda}} \prod_{i=1}^{3} \norm{I_Nu_i}_{X_{1,\frac{1}{2}+\epsilon,\lambda}},  \]
which holds provided $0<\epsilon \ll 1$ is chosen sufficiently small. \\
(ii.1.2) \underline{$\abs{\abs{(\xi_0,q_0)}^2 - \abs{(\xi_2,q_2)}^2} \ll \abs{\xi_0}^2$}: \\
(ii.1.2.1) \underline{$\abs{\abs{(\xi_i,q_i)}^2 - \abs{(\xi_j,q_j)}^2} \gtrsim \abs{\xi_0}^2$ for some tuple $(i,j) \in \set{(0,1),(1,2)}$}: \\
We begin with the case $(i,j) = (0,1)$. By the active assumptions (ii.1) and (ii.1.2), we have
\[ \abs{(\xi_0,q_0)} \sim \abs{(\xi_1,q_1)} \sim \abs{(\xi_2,q_2)} \gg \abs{(\xi_3,q_3)}, \]
which allows us to conclude
\[ Diff_{\text{mZK}_1} \les N^{s-1} \abs{(\xi_1,q_1)}^{1-s} \frac{1}{m(\xi_3,q_3)} \]
and hence
\begin{align*}
\abs{\xi_0} \langle (\xi_0,q_0) \rangle Diff_{\text{mZK}_1} &\les N^{s-1} \abs{(\xi_1,q_1)}^{-\frac{1}{2}-s+} \frac{\langle (\xi_3,q_3) \rangle^{-\frac{1}{2}+}}{m(\xi_3,q_3)} \\ & \ \ \ \cdot \abs{\abs{(\xi_2,q_2)}^2 - \abs{(\xi_3,q_3)}^2}^\frac{1}{2} \langle (\xi_2,q_2) \rangle \langle (\xi_3,q_3) \rangle^{\frac{1}{2}-} \\ & \ \ \ \cdot \abs{\abs{(\xi_0,q_0)}^2-\abs{(\xi_1,q_1)}^2}^\frac{1}{2} \langle (\xi_0,q_0) \rangle^{0-} \langle (\xi_1,q_1) \rangle^{\frac{1}{2}-}.
\end{align*}
We further note that $N^{s-1} \abs{(\xi_1,q_1)}^{-\frac{1}{2}-s+} \les N^{-\frac{3}{2}+}$ for $s > -\frac{1}{2}$ and that $\frac{\langle (\xi_3,q_3) \rangle^{-\frac{1}{2}+}}{m(\xi_3,q_3)} \\ \les 1$ for $s > \frac{1}{2}$, so that undoing Plancherel and applying Hölder's inequality leads us to
\[ I_f \les N^{-\frac{3}{2}+} \norm{MP_\lambda(J^1I_Nu_2,J^{\frac{1}{2}-}I_Nu_3)}_{L_{txy,\lambda}^2} \norm{MP_\lambda(J^{0-}f,J^{\frac{1}{2}-}I_N\widetilde{u}_2)}_{L_{txy,\lambda}^2}. \]
By applying \eqref{MPlambda} to the first and \eqref{MPlambdadual} to the second factor, we finally arrive at
\[...\les N^{-\frac{3}{2}+} \norm{f}_{X_{0,\frac{1}{2}-2\epsilon,\lambda}} \prod_{i=1}^{3} \norm{I_Nu_i}_{X_{1,\frac{1}{2}+\epsilon,\lambda}}, \]
which is valid as long as $0<\epsilon \ll 1$ is chosen small enough.
The case $(i,j) = (1,2)$ is treated in a completely analogous manner, with \eqref{MPlambda} and \eqref{MPlambdadual} applied to the pairings $(u_1,u_2)$ and $(f,\widetilde{u}_3)$, respectively. This concludes the discussion of this subcase. \\
(ii.1.2.2) \underline{$\abs{\abs{(\xi_i,q_i)}^2-\abs{(\xi_j,q_j)}^2} \ll \abs{\xi_0}^2$ for all tuples $(i,j) \in \set{(0,1),(1,2)}$}: \\
From this point on, the arguments become case-specific, and we begin with the treatment of the non-periodic case. \\
(ii.1.2.2.1)$_{\R^2}$ \underline{$\abs{3\xi_i^2-\eta_i^2} \gtrsim \abs{\xi_0}^2$ and $\abs{3\xi_j^2-\eta_j^2} \gtrsim \abs{\xi_0}^2$ for some tuple} \\ \underline{$(i,j) \in \set{(0,1),(0,2),(1,2)}$}: \\
Under the active assumptions (ii), (ii.1.2), and (ii.1.2.2), we have
\[ Diff_{\text{mZK}_1} \les N^{s-1} \abs{(\xi_1,\eta_1)}^{1-s} \frac{1}{m(\xi_3,\eta_3)}, \]
and we now first turn our attention to the case $(i,j) = (0,1)$. In this situation, we may pass pointwise to
\begin{align*}
\abs{\xi_0} \langle (\xi_0,\eta_0) \rangle Diff_{\text{mZK}_1} &\les N^{s-1} \abs{(\xi_1,\eta_1)}^{-\frac{1}{2}-s+} \frac{\langle (\xi_3,\eta_3) \rangle^{-\frac{1}{2}+}}{m(\xi_3,\eta_3)} \\ & \ \ \ \cdot \abs{\abs{(\xi_2,\eta_2)}^2-\abs{(\xi_3,\eta_3)}^2}^\frac{1}{2} \langle (\xi_2,\eta_2) \rangle \langle (\xi_3,\eta_3) \rangle^{\frac{1}{2}-} \\ & \ \ \ \cdot \abs{3\xi_0^2-\eta_0^2}^\frac{1}{8} \langle (\xi_0,\eta_0) \rangle^{0-} \abs{3\xi_1^2-\eta_1^2}^\frac{1}{8} \langle (\xi_1,\eta_1) \rangle^{1-}.
\end{align*}
Taking into account that for $s > -\frac{1}{2}$ we have $N^{s-1}  \abs{(\xi_1,\eta_1)}^{-\frac{1}{2}-s+} \les N^{-\frac{3}{2}+}$, and that the quantity $\frac{\langle (\xi_3,\eta_3) \rangle^{-\frac{1}{2}+}}{m(\xi_3,\eta_3)}$ is essentially bounded by $1$ for $s>\frac{1}{2}$, we obtain, after applying Parseval's identity and Hölder's inequality, that
\begin{align*}
I_f &\les N^{-\frac{3}{2}+} \norm{MP_\infty(J^1I_Nu_2,J^{\frac{1}{2}-}I_Nu_3)}_{L_{txy,\infty}^2} \norm{\abs{K(D)}^\frac{1}{8}J^{0-}f \abs{K(D)}^\frac{1}{8}J^{1-}I_N \widetilde{u}_1}_{L_{txy,\infty}^2} \\ & \leq N^{-\frac{3}{2}+}  \norm{MP_\infty(J^1I_Nu_2,J^{\frac{1}{2}-}I_Nu_3)}_{L_{txy,\infty}^2} \norm{\abs{K(D)}^\frac{1}{8}J^{0-}f}_{L_{txy,\infty}^{4-}} \\ & \ \ \ \cdot \norm{\abs{K(D)}^\frac{1}{8}J^{1-}I_N \widetilde{u}_1}_{L_{txy,\infty}^{4+}}.
\end{align*}
The first factor can now be estimated by \eqref{MPlambda}, while the second and third factors are handled by \eqref{L^4derivativegain-} and \eqref{L^4derivativegain+}, respectively.
Consequently, we obtain
\[... \les N^{-\frac{3}{2}+} \norm{f}_{X_{0,\frac{1}{2}-2\epsilon,\infty}} \prod_{i=1}^{3} \norm{I_Nu_i}_{X_{1,\frac{1}{2}+\epsilon,\infty}},  \]
and this argument works provided that $\epsilon > 0$ is chosen sufficiently small.
The case $(i,j) = (0,2)$ now follows completely analogously by symmetry. Likewise, in the case $(i,j) = (1,2)$, the estimate
\[ I_f \les  N^{-\frac{3}{2}+} \norm{f}_{X_{0,\frac{1}{2}-2\epsilon,\infty}} \prod_{i=1}^{3} \norm{I_Nu_i}_{X_{1,\frac{1}{2}+\epsilon,\infty}} \] can be obtained in a similar manner, by applying \eqref{MPlambdadual} to the pair $(f,\widetilde{u}_3)$ and estimating the factors corresponding to $u_1$ and $u_2$ using \eqref{L^4derivativegain}, respectively. \\
(ii.1.2.2.2)$_{\R^2}$ \underline{$\abs{3\xi_i^2-\eta_i^2} \ll \abs{\xi_0}^2$ or $\abs{3\xi_j^2-\eta_j^2} \ll \abs{\xi_0}^2$ for all tuples} \\ \underline{$(i,j) \in \set{(0,1),(0,2),(1,2)}$}: \\
In this situation, the following three cases must be distinguished:
\[ \abs{3\xi_1^2-\eta_1^2} \ll \abs{\xi_0}^2 \quad \text{and} \quad \abs{3\xi_2^2-\eta_2^2} \ll \abs{\xi_0}^2, \]
or
\[ \abs{3\xi_0^2-\eta_0^2} \ll \abs{\xi_0}^2 \quad \text{and} \quad \abs{3\xi_1^2-\eta_1^2} \ll \abs{\xi_0}^2, \]
or
\[ \abs{3\xi_0^2-\eta_0^2} \ll \abs{\xi_0}^2 \quad \text{and} \quad \abs{3\xi_2^2-\eta_2^2} \ll \abs{\xi_0}^2. \]
We first turn to the first case. Assume that $\abs{3\xi_1^2-\eta_1^2} \ll \abs{\xi_0}^2$ and $\abs{3\xi_2^2-\eta_2^2} \ll \abs{\xi_0}^2$. Then, taking into account assumptions (ii.1.2) and (ii.1.2.2), it follows that
\[ \abs{\xi_1} \sim \abs{\xi_2} \sim \abs{\eta_1} \sim \abs{\eta_2}, \quad \text{with} \quad \abs{\abs{\xi_1}-\abs{\xi_2}} \ll \abs{\xi_0} \quad \text{and} \quad \abs{\abs{\eta_1}-\abs{\eta_2}} \ll \abs{\xi_0}. \]
Depending on the signs of the wave numbers $\xi_1,\xi_2,\eta_1$, and $\eta_2$ (note that $\abs{(\xi_1,\eta_1)} \gg \abs{(\xi_3,\eta_3)}$ by assumption (ii.1)), this then leads to the following three possible scenarios
\[ \abs{\underbrace{3(\xi_1+\xi_2+\xi_3)^2+(\eta_1+\eta_2+\eta_3)^2}_{= 3\xi_0^2+\eta_0^2} - \begin{cases} 0 & \\ 4\eta_1^2 & \\ 8\eta_1^2 \end{cases}} \ll \abs{(\xi_1,\eta_1)}^2, \]
all of which contradict
\[\abs{3\xi_0^2+\eta_0^2 -2\eta_1^2} \leq \abs{\abs{(\xi_0,\eta_0)}^2-\abs{(\xi_2,\eta_2)}^2} + \abs{3\xi_1^2-\eta_1^2} \ll \abs{\xi_0}^2. \]
Hence, this case is ruled out.
On the other hand, suppose (without loss of generality) that $\abs{3\xi_0^2-\eta_0^2} \ll \abs{\xi_0}^2$ and $\abs{3\xi_1^2-\eta_1^2} \ll \abs{\xi_0}^2$. Then, again by (ii.1.2) and (ii.1.2.2), we obtain the relations
\[ \abs{\xi_0} \sim \abs{\xi_1} \sim \abs{\eta_0} \sim \abs{\eta_1}, \quad \text{with} \quad \abs{\abs{\xi_0}-\abs{\xi_1}} \ll \abs{\xi_0} \quad \text{and} \quad \abs{\abs{\eta_0}-\abs{\eta_1}} \ll \abs{\xi_0}. \]
Taking into account that $\abs{(\xi_1,\eta_1)} \gg \abs{(\xi_3,\eta_3)}$, the different sign configurations of the wave numbers $\xi_0, \xi_1, \eta_0$, and $\eta_1$ then give rise to the following three possible cases
\[ \abs{\underbrace{3(\xi_0-\xi_1-\xi_3)^2+(\eta_0-\eta_1-\eta_3)^2}_{= 3\xi_2^2+\eta_2^2} - \begin{cases} 0 & \\ 4\eta_1^2 & \\ 8\eta_1^2 \end{cases}} \ll \abs{(\xi_1,\eta_1)}^2. \]
As all of these cases are incompatible with
\[ \abs{3\xi_2^2+\eta_2^2-2\eta_1^2} \leq \abs{\abs{(\xi_2,\eta_2)}^2-\abs{(\xi_1,\eta_1)}^2} + \abs{3\xi_1^2-\eta_1^2} \ll \abs{\xi_0}^2, \]
this subcase is resolved as well. \\
(ii.2)$_{\R^2}$ \underline{$\abs{(\xi_1,\eta_1)} \sim \abs{(\xi_3,\eta_3)}$}: \\
In this final subcase, we can roughly estimate $Diff_{\text{mZK}_1}$ in the following way:
\[ Diff_{\text{mZK}_1} \les N^{3s-3} \abs{(\xi_1,\eta_1)}^{3-3s}. \]
Moreover, since $\abs{(\xi_1,\eta_1)} \sim \abs{(\xi_2,\eta_2)} \sim \abs{(\xi_3,\eta_3)}$, we may pass pointwise to
\[ \abs{\xi_0} \langle (\xi_0,\eta_0) \rangle Diff_{\text{mZK}_1} \les N^{3s-3} \abs{(\xi_1,\eta_1)}^{\frac{3}{2}-3s+} \abs{\xi_0} \langle (\xi_0,\eta_0) \rangle^{\frac{1}{4}+} \prod_{i=1}^{3} \langle (\xi_i,\eta_i) \rangle^{\frac{3}{4}-}, \] and since we also have $N^{3s-3} \abs{(\xi_1,\eta_1)}^{\frac{3}{2}-3s+} \les N^{-\frac{3}{2}+}$ ($s > \frac{1}{2}$), this allows us to infer the bound
\[ I_f \les N^{-\frac{3}{2}+} \norm{I_x^1(J^{\frac{3}{4}-}I_Nu_1J^{\frac{3}{4}-}I_Nu_2J^{\frac{3}{4}-}I_Nu_3)}_{X_{\frac{1}{4}+,-\frac{1}{2}+2\epsilon,\infty}}. \]
For sufficiently small $\epsilon > 0$, we may now apply the trilinear local well-posedness estimate \eqref{LWPmZKinH^s1/4} to further deduce
\[... \les N^{-\frac{3}{2}+} \prod_{i=1}^{3} \norm{I_Nu_i}_{X_{1,\frac{1}{2}+\epsilon,\infty}}, \]
which is exactly what we wanted to show. \\
Thus, in the non-periodic setting, the case-by-case analysis is complete and results in the commutator estimate
\begin{equation} \label{nonlinearsmoothingmZKinftyrough}
\norm{\partial_x(I_N(u_1u_2u_3) - I_Nu_1I_Nu_2I_Nu_3)}_{X_{1,-\frac{1}{2}+2\epsilon,\infty}} \les N^{-\frac{3}{2}+} \prod_{i=1}^{3} \norm{I_Nu_i}_{X_{1,\frac{1}{2}+\epsilon,\infty}}.
\end{equation}
In view of the preceding discussion, this proves \eqref{varianttrilinmZK} for $\lambda = \infty$.
We now address the remaining subcases arising in the semiperiodic setting: \\
(ii.1.2.2.1)$_{\R \times \T_\lambda}$ \underline{$\max_{i=1}^{2} \abs{\xi_i} \ll \abs{(\xi_1,q_1)}$}: \\
In this situation, taking assumption (ii.1) into account, we obtain
\[ \abs{\xi_0},\abs{\xi_1},\abs{\xi_2},\abs{\xi_3} \ll \abs{(\xi_1,q_1)}. \]
Together with (ii.1.2) and (ii.1.2.2), we can thus conclude that
\[ \abs{q_0} \sim \abs{q_1} \sim \abs{q_2} \sim \abs{(\xi_1,q_1)} \quad \text{with} \quad \abs{\abs{q_i} - \abs{q_j}} \ll \abs{(\xi_1,q_1)} \quad \text{for all} \quad i,j \in \set{0,1,2}, \]
which is incompatible with $\abs{q_3} \ll \abs{(\xi_1,q_1)}$ and the convolution constraint $q_0 = q_1+q_2+q_3$. Hence, this case is excluded. \\
(ii.1.2.2.2)$_{\R \times \T_\lambda}$ \underline{$\max_{i=1}^{2} \abs{\xi_i} \sim \abs{(\xi_1,q_1)}$}: \\
(ii.1.2.2.2.1)$_{\R \times \T_\lambda}$ \underline{$\min_{i=1}^{2} \abs{\xi_i} \ll \abs{(\xi_1,q_1)}$}: \\
We first fix, for the remainder of the argument, the bound
\[ Diff_{\text{mZK}_1} \les N^{s-1} \abs{(\xi_1,q_1)}^{1-s} \frac{1}{m(\xi_3,q_3)}. \]
Without loss of generality, let $\max_{i=1}^{2} \abs{\xi_i} = \abs{\xi_1}$. On the one hand, assumption (ii.1.2.2) then immediately implies
\begin{equation} \label{relations1}
\abs{q_2} \sim \abs{(\xi_1,q_1)} \quad \text{with} \quad \abs{\abs{q_2} - \abs{(\xi_1,q_1)}} \ll \abs{(\xi_1,q_1)},
\end{equation}
and on the other hand, we also have the pointwise estimates $\abs{\xi_0} \les \abs{\xi_1+\xi_2}$ and $\abs{\xi_0} \les \abs{\xi_1+\xi_3}$ at our disposal, which pave the way for an application of \eqref{Bilinreflambda}:
If one of the relations
\[ \abs{3(\xi_1+\xi_2)^2 - (q_1+q_2)^2} \gtrsim \abs{\xi_1+\xi_2} \]
or
\[ \abs{3(\xi_1+\xi_3)^2-(q_1+q_3)^2} \gtrsim \abs{\xi_1+\xi_3} \]
holds - say, the former - then we may pass pointwise to
\begin{align*}
\abs{\xi_0} \langle (\xi_0,q_0) \rangle Diff_{\text{mZK}_1} &\les N^{s-1} \abs{(\xi_1,q_1)}^{-\frac{1}{4}-s+} \frac{\langle (\xi_3,q_3) \rangle^{-\frac{1}{2}+}}{m(\xi_3,q_3)} \\ & \ \ \ \cdot \abs{\abs{(\xi_0,q_0)}^2 - \abs{(\xi_3,q_3)}^2}^\frac{1}{2} \langle (\xi_0,q_0) \rangle^{0-} \langle (\xi_3,q_3) \rangle^{\frac{1}{2}-} \\ & \ \ \ \cdot \abs{\xi_1+\xi_2}^\frac{1}{4} \langle (\xi_1,q_1) \rangle^{1-} \langle (\xi_2,q_2) \rangle^{1-}.
\end{align*}
Taking into account that $N^{s-1} \abs{(\xi_1,q_1)}^{-\frac{1}{4}-s+} \les N^{-\frac{5}{4}+}$ for $s > -\frac{1}{4}$ and \\ $\frac{\langle (\xi_3,q_3) \rangle^{-\frac{1}{2}+}}{m(\xi_3,q_3)} \les 1$ for $s > \frac{1}{2}$, it follows, after undoing Plancherel and applying Hölder's inequality, that
\[ I_f \les N^{-\frac{5}{4}+} \norm{MP_\lambda(J^{0-}f,J^{\frac{1}{2}-}I_N\widetilde{u}_3)}_{L_{txy,\lambda}^2} \norm{I_x^\frac{1}{4} P_\lambda^1(J^{1-}I_Nu_1 J^{1-}I_Nu_2)}_{L_{txy,\lambda}^2}, \]
and the first factor can be further estimated using \eqref{MPlambdadual}, while for the second we may appeal to \eqref{Bilinreflambda}. This then leads us to
\[ ... \les N^{-\frac{5}{4}+} \norm{f}_{X_{0,\frac{1}{2}-2\epsilon,\lambda}} \prod_{i=1}^{3} \norm{I_Nu_i}_{X_{1,\frac{1}{2}+\epsilon,\lambda}}, \]
and this argument closes provided $\epsilon > 0$ is chosen sufficiently small.
The case $\abs{3(\xi_1+\xi_3)^2-(q_1+q_3)^2} \gtrsim \abs{\xi_1+\xi_3}$ is treated analogously and yields the same conclusion, with \eqref{Bilinreflambdadual} applied to the pair $(f,\widetilde{u}_2)$ (note that $\abs{\xi_1+\xi_3} = \abs{\xi_0 + (-\xi_2)}$) and \eqref{MPlambda} applied to $(u_1,u_3)$.
It remains to show that the case
\begin{equation} \label{relations2}
\begin{aligned}
\abs{3(\xi_1+\xi_2)^2 - (q_1+q_2)^2} &\ll \abs{\xi_1+\xi_2} \\
\text{and} \quad  \abs{3(\xi_1+\xi_3)^2-(q_1+q_3)^2} &\ll \abs{\xi_1+\xi_3}. 
\end{aligned}
\end{equation}
cannot occur: From the relations in \eqref{relations1} and \eqref{relations2}, and additionally taking into account \\ $\abs{\xi_2},\abs{\xi_3},\abs{q_3} \ll \abs{(\xi_1,q_1)}$ and $\abs{\xi_1} \sim \abs{(\xi_1,q_1)}$, it can be concluded that $q_1$ and $q_2$ must have opposite signs, with
\[ \abs{\sqrt{3}\abs{\xi_1} - \abs{q_1}} \ll \abs{(\xi_1,q_1)} \quad \text{and} \quad \abs{2q_1 + q_2} \ll \abs{(\xi_1,q_1)}. \]
However, this would imply
\[ \abs{\abs{(\xi_1,q_1)}^2-\abs{(\xi_2,q_2)}^2} \sim \abs{q_1}^2 \sim \abs{(\xi_1,q_1)}^2, \]
which contradicts the active assumption (ii.1.2.2). Hence, this subcase is settled. \\
(ii.1.2.2.2.2)$_{\R \times \T_\lambda}$ \underline{$\min_{i=1}^{2} \abs{\xi_i} \sim \abs{(\xi_1,q_1)}$}: \\
(ii.1.2.2.2.2.1)$_{\R \times \T_\lambda}$ \underline{$\abs{\xi_3} \gtrsim \abs{\xi_0}$}: \\
In this case, we estimate the quantity $Diff_{\text{mZK}_1}$ by
\[ Diff_{\text{mZK}_1} \les N^{2s-2} \abs{(\xi_1,q_1)}^{2-2s}. \]
Since $\abs{\xi_0} \les \abs{\xi_3}$, we can shift almost half a derivative from the product onto the low-frequency factor $u_3$, which allows us to infer the pointwise bound
\begin{align*}
\abs{\xi_0} \langle (\xi_0,q_0) \rangle Diff_{\text{mZK}_1} &\les N^{2s-2} \abs{(\xi_1,q_1)}^{\frac{1}{2}-2s+} \langle (\xi_0,q_0) \rangle^{0-} \langle (\xi_2,q_2) \rangle^{1-} \\ & \ \ \ \cdot \abs{\abs{(\xi_1,q_1)}^2-\abs{(\xi_3,q_3)}^2}^\frac{1}{2} \langle (\xi_1,q_1) \rangle \langle (\xi_3,q_3) \rangle^{\frac{1}{2}-}.
\end{align*}
Now, since we have $N^{2s-2} \abs{(\xi_1,q_1)}^{\frac{1}{2}-2s+} \les N^{-\frac{3}{2}+}$ for $s>\frac{1}{4}$, an application of Parseval's identity, followed by Hölder's inequality yields
\begin{align*}
I_f &\les N^{-\frac{3}{2}+} \norm{MP_\lambda(J^1I_Nu_1,J^{\frac{1}{2}-}I_Nu_3)}_{L_{txy,\lambda}^2} \norm{J^{0-}fJ^{1-}I_N\widetilde{u}_2}_{L_{txy,\lambda}^2} \\ & \leq N^{-\frac{3}{2}+}\norm{MP_\lambda(J^1I_Nu_1,J^{\frac{1}{2}-}I_Nu_3)}_{L_{txy,\lambda}^2} \norm{J^{0-}f}_{L_{txy,\lambda}^{4-}} \norm{J^{1-}I_N\widetilde{u}_2}_{L_{txy,\lambda}^{4+}},
\end{align*}
and the first factor can now be estimated by means of \eqref{MPlambda}, while the second and third factors are controlled by \eqref{L^4-lambda} and \eqref{L^4+lambda}, respectively.
Altogether, this gives
\[ ... \les N^{-\frac{3}{2}+} \norm{f}_{X_{0,\frac{1}{2}-2\epsilon,\lambda}} \prod_{i=1}^{3} \norm{I_Nu_i}_{X_{1,\frac{1}{2}+\epsilon,\lambda}}, \]
as long as $0<\epsilon \ll 1$ is chosen sufficiently small. \\
(ii.1.2.2.2.2.2)$_{\R \times \T_\lambda}$ \underline{$\abs{\xi_3} \ll \abs{\xi_0}$}: \\
We again fix the estimate
\[ Diff_{\text{mZK}_1} \les N^{2s-2} \abs{(\xi_1,q_1)}^{2-2s} \]
and now make use of the resonance function $R_{\text{mZK}}$:
A straightforward computation shows that
\[ R_{\text{mZK}} = -6(\xi_1+\xi_2)(\xi_1+\xi_3)(\xi_2+\xi_3) + \sum_{i=1}^{3} \xi_i (\abs{(\xi_0,q_0)}^2 - \abs{(\xi_i,q_i)}^2), \]
and taking into account $\abs{\xi_1+\xi_2} \geq \abs{\xi_0} - \abs{\xi_3} \geq \frac{1}{2} \abs{\xi_0}$, $\abs{\xi_1+\xi_3},\abs{\xi_2+\xi_3} \sim \abs{(\xi_1,q_1)}$, the active assumptions (ii.1.2) and (ii.1.2.2), and the fact that $\abs{\xi_3} \ll \abs{\xi_0}$, the inverse triangle inequality implies
\[ \abs{R_{\text{mZK}}} \gtrsim \abs{\xi_0} \abs{(\xi_1,q_1)}^2, \]
and hence also
\[ \max_{i=0}^{3} \langle \sigma_i \rangle \coloneqq \max_{i=0}^{3} \langle \tau_i - \phi(\xi_i,q_i) \rangle \gtrsim \abs{R_{\text{mZK}}} \gtrsim \abs{\xi_0} \abs{(\xi_1,q_1)}^2.  \]
We are now led to distinguish between two cases.
In the case $\max_{i=0}^{3} \langle \sigma_i \rangle = \langle \sigma_3 \rangle$, we obtain the pointwise bound
\begin{align*}
\abs{\xi_0} \langle (\xi_0,q_0) \rangle Diff_{\text{mZK}_1} \langle \sigma_3 \rangle^{-\frac{1}{2}-\epsilon} &\les N^{2s-2} \abs{(\xi_1,q_1)}^{\frac{1}{2}-2s+} \langle (\xi_3,q_3) \rangle^{0-} \langle (\xi_0,q_0) \rangle^{0-} \\ & \ \ \ \cdot \langle (\xi_1,q_1) \rangle \langle (\xi_2,q_2) \rangle^{1-},
\end{align*}
so that an application of Plancherel's theorem and Hölder's inequality yields
\begin{align*}
I_f &\les N^{-\frac{3}{2}+} \norm{I_Nu_3}_{X_{1,\frac{1}{2}+\epsilon,\lambda}} \norm{J^1I_N\widetilde{u}_1J^{1-}I_N\widetilde{u}_2J^{0-}f}_{L_{t}^2L_{xy,\lambda}^1} \\ &\leq  N^{-\frac{3}{2}+} \norm{I_Nu_3}_{X_{1,\frac{1}{2}+\epsilon,\lambda}} \norm{J^1I_N\widetilde{u}_1}_{L_t^\infty L_{xy,\lambda}^2} \norm{J^{1-}I_N\widetilde{u}_2}_{L_{txy,\lambda}^{4+}} \norm{J^{0-}f}_{L_{txy,\lambda}^{4-}}
\end{align*}
($N^{2s-2} \abs{(\xi_1,q_1)}^{\frac{1}{2}-2s+} \les N^{-\frac{3}{2}+}$ for $s > \frac{1}{4}$). By employing the Sobolev embedding theorem in the $t$-variable for the second factor and applying estimates \eqref{L^4+lambda} and \eqref{L^4-lambda} to the third and fourth factors, respectively, we further obtain
\[ ... \les N^{-\frac{3}{2}+} \norm{f}_{X_{0,\frac{1}{2}-2\epsilon,\lambda}} \prod_{i=1}^{3} \norm{I_Nu_i}_{X_{1,\frac{1}{2}+\epsilon,\lambda}}, \]
which holds for sufficiently small $\epsilon > 0$. It thus remains to consider the case $\max_{i=0}^{3} \langle \sigma_i \rangle \neq \langle \sigma_3 \rangle$, for instance $\max_{i=0}^{3} \langle \sigma_i \rangle = \langle \sigma_0 \rangle$ (the other cases are treated analogously): In this situation, we can infer the pointwise bound
\begin{align*}
\abs{\xi_0} \langle (\xi_0,q_0) \rangle Diff_{\text{mZK}_1} \langle \sigma_0 \rangle^{-\frac{1}{2}+2\epsilon} &\les N^{2s-2} \abs{(\xi_1,q_1)}^{\frac{1}{2}-2s+} \langle (\xi_3,q_3) \rangle^{0-} \\ & \ \ \ \cdot \langle (\xi_1,q_1) \rangle^{1-} \langle (\xi_2,q_2) \rangle^{1-},
\end{align*}
which, taking into account that $N^{2s-2} \abs{(\xi_1,q_1)}^{\frac{1}{2}-2s+} \les N^{-\frac{3}{2}-}$ for $s > \frac{1}{4}$, leads us to
\begin{align*}
I_f &\les N^{-\frac{3}{2}+} \norm{f}_{X_{0,\frac{1}{2}-2\epsilon,\lambda}} \norm{J^{1-}I_Nu_1J^{1-}I_Nu_2J^{0-}I_Nu_3}_{L_{txy,\lambda}^2} \\ &\leq  N^{-\frac{3}{2}+} \norm{f}_{X_{0,\frac{1}{2}-2\epsilon,\lambda}} \norm{J^{1-}I_Nu_1}_{L_{txy,\lambda}^4} \norm{J^{1-}I_Nu_2}_{L_{txy,\lambda}^4} \norm{J^{0-}I_Nu_3}_{L_{t}^\infty L_{xy,\lambda}^\infty}.
\end{align*}
We then apply the Sobolev embedding theorem in the $x,y$-variables to the low-frequency factor $u_3$, while the remaining two factors can be handled using estimate \eqref{L^4lambda}. This ultimately yields
\[ ... \les N^{-\frac{3}{2}+} \norm{f}_{X_{0,\frac{1}{2}-2\epsilon,\lambda}} \prod_{i=1}^{3} \norm{I_Nu_i}_{X_{1,\frac{1}{2}+\epsilon,\lambda}}, \]
which again works for sufficiently small $\epsilon > 0$, thereby completing the discussion of this subcase as well. \\
(ii.2)$_{\R \times \T_\lambda}$ \underline{$\abs{(\xi_1,q_1)} \sim \abs{(\xi_3,q_3)}$}: \\
(ii.2.1)$_{\R \times \T_\lambda}$ \underline{$\abs{(\xi_1,q_1)} \gg \abs{(\xi_0,q_0)}$}: \\
In this situation, we write
\[ Diff_{\text{mZK}_1} \les N^{3s-3} \abs{(\xi_1,q_1)}^{3-3s}, \]
which allows us to deduce the pointwise bound
\begin{align*}
\abs{\xi_0} \langle (\xi_0,q_0) \rangle Diff_{\text{mZK}_1} &\les N^{3s-3} \abs{(\xi_1,q_1)}^{\frac{3}{2}-3s+} \langle (\xi_2,q_2) \rangle^{1-} \langle (\xi_3,q_3) \rangle^{1-} \\ & \ \ \ \cdot \abs{\abs{(\xi_1,q_1)}^2-\abs{(\xi_0,q_0)}^2}^\frac{1}{2} \langle (\xi_0,q_0) \rangle^{0-} \langle (\xi_1,q_1) \rangle^{\frac{1}{2}-}.
\end{align*}
Taking into account that $N^{3s-3} \abs{(\xi_1,q_1)}^{\frac{3}{2}-3s+} \les N^{-\frac{3}{2}+}$ holds for $s> \frac{1}{2}$, it follows - after undoing Plancherel and applying Hölder's inequality - that
\begin{align*}
I_f &\les N^{-\frac{3}{2}+} \norm{MP_\lambda(J^{0-}f,J^{\frac{1}{2}-}I_N\widetilde{u}_1)}_{L_{txy,\lambda}^2} \norm{J^{1-}I_Nu_2J^{1-}I_Nu_3}_{L_{txy,\lambda}^2} \\ & \leq N^{-\frac{3}{2}+}\norm{MP_\lambda(J^{0-}f,J^{\frac{1}{2}-}I_N\widetilde{u}_1)}_{L_{txy,\lambda}^2} \norm{J^{1-}I_Nu_2}_{L_{txy,\lambda}^4} \norm{J^{1-}I_Nu_3}_{L_{txy,\lambda}^4},
\end{align*}
and the first factor can now be treated using \eqref{MPlambdadual}, while the second and fourth factors can be handled using \eqref{L^4lambda}. Altogether, this yields
\[ ... \les N^{-\frac{3}{2}+} \norm{f}_{X_{0,\frac{1}{2}-2\epsilon,\lambda}} \prod_{i=1}^{3} \norm{I_Nu_i}_{X_{1,\frac{1}{2}+\epsilon,\lambda}}, \]
which holds provided that $\epsilon > 0$ is chosen sufficiently small. \\
(ii.2.2)$_{\R \times \T_\lambda}$ \underline{$\abs{(\xi_1,q_1)} \sim \abs{(\xi_0,q_0)}$}: \\
In this final subcase, we are in the regime
\[ \abs{(\xi_0,q_0)} \sim \abs{(\xi_1,q_1)} \sim \abs{(\xi_2,q_2)} \sim \abs{(\xi_3,q_3)} \gtrsim N, \]
which implies
\[ Diff_{\text{mZK}_1} \les N^{2s-2} \abs{(\xi_1,q_1)}^{2-2s}. \]
As a consequence, we obtain the pointwise bound
\[ \abs{\xi_0} \langle (\xi_0,q_0) \rangle Diff_{\text{mZK}_1} \les N^{2s-2} \abs{(\xi_1,q_1)}^{\frac{11}{12}-2s+} \abs{\xi_0} \langle (\xi_0,q_0) \rangle^{\frac{11}{24}+} \prod_{i=1}^{3} \langle (\xi_i,q_i) \rangle^{\frac{13}{24}-}, \]
and since $N^{2s-2} \abs{(\xi_1,q_1)}^{\frac{11}{12}-2s+} \les N^{-\frac{13}{12}+}$ holds for $s > \frac{11}{24}$, undoing Plancherel and appealing to duality yields
\[ I_f \les N^{-\frac{13}{12}+} \norm{I_x^1(J^{\frac{13}{24}-}I_Nu_1J^{\frac{13}{24}-}I_Nu_2J^{\frac{13}{24}-}I_Nu_3)}_{X_{\frac{11}{24}+,-\frac{1}{2}+2\epsilon,\lambda}}. \]
For this remaining term, we may now invoke the trilinear estimate \eqref{LWPmZKinH^slambda11/24} for some small number $0<\epsilon \ll 1$, to further obtain
\[ ... \les N^{-\frac{13}{12}+} \prod_{i=1}^{3} \norm{I_Nu_i}_{X_{1,\frac{1}{2}+\epsilon,\lambda}}, \]
wherein we harmlessly absorbed the resulting loss of $\lambda^{\widetilde{\epsilon}} \les N^{100\widetilde{\epsilon}} \sim N^{0+}$ into the factor $N^{-\frac{13}{12}+}$. \\
Gathering together all intermediate bounds in the semiperiodic setting, we conclude that
\begin{equation} \label{nonlinearsmoothingmZKlambdarough}
\norm{\partial_x(I_N(u_1u_2u_3) - I_Nu_1I_Nu_2I_Nu_3)}_{X_{1,-\frac{1}{2}+2\epsilon,\lambda}} \les N^{-\frac{13}{12}+} \prod_{i=1}^{3} \norm{I_Nu_i}_{X_{1,\frac{1}{2}+\epsilon,\lambda}},
\end{equation}
which, combined with the preliminary discussion preceding this case-by-case analysis, yields the desired estimate \eqref{varianttrilinmZK}.
\end{proof}

\begin{rem}
Upon careful inspection, one can extract from the proof of \eqref{varianttrilinmZK} that we have in fact established the two commutator estimates \eqref{nonlinearsmoothingmZKinftyrough} and \eqref{nonlinearsmoothingmZKlambdarough} in the following precise form: Let $\lambda \in \intcc{1,N^{100}} \cup \set{\infty}$ and $s \in \intoo{\frac{1}{2},1}$ be arbitrary. Then, for every $\widetilde{\epsilon} > 0$, there exists $\epsilon > 0$ such that
\begin{equation} \label{nonlinearsmoothingmZKprecise}
\begin{aligned}
&\norm{\partial_x(I_N(u_1u_2u_3)-I_Nu_1I_Nu_2I_Nu_3)}_{X_{1,-\frac{1}{2}+2\epsilon',\lambda}} \les_{s,\widetilde{\epsilon},\epsilon'} \\ & \begin{cases} N^{-\frac{3}{2}+\widetilde{\epsilon}} \prod_{i=1}^{3} \norm{I_Nu_i}_{X_{1,\frac{1}{2}+\epsilon',\infty}}, & \text{if} \quad \lambda = \infty, \\ N^{-\frac{13}{12}+\widetilde{\epsilon}} \prod_{i=1}^{3} \norm{I_Nu_i}_{X_{1,\frac{1}{2}+\epsilon',\lambda}}, & \text{if} \quad \lambda \in \intcc{1,N^{100}} \end{cases}
\end{aligned}
\end{equation}
holds for all $0<\epsilon'\leq \epsilon$ and time-localized functions $u_1,u_2,u_3$ with \\ $I_Nu_1,I_Nu_2,I_Nu_3 \in X_{1,\frac{1}{2}+\epsilon',\lambda}$.
\end{rem}

Exactly as described in the proof of \eqref{cubictermZK}, we now obtain, using \eqref{nonlinearsmoothingmZKprecise}, the following two estimates for the quartic term in the modified energy of mZK.

\begin{lemma}
Let $\lambda \in \intcc{1,N^{100}} \cup \set{\infty}$ be arbitrary. Given $\delta \in \R$, $s \in \intoo{\frac{1}{2},1}$, $\widetilde{\epsilon} > 0$, and $\epsilon' > 0$, one has
\begin{equation} \label{quartictermmZK}
\begin{aligned}
&\left| \int_{0}^{\delta} \int_{\R \times \T_\lambda} \Delta I_Nu_1 \partial_x(I_N(u_2 u_3 u_4) - I_Nu_2 I_Nu_3 I_Nu_4) \ \mathrm{d}(x,y) \mathrm{d}t \right| \les_{s, \widetilde{\epsilon}, \epsilon ', \delta} \\ & \begin{cases} N^{-\frac{3}{2}+\widetilde{\epsilon}} \prod_{i=1}^{4} \norm{I_Nu_i}_{X_{1,\frac{1}{2}+\epsilon ',\infty}^{\abs{\delta}}}, & \text{if} \quad \lambda = \infty, \\ N^{-\frac{13}{12}+\widetilde{\epsilon}} \prod_{i=1}^{4} \norm{I_Nu_i}_{X_{1,\frac{1}{2}+\epsilon ',\lambda}^{\abs{\delta}}}, & \text{if} \quad \lambda \in \intcc{1,N^{100}} \end{cases}
\end{aligned}
\end{equation}
for all $u_1,u_2,u_3,u_4$ with $I_Nu_i \in X_{1,\frac{1}{2}+\epsilon',\lambda}^{\abs{\delta}}$ ($i \in \set{1,2,3,4}$).
\end{lemma}

To set the stage for the proof of Theorem \ref{GWPmZK}, it remains to obtain a suitable estimate for the sextic term appearing in the modified energy of mZK. We carry this out in a single argument that  applies simultaneously to both the non-periodic case and the rescaled semiperiodic case.

\begin{lemma} Let $\lambda \in \intcc{1,\infty}$, $\delta \in \R$, $s\in \intoo{\frac{1}{2},1}$, $\widetilde{\epsilon} > 0$, and $\epsilon' > 0$ be given. Then the estimate
\begin{equation} \label{sextictermmZK}
\begin{aligned}
&\left| \int_{0}^{\delta} \int_{\R \times \T_\lambda} I_N(u_1 u_2 u_3) \partial_x(I_N(u_4 u_5 u_6) - I_Nu_4 I_Nu_5 I_Nu_6) \ \mathrm{d}(x,y) \mathrm{d}t \right| \les_{s, \widetilde{\epsilon}, \epsilon', \delta} \\ & N^{-2+\widetilde{\epsilon}} \prod_{i=1}^{6} \norm{I_Nu_i}_{X_{1,\frac{1}{2}+\epsilon', \lambda}^{\abs{\delta}}}
\end{aligned}
\end{equation}
holds for all $u_1,u_2,u_3,u_4,u_5,u_6$ with $I_Nu_i \in X_{1,\frac{1}{2}+\epsilon',\lambda}^{\abs{\delta}}$ ($i\in \set{1,2,3,4,5,6}$).
\end{lemma}

\begin{proof}
As in the proof of \eqref{quartictermZK}, we can restrict ourselves to verifying the estimate
\begin{equation} \label{sextictermmZKnodelta}
\begin{aligned}
&\left| \int_{\R} \int_{\R \times \T_\lambda} I_N(v_1 v_2 v_3) \partial_x(I_N(v_4 v_5 v_6) - I_Nv_4 I_Nv_5 I_Nv_6) \ \mathrm{d}(x,y) \mathrm{d}t \right| \les_{s, \widetilde{\epsilon}, \epsilon'} \\ & N^{-2+\widetilde{\epsilon}} \prod_{i=1}^{6} \norm{I_Nv_i}_{X_{1,\frac{1}{2}+\epsilon', \lambda}}
\end{aligned}
\end{equation}
for all functions $v_i$ with $I_Nv_i \in  X_{1,\frac{1}{2}+\epsilon',\lambda}$ and $\widehat{\! v_i}^\lambda \geq 0$, where $i \in \set{1,2,3,4,5,6}$. To make the left-hand side of \eqref{sextictermmZKnodelta} more amenable to a case-by-case analysis in Fourier space, we use Parseval's identity to pass to
\begin{align*}
&\frac{1}{\lambda^5} \left| \int_{\R^2} \sum_{q_0 \in \Z/\lambda} \int_{\R^8} \sum_{\substack{q_1,q_2,q_4,q_5 \in \Z/\lambda \\ \ast}} \xi_0 \frac{m(\xi_0,q_0)(m(\xi_0,q_0)-m(\xi_4,q_4)m(\xi_5,q_5)m(\xi_6,q_6))}{\prod_{i=1}^{6}m(\xi_i,q_i)} \right. \\
& \cdot \left. \prod_{i=1}^{6} \widehat{I_Nv_i}^\lambda(\tau_i,\xi_i,q_i) \ \mathrm{d}(\tau_1,\tau_2,\tau_4,\tau_5,\xi_1,\xi_2,\xi_4,\xi_5) \mathrm{d}(\tau_0,\xi_0) \right| \eqqcolon Int
\end{align*}
subject to the convolution constraint $(\tau_0,\xi_0,q_0) = (\tau_1+\tau_2+\tau_3,\xi_1+\xi_2+\xi_3,q_1+q_2+q_3) = (\tau_4+\tau_5+\tau_6, \xi_4+\xi_5+\xi_6, q_4+q_5+q_6)$, and with the obvious modification in the case $\lambda = \infty$. Moreover, by symmetry, we may assume without loss of generality that $\abs{(\xi_1,q_1)} \geq \abs{(\xi_2,q_2)} \geq \abs{(\xi_3,q_3)}$ and $\abs{(\xi_4,q_4)} \geq \abs{(\xi_5,q_5)} \geq \abs{(\xi_6,q_6)}$, and we begin with case \\ \\
(i) \underline{$\abs{(\xi_4,q_4)} \ll N$}: \\
This case is trivial, since the identity
\[ m(\xi_0,q_0) - m(\xi_4,q_4)m(\xi_5,q_5)m(\xi_6,q_6) = 1 - 1 \cdot 1 \cdot 1 = 0 \]
forces $Int$ to be zero. \\
(ii) \underline{$\abs{(\xi_4,q_4)} \gtrsim N$}: \\
(ii.1) \underline{$\abs{(\xi_1,q_1)} \ll N$}: \\
In this situation, the convolution constraint forces $\abs{(\xi_4,q_4)} \sim \abs{(\xi_5,q_5)} \gtrsim N$, so that we may pass to
\begin{align*} Diff_{\text{mZK}_2} &\coloneqq \frac{m(\xi_0,q_0)(m(\xi_0,q_0)-m(\xi_4,q_4)m(\xi_5,q_5)m(\xi_6,q_6))}{\prod_{i=1}^{6}m(\xi_i,q_i)} \\ &\les N^{2s-2} \abs{(\xi_4,q_4)}^{2-2s} \frac{1}{m(\xi_6,q_6)},
\end{align*}
and hence
\begin{align*}
\abs{\xi_0} Diff_{\text{mZK}_2} &\les N^{2s-2} \abs{(\xi_4,q_4)}^{-2s+} \frac{\langle (\xi_6,q_6) \rangle^{-1+}}{m(\xi_6,q_6)} \langle (\xi_1,q_1) \rangle^{1-} \langle (\xi_4,q_4) \rangle^{1-} \\ & \ \ \ \cdot \langle (\xi_5,q_5) \rangle^{1-} \langle (\xi_6,q_6) \rangle^{1-} \langle (\xi_2,q_2) \rangle^{0-} \langle (\xi_3,q_3) \rangle^{0-}, 
\end{align*}
where we have used that $\abs{\xi_0} \les \abs{(\xi_1,q_1)} \land \abs{(\xi_4,q_4)} = \abs{(\xi_1,q_1)}$.
Now, taking into account that $N^{2s-2} \abs{(\xi_4,q_4)}^{-2s+} \les N^{-2+\widetilde{\epsilon}}$ for $s>0$, we undo the application of Parseval's indentity and apply Hölder's inequality to obtain
\begin{align*}
Int & \les N^{-2+\widetilde{\epsilon}} \left( \prod_{\substack{i=1 \\ i \neq 2,3}}^{6} \norm{J^{1-}I_Nv_i}_{L^4_{txy,\lambda}} \right) \norm{J^{0-}I_Nv_2}_{L_{t}^\infty L_{xy,\lambda}^\infty}  \norm{J^{0-}I_Nv_3}_{L_{t}^\infty L_{xy,\lambda}^\infty}.
\end{align*}
By applying the estimate \eqref{L^4lambda} or \eqref{L^4non-periodic} four times and the Sobolev embedding theorem twice, it then follows that
\[... \les_{s,\widetilde{\epsilon},\epsilon'} N^{-2+\widetilde{\epsilon}} \prod_{i=1}^{6} \norm{I_Nv_i}_{X_{1,\frac{1}{2}+\epsilon',\lambda}}, \]
which is exactly the bound we aimed to obtain. \\
(ii.2) \underline{$\abs{(\xi_1,q_1)} \gtrsim N$}: \\
We globally order the quantities $\abs{(\xi_i,q_i)}$ by size, denote them by $\abs{(\xi_{i_1},q_{i_1})} \geq \abs{(\xi_{i_2},q_{i_2})} \geq \abs{(\xi_{i_3},q_{i_3})} \geq \abs{(\xi_{i_4},q_{i_4})} \geq \abs{(\xi_{i_5},q_{i_5})} \geq \abs{(\xi_{i_6},q_{i_6})}$, and consider two subcases: \\
(ii.2.1) \underline{$\abs{(\xi_{i_4},q_{i_4})} \ll N$}: \\
This case splits into two parts. If $\abs{(\xi_5,q_5)} \ll N$, then $\abs{(\xi_3,q_3)}, \abs{(\xi_5,q_5)},\abs{(\xi_6,q_6)} \ll N$ and $\abs{(\xi_0,q_0)} \sim \abs{(\xi_4,q_4)}$, so that the mean value theorem yields
\[ Diff_{\text{mZK}_2} \les \abs{(\xi_1,q_1)}^{1-s} \abs{(\xi_4,q_4)}^{s-2+} \frac{1}{m(\xi_2,q_2)} \abs{(\xi_5,q_5)}^{1-}. \]
Since $\abs{\xi_0} \les \abs{(\xi_1,q_1)} \land \abs{(\xi_4,q_4)}$, this further implies
\begin{align*}
\abs{\xi_0} Diff_{\text{mZK}_2} &\les \abs{(\xi_1,q_1)}^{-s+} \abs{(\xi_4,q_4)}^{s-2+} \frac{\langle (\xi_2,q_2) \rangle^{-1+}}{m(\xi_2,q_2)} \langle (\xi_1,q_1) \rangle^{1-} \\ & \ \ \ \cdot \langle (\xi_4,q_4) \rangle^{1-} \langle (\xi_2,q_2) \rangle^{1-} \langle (\xi_5,q_5) \rangle^{1-} \langle (\xi_3,q_3) \rangle^{0-} \langle (\xi_6,q_6) \rangle^{0-} \\ & \les N^{-2+\widetilde{\epsilon}} \left( \prod_{\substack{i=1 \\ i \neq 3,6}}^{6} \langle (\xi_i,q_i) \rangle^{1-} \right) \langle (\xi_3,q_3) \rangle^{0-} \langle (\xi_6,q_6) \rangle^{0-},
\end{align*}
where in the last step we have used that $\abs{(\xi_1,q_1)}^{-s+} \abs{(\xi_4,q_4)}^{s-2+} \frac{\langle (\xi_2,q_2) \rangle^{-1+}}{m(\xi_2,q_2)} \les N^{-2+\widetilde{\epsilon}}$ for $s \in \intoo{0,2}$. \\ If, on the other hand, $\abs{(\xi_5,q_5)} \gtrsim N$, it follows that $\abs{(\xi_2,q_2)},\abs{(\xi_3,q_3)},\abs{(\xi_6,q_6)} \ll N$, and thus also $\abs{(\xi_0,q_0)} \sim \abs{(\xi_1,q_1)}$, which allows us to infer the bound
\[ Diff_{\text{mZK}_2} \les N^{2s-2} \abs{(\xi_4,q_4)}^{1-s} \abs{(\xi_5,q_5)}^{1-s}. \]
This then leads us to
\begin{align*}
\abs{\xi_0} Diff_{\text{mZK}_2} &\les N^{2s-2} \abs{(\xi_4,q_4)}^{-s+} \abs{(\xi_5,q_5)}^{-s+} \langle (\xi_2,q_2) \rangle^{-1+} \langle (\xi_1,q_1) \rangle^{1-} \\ & \ \ \ \cdot \langle (\xi_2,q_2) \rangle^{1-} \langle (\xi_4,q_4) \rangle^{1-} \langle (\xi_5,q_5) \rangle^{1-} \langle (\xi_3,q_3) \rangle^{0-} \langle (\xi_6,q_6) \rangle^{0-} \\ & \les N^{-2+\widetilde{\epsilon}} \left( \prod_{\substack{i=1 \\ i \neq 3,6}}^{6} \langle (\xi_i,q_i) \rangle^{1-} \right) \langle (\xi_3,q_3) \rangle^{0-} \langle (\xi_6,q_6) \rangle^{0-}
\end{align*}
(note that $\abs{\xi_0} \les \abs{(\xi_1,q_1)} \land \abs{(\xi_4,q_4)}$ and $ N^{2s-2} \abs{(\xi_4,q_4)}^{-s+} \abs{(\xi_5,q_5)}^{-s+} \langle (\xi_2,q_2) \rangle^{-1+} \\ \les N^{-2+\widetilde{\epsilon}}$ for $s > 0$), and we therefore arrive at the same pointwise estimate in both cases. By undoing Plancherel and a subsequent application of Hölder's inequality, we may therefore deduce
\begin{align*}
Int &\les N^{-2+\widetilde{\epsilon}} \left( \prod_{\substack{i=1 \\ i \neq 3,6}}^{6} \norm{J^{1-}I_Nv_i}_{L_{txy,\lambda}^4} \right) \norm{J^{0-}I_Nv_3}_{L_{t}^\infty L_{xy,\lambda}^\infty} \norm{J^{0-}I_Nv_6}_{L_{t}^\infty L_{xy,\lambda}^\infty},
\end{align*}
which can be further estimated by means of \eqref{L^4lambda} or \eqref{L^4non-periodic} and the Sobolev embedding theorem to finally obtain
\[ ...\les_{s,\widetilde{\epsilon},\epsilon'} N^{-2+\widetilde{\epsilon}} \prod_{i=1}^{6} \norm{I_Nv_i}_{X_{1,\frac{1}{2}+\epsilon',\lambda}}. \]
This completes the discussion of this subcase. \\
(ii.2.2) \underline{$\abs{(\xi_{i_4},q_{i_4})} \gtrsim N$}: \\
In this final scenario we have
\[ Diff_{\text{mZK}_2} \les N^{4s-4} \left( \prod_{j=1}^{4} \abs{(\xi_{i_j},q_{i_j})}^{1-s} \right) \frac{1}{m(\xi_{i_5},q_{i_5})m(\xi_{i_6},q_{i_6})}, \]
and this, together with $\abs{\xi_0} \les \abs{(\xi_1,q_1)} \land \abs{(\xi_4,q_4)}$ and the fact that $1,4 \in \set{i_1,i_2,i_3,i_4}$, allows us to pass to
\begin{align*}
\abs{\xi_0} Diff_{\text{mZK}_2} &\les N^{4s-4} \left( \prod_{j=1}^{4} \abs{(\xi_{i_j},q_{i_j})}^{\frac{1}{2}-s+} \right) \frac{\langle (\xi_{i_5},q_{i_5}) \rangle^{-\frac{1}{2}+} \langle (\xi_{i_6},q_{i_6}) \rangle^{-\frac{1}{2}+}}{m(\xi_{i_5},q_{i_5})m(\xi_{i_6},q_{i_6})} \\ & \ \ \ \cdot \left( \prod_{j=1}^{4} \langle (\xi_{i_j},q_{i_j}) \rangle^{1-} \right) \langle (\xi_{i_5},q_{i_5}) \rangle^{0-} \langle (\xi_{i_6},q_{i_6}) \rangle^{0-}.
\end{align*}
Now, for $s > \frac{1}{2}$ we have
\[ N^{4s-4} \left( \prod_{j=1}^{4} \abs{(\xi_{i_j},q_{i_j})}^{\frac{1}{2}-s+} \right) \frac{\langle (\xi_{i_5},q_{i_5}) \rangle^{-\frac{1}{2}+} \langle (\xi_{i_6},q_{i_6}) \rangle^{-\frac{1}{2}+}}{m(\xi_{i_5},q_{i_5})m(\xi_{i_6},q_{i_6})} \les N^{-2+\widetilde{\epsilon}}, \]
so that undoing Plancherel and applying Hölder's inequality gives
\[ Int \les N^{-2+\widetilde{\epsilon}} \left( \prod_{j=1}^{4} \norm{J^{1-}I_Nv_{i_j}}_{L_{txy,\lambda}^4} \right) \norm{J^{0-}I_Nv_{i_5}}_{L_{t}^\infty L_{xy,\lambda}^\infty} \norm{J^{0-}I_Nv_{i_6}}_{L_{t}^\infty L_{xy,\lambda}^\infty}.  \]
As discussed earlier, applying \eqref{L^4lambda} or \eqref{L^4non-periodic} four times and the Sobolev embedding theorem twice then yields
\[ ... \les_{s,\widetilde{\epsilon},\epsilon'} N^{-2+\widetilde{\epsilon}} \prod_{i=1}^{6} \norm{I_Nv_i}_{X_{1,\frac{1}{2}+\epsilon',\lambda}}, \]
which is exactly what we wanted to establish.
\end{proof}

\begin{rem} Bhattacharya, Farah, and Roudenko obtained the same $N^{-2+}$ bound in their analysis of the non-periodic case \cite{Bhattacharya2020} for the sextic term of the modified energy, which arises when the mZK equation is symmetrized in the manner of Grünrock and Herr \cite{GrünrockHerr2014}.
\end{rem}

Having established all the required preliminaries for mZK, we now turn to the

\begin{proof}[Proof of Theorem \ref{GWPmZK}]
We first focus on the semiperiodic case. Without loss of generality, let $s \in \intoo{\frac{36}{49},1}$, that is, $s = \frac{36}{49}+\widetilde{\epsilon}$ for some $\widetilde{\epsilon} \in \intoo{0,\frac{13}{49}}$. Moreover, let $u_0 \in H^s(\R \times \T)$ (satisfying the smallness condition specified in Theorem \ref{GWPmZK} in the focusing case of mZK) and fix an arbitrary $T>0$.
Our objective is to show that the (local) solution $u$ of $\mathrm{(CP}_{2,\R \times \T} \mathrm{)}$ exists on the entire time interval $\intcc{-T,T}$. As already used in the proof of Theorem \ref{GWPZK}, this is equivalent to showing that the rescaled solution $u_\lambda$ of $\mathrm{(CP}_{2,\R \times \T_\lambda} \mathrm{)}$ persists on the interval $\intcc{-\lambda^3T,\lambda^3T}$.
To this end, we choose
\[ \lambda = \lambda_{\text{mZK}} = C_{\text{mZK}} N^{\frac{1-s}{s}} \]
as in Remark \ref{RemarklambdaZKandlambdamZK}, and since $\lambda_{\text{mZK}} \les N^\frac{13}{36} \leq N^{100}$, Proposition \ref{variantLWPmZK} yields the lower bound
\[ \delta_0 \geq c\norm{I_Nu_{0,\lambda_{\text{mZK}}}}_{H^1_{\lambda{\text{mZK}}}}^{-\gamma} \]
for the initial lifespan $\delta_0 > 0$ of the rescaled solution $u_{\lambda_{\text{mZK}}}$, where $\gamma$ is some fixed positive real number. By combining \eqref{controloverH^1mZKsp}, Lemma \ref{modifiedenergylemma}, \eqref{quartictermmZK}, \eqref{sextictermmZK}, and the continuity estimate \eqref{contdepmZKlambda}, we obtain
\begin{align*}
&\norm{I_Nu_{\lambda_{\text{mZK}}}(\delta_0)}_{H^1_{\lambda_{\text{mZK}}}}^2  \leq C_0\left( \norm{u_0}_{L_1^2}, E_\pm[I_Nu_{\lambda_{\text{mZK}}}](0) \right) \\ & + \widetilde{C} \left( N^{-\frac{13}{12}+\frac{\widetilde{\epsilon}}{2}} \norm{I_Nu_{0,\lambda_{\text{mZK}}}}_{H^1_{\lambda_{\text{mZK}}}}^4 + N^{-2+\frac{\widetilde{\epsilon}}{2}} \norm{I_Nu_{0,\lambda_{\text{mZK}}}}_{H^1_{\lambda_{\text{mZK}}}}^6 \right),
\end{align*}
and our choice of $\lambda_{\text{mZK}}$ (see Remark \ref{RemarklambdaZKandlambdamZK} (ii)) allows us to further infer the bound
\[ C_0\left( \norm{u_0}_{L_1^2}, E_\pm[I_Nu_{\lambda_{\text{mZK}}}](0) \right) \leq \widetilde{C}_0 = \widetilde{C}_0 \left( \norm{u_0}_{L_1^2} \right). \]
Choosing $N \gg 1$ sufficiently large such that
\begin{align*}
&\widetilde{C} \left( N^{-\frac{13}{12}+\frac{\widetilde{\epsilon}}{2}} \left( N^{\frac{13}{12}-\widetilde{\epsilon}}4\widetilde{C}_0^2 + \norm{I_Nu_{0,\lambda_{\text{mZK}}}}_{H^1_{\lambda_{\text{mZK}}}}^4 \right) \right. \\ & + \left. N^{-2+\frac{\widetilde{\epsilon}}{2}} \left( N^{\frac{13}{12}-\widetilde{\epsilon}}8\widetilde{C}_0^3 + \norm{I_Nu_{0,\lambda_{\text{mZK}}}}_{H^1_{\lambda_{\text{mZK}}}}^6 \right) \right) \leq \widetilde{C}_0, 
\end{align*}
the fixed point argument underlying Proposition \ref{variantLWPmZK} can then be iterated $N^{\frac{13}{12}-\widetilde{\epsilon}}$ times, with the lifespan being extended by at least
\[ c (2\widetilde{C}_0)^{-\frac{\gamma}{2}} \eqqcolon \delta^* \]
both forward and backward in time at each iteration step.
In order to reach the desired lifespan through these iterations, we now require
\[ N^{\frac{13}{12}-\widetilde{\epsilon}} \delta^* > \lambda_{\text{mZK}}^3T = C_{\text{mZK}}^3 N^{\frac{3(1-s)}{s}} T, \]
which is satisfied for large $N$, provided that
\[ \frac{13}{12}-\widetilde{\epsilon} - \frac{3(1-s)}{s} > 0 \Leftrightarrow s > \frac{3}{\frac{49}{12}-\widetilde{\epsilon}}. \]
A closer inspection now shows that our initially fixed choice of $s$ satisfies this constraint. Consequently, after scaling back to the original problem, we conclude that the solution $u$ exists on the entire interval $\intcc{-T,T}$, thereby closing the argument.
Moreover, in the preceding argument we may choose
\[ N \sim (1+T)^{\frac{12s}{49s-36}+}, \]
which, in conjunction with $\lambda_{\text{mZK}} \sim N^{\frac{1-s}{s}}$ and the bound
\[\norm{u(\lambda_{\text{mZK}}^{-3}t)}_{H^s_1} \les \lambda_{\text{mZK}}^s \norm{I_Nu_{\lambda_{\text{mZK}}}(t)}_{H^1_{\lambda_{\text{mZK}}}}, \]
leads us to
\[\norm{u(\lambda_{\text{mZK}}^{-3}t)}_{H^s_1} \les (1+T)^{\frac{12s(1-s)}{49s-36}+} \]
for all $t\in \intcc{-\lambda_{\text{mZK}}^3T,\lambda_{\text{mZK}}^3T}$.
Taking the supremum over this interval then finally yields the polynomial bound asserted in \eqref{polygrowthmZKbelowH^1} for the semiperiodic case. \\
The non-periodic case is treated in a completely analogous manner: In this setting, the inequality \eqref{controloverH^1mZKnp} is used in place of \eqref{controloverH^1mZKsp}, and the improved bound $N^{-\frac{3}{2}+\frac{\widetilde{\epsilon}}{2}}$ from \eqref{quartictermmZK} is employed. This allows for $N^{\frac{3}{2}-\widetilde{\epsilon}}$ iterations of the fixed point argument, ultimately leading to the condition
\[ N^{\frac{3}{2}-\widetilde{\epsilon}} > CN^{\frac{3(1-s)}{s}}T, \]
and hence to the constraint
\[ \frac{3}{2}-\widetilde{\epsilon} - \frac{3(1-s)}{s} > 0 \Leftrightarrow s > \frac{3}{\frac{9}{2}-\widetilde{\epsilon}} \]
required to close the argument. By choosing $\widetilde{\epsilon} > 0$ sufficiently small, one sees that any prescribed $s>\frac{2}{3}$ satisfies this constraint, thereby proving the asserted global well-posedness result.
Moreover, one can read off that the initial inequality is satisfied for the choice
\[ N \sim (1+T)^{\frac{2s}{3(3s-2)}+}. \]
Combining this once again with the relation $\lambda_{\text{mZK}} \sim N^{\frac{1-s}{s}}$ and the estimate
\[\norm{u(\lambda_{\text{mZK}}^{-3}t)}_{H^s_\infty} \les \lambda_{\text{mZK}}^s \norm{I_Nu_{\lambda_{\text{mZK}}}(t)}_{H^1_\infty}, \]
we obtain
\[ \norm{u(\lambda_{\text{mZK}}^{-3}t)}_{H^s_\infty} \les (1+T)^{\frac{2s(1-s)}{3(3s-2)}+} \]
for all $t\in \intcc{-\lambda_{\text{mZK}}^3T,\lambda_{\text{mZK}}^3T}$. Since the right-hand side of this inequality is independent of the time variable $t$, we obtain the polynomial bound asserted in \eqref{polygrowthmZKbelowH^1} for the non-periodic case. This completes the proof of Theorem \ref{GWPmZK}.
\end{proof}

\begin{rem}
One can see from the proof that the decay rates $N^{-\frac{13}{12}+}$ and $N^{-\frac{3}{2}+}$ of the respective quartic terms in the modified energy determine the strength of the resulting global results. Similar to what we achieved in the case of ZK, the extracted exponents $-\frac{13}{12}+ = -(1+\frac{1}{12})+$ and $-\frac{3}{2}+=-(1+\frac{1}{2})+$ realize the full strength of the nonlinear smoothing effect, as manifested in the trilinear estimates \eqref{LWPmZKinH^slambda11/24} and \eqref{LWPmZKinH^s1/4}, respectively.
\end{rem}

\section{Polynomial growth bounds for generalized ZK on $\mathbb{R} \times \mathbb{T}$}

\subsection{General arguments}

We devote this section to the proof of Theorem \ref{GrowthgZK}, which, as previously mentioned, is methodologically inspired by the ideas of Bourgain \cite{Bourgain1993, Bourgain19932} and Staffilani \cite{Staffilani1997}. In addition to the linear and bilinear estimates fixed in Section 3, we require one more technical ingredient, which is reminiscient of Grönwall's integral inequality. We state it here for later use (see, e.g., Lemma 13 in \cite{Valet2021}):

\begin{lemma} \label{Gronwall}
 Let $K_{1} > 0$, $ \epsilon \in (0,1) $, and let $(a_{m})_{m \in \mathbb{N}_{0}} $ be a sequence of nonnegative real numbers with 
\[ a_{m+1} \leq a_{m} + K_{1} \left( 1+a_{m}^{1- \epsilon} \right) \qquad \forall  m \in \mathbb{N}_{0}. \]
Then for every $d > \frac{1}{\epsilon} $, there exists a constant $K_{2} = K_{2}(K_{1},d) $, such that 
\[ a_{m+1} \leq K_{2}(1+m)^{d} (1+a_{0}) \qquad \forall  m \in \mathbb{N}_{0} \]
holds.
\end{lemma}

For the remainder of this section, let $k \in \N$, $s \in 2\N$, and let $u_0 \in H^s(\R \times \T)$ be a (small) real-valued initial datum. Moreover, we denote the associated unique global real-valued solution of  $\mathrm{(CP}_{k, \R \times \T} \mathrm{)}$ by $u \in C(\R,H^s(\R \times \T)) \cap L^\infty(\R, H^1(\R \times \T)) \cap \left( \bigcap_{T > 0} X_{s,\frac{1}{2}+}^T \right)$, and we note that the existence of such a solution was established by Molinet and Pilod in the case $k=1$ (see Theorem 1.3 in \cite{Pilod2015}), while the cases $k\geq 2$ were treated by Farah and Molinet (see Theorem 1.2 in \cite{Farah2024} - in general, one has to impose a smallness condition on the $L^2$- or $H^1$-norm of the initial data). \\
We now outline the strategy of the proof: Let $\delta_k > 0$ denote the lifespan of the solution $u$ obtained after a single iteration of the fixed-point argument. For $t \in \intcc{0,\frac{\delta_k}{2}}$, we consider differences
\[ \norm{u(t)}_{\dot{H}^s}^2 - \norm{u(0)}_{\dot{H}^s}^2, \]
and derive an estimate of the form
\[ \norm{u(t)}_{H^s} \leq \norm{u(0)}_{H^s} + C\norm{u(0)}_{H^s}^{1-\epsilon}, \]
by strongly relying on the Strichartz-type estimates collected in Section 3 and on the continuous dependence of the solution on the initial data. Having accomplished this, Theorem \ref{GrowthgZK} can then be proved with the help of Lemma \ref{Gronwall} in the context of a simple iteration argument. \\
We start with

\begin{lemma} \label{equalityfordifferencek>1}
For $k \in \N^{\geq 2}$ and $t \in \intcc{0,\frac{\delta_k}{2}}$, it holds that
\begin{align*}
&\norm{u(t)}_{\dot{H}^{s}}^{2} -  \norm{u(0)}_{\dot{H}^{s}}^{2} = \\ & \mp  \sum_{\substack{ 0 \leq \abs{\alpha_{i}} \leq s \\ \abs{\alpha_{1}} + ...+ \abs{\alpha_{k+1}} = s+1}} C(\alpha_{1},... , \alpha_{k+1}) \int_{0}^{t} \langle \prod_{i=1}^{k+1} D^{\alpha_{i}}u(t'), I^{s}u(t') \rangle_{L^{2}_{xy}} \ dt'.
\end{align*}
\end{lemma}

\begin{proof}
Since $u$ is a real-valued solution of $k$-gZK, we have that
\begin{align*} 
\norm{u(t)}_{\dot{H}^{s}}^{2} - \norm{u(0)}_{\dot{H}^{s}}^{2} &= \int_{0}^{t} \frac{\partial}{\partial t'} \lVert u(t') \rVert_{\dot{H}^{s}}^{2} \ \mathrm{d}t' \\ & = \int_{0}^{t} \frac{\partial}{\partial t'} \langle I^{s}u(t'),I^{s}u(t') \rangle_{L_{xy}^{2}} \ \mathrm{d}t' \\ &= 2 \int_{0}^{t} \langle I^{s} \frac{\partial u}{\partial t'} \left(t' \right),I^{s}u(t') \rangle_{L_{xy}^{2}} \ \mathrm{d}t' \\ &= \underbrace{-2 \int_{0}^{t} \langle I^{s} \frac{\partial}{\partial x} \Delta u(t'), I^{s}u(t') \rangle_{L_{xy}^{2}} \ \mathrm{d}t'}_{\eqqcolon \ (I)} \\ & \ \ \ \ \underbrace{\mp 2 \int_{0}^{t} \langle I^{s} \frac{\partial}{\partial x} \left( u^{k+1} \right) (t'),I^{s}u(t') \rangle_{L_{xy}^{2}} \ \mathrm{d}t'}_{\eqqcolon \ (II)}. \end{align*}

Integration by parts at the $L^2_{xy}$ level then shows that $(I)$ vanishes, and to further treat $(II)$, we can write
\begin{align*} I^{s} \left( u^{k+1} \right) (t') &= (k+1) I^{s} u(t') \cdot (u(t'))^{k} \\ & \ \ \  + \sum_{\substack{0 \leq \lvert \alpha_{i} \rvert < s \\ \lvert \alpha_{1} \rvert + ... + \lvert \alpha_{k+1} \rvert = s}} \widetilde{C}(\alpha_{1},\dots, \alpha_{k+1}) \prod_{i=1}^{k+1} D^{\alpha_{i}}u(t')  
\end{align*}
with suitable real numbers $\widetilde{C}(\alpha_{1},\dots, \alpha_{k+1})$, using the assumption that $s \in 2\N$. This leads us to
\begin{align*}
&(II) = \mp 2 (k+1) \int_{0}^{t} \langle \frac{\partial}{\partial x} \left( I^{s}u(t') \cdot (u(t'))^{k} \right), I^{s}u(t') \rangle_{L_{xy}^{2}} \ \mathrm{d}t' \\ & \mp 2 \sum_{\substack{0 \leq \abs{\alpha_{i}} < s \\ \abs{\alpha_{1}} +...+ \abs{\alpha_{k+1}} =s}} \widetilde{C}(\alpha_{1},...,\alpha_{k+1}) \int_{0}^{t} \langle \frac{\partial}{\partial x} \left( \prod_{i=1}^{k+1} D^{\alpha_{i}}u(t') \right), I^{s}u(t') \rangle_{L_{xy}^{2}} \ \mathrm{d}t',
\end{align*}
and in the case where all $s$ derivatives fall on a single factor, we see that
\begin{align*} \langle \frac{\partial}{\partial x} \left( I^{s}u(t') \cdot (u(t'))^{k} \right), I^{s}u(t') \rangle_{L_{xy}^{2}} &=  \langle \frac{\partial}{\partial x} I^{s}u(t') \cdot (u(t'))^{k}, I^{s}u(t') \rangle_{L_{xy}^{2}} \\ & \ \ \ + k \cdot \langle  I^{s}u(t')\cdot (u(t'))^{k-1} \frac{\partial u}{\partial x} \left(t' \right), I^{s}u(t') \rangle_{L_{xy}^{2}} \\ &= \langle \frac{\partial}{\partial x}I^{s}u(t'),I^{s}u(t') \cdot (u(t'))^{k} \rangle_{L_{xy}^{2}} + ...  \\ &= - \langle I^{s}u(t'), \frac{\partial}{\partial x} \left( I^{s}u(t') \cdot (u(t'))^{k} \right) \rangle_{L_{xy}^{2}} + ... \\ &= - \langle \frac{\partial}{\partial x} \left( I^{s}u(t') \cdot (u(t'))^{k} \right), I^{s}u(t') \rangle_{L_{xy}^{2}} + .... 
\end{align*}
Hence, we have
\[ \langle \frac{\partial}{\partial x} \left( I^{s}u(t') \cdot (u(t'))^{k} \right), I^{s}u(t') \rangle_{L_{xy}^{2}} = \frac{k}{2} \cdot \langle I^{s}u(t') \cdot (u(t'))^{k-1} \frac{\partial u}{\partial x} \left( t' \right), I^{s}u(t') \rangle_{L_{xy}^{2}}, \]
and on the right-hand side there is now at least one derivative acting on another factor. Since all remaining contributions of $(II)$ are of a similar form, we may summarize and finally obtain
\[ (II) = \mp  \sum_{\substack{ 0 \leq \lvert \alpha_{i} \rvert \leq s \\ \lvert \alpha_{1} \rvert + ...+ \lvert \alpha_{k+1} \rvert = s+1}} C(\alpha_{1},\dots , \alpha_{k+1}) \int_{0}^{t} \langle \prod_{i=1}^{k+1} D^{\alpha_{i}}u(t'), I^{s}u(t') \rangle_{L_{xy}^{2}} \ \mathrm{d}t'. \]
This completes the proof.
\end{proof}

While the computations in Lemma \ref{equalityfordifferencek>1} remain valid for $k=1$, they are not well suited to this case, as it will become important, for certain frequency configurations, to have the $x$-derivative acting on the product. We therefore present an alternative approach for ZK in the following

\begin{lemma} \label{equalityfordifferencek=1}
Let $k=1$ and $t \in \intcc{0,\frac{\delta_1}{2}}$. Moreover, we define the three Fourier projectors $Pr_{\set{\abs{(\xi_1,q_1)} \gg \abs{(\xi_2,q_2)}}}$, $Pr_2$, and $Pr_3$ by means of their respective Fourier transforms:
\begin{align*} 
&\Fcal_{xy}(Pr_{\set{\abs{(\xi_1,q_1)} \gg \abs{(\xi_2,q_2)}}}(v_1v_2))(\xi,q) \coloneqq \\ & \frac{1}{(2\pi)^2} \int_{\R} \sum_{q_1 \in \Z} \chi_{\set{\abs{(\xi_1,q_1)} \geq 100 \abs{(\xi-\xi_1,q-q_1)}}} \Fcal_{xy}v_1(\xi_1,q_1) \Fcal_{xy}v_2(\xi-\xi_1,q-q_1) \ \mathrm{d}\xi_1,
\end{align*}

\begin{align*}
&\Fcal_{xy}(Pr_{2}(v_1v_2))(\xi,q) \coloneqq \\ & \frac{1}{(2\pi)^2} \int_{\R} \sum_{q_1 \in \Z} 2\chi_{\set{\abs{(\xi_1,q_1)} < 100 \abs{(\xi-\xi_1,q - q_1)} \ \text{and} \ \abs{(\xi,q)} < 100 \abs{(\xi - \xi_1, q-q_1)}}} \\ & \cdot \Fcal_{xy}v_1(\xi_1,q_1) \Fcal_{xy}v_2(\xi-\xi_1,q-q_1) \ \mathrm{d} \xi_1,
\end{align*}
and

\begin{align*}
&\Fcal_{xy}(Pr_{3}(v_1v_2))(\xi,q) \coloneqq \\ & \frac{1}{(2\pi)^2} \int_{\R} \sum_{q_1 \in \Z} \chi_{\set{ \text{either} \ \abs{(\xi_1,q_1)} < 100 \abs{(\xi-\xi_1,q-q_1)} \ \text{or} \ \abs{(\xi,q)} < 100 \abs{(\xi-\xi_1,q-q_1)}}} \\ & \cdot \Fcal_{xy}v_1(\xi_1,q_1) \Fcal_{xy}v_2(\xi-\xi_1,q-q_1) \ \mathrm{d}\xi_1.
\end{align*}

Then, it holds that

\begin{align*}
& \norm{u(t)}_{\dot{H}^s}^2 - \norm{u(0)}_{\dot{H}^s}^2 = \mp 2 \left( \int_{0}^{t} \langle Pr_{\set{\abs{(\xi_1,q_1)} \gg \abs{(\xi_2,q_2)}}}(I^su(t') \frac{\partial u}{\partial x} (t')),I^su(t') \rangle_{L_{xy}^2} \right. \\ & + \int_{0}^{t} \langle  \frac{\partial}{\partial x}Pr_{2}(I^su(t') u(t')), I^su(t') \rangle_{L_{xy}^2} \ \mathrm{d}t' \\ & + \int_{0}^{t} \langle \frac{\partial}{\partial x}Pr_{3}(I^su(t') u(t')), I^su(t') \rangle_{L_{xy}^2} \ \mathrm{d}t' \\ & \left. + \sum_{\substack{0 \leq \abs{\alpha_{1,2}} < s \\ \abs{\alpha_{1}} + \abs{\alpha_{2}} =s}} C(\alpha_{1},\alpha_{2}) \int_{0}^{t} \langle \frac{\partial}{\partial x} (D^{\alpha_{1}}u(t') D^{\alpha_{2}}u(t') ), I^{s}u(t') \rangle_{L_{xy}^{2}} \ \mathrm{d}t'  \right).
\end{align*}
\end{lemma}

\begin{proof}
We first proceed excactly as in Lemma \ref{equalityfordifferencek>1} and thereby arrive at
\begin{align*}
&\norm{u(t)}_{\dot{H}^s}^2 - \norm{u(0)}_{\dot{H}^s}^2 = \mp 2 \left( 2 \int_{0}^{t} \langle \frac{\partial}{\partial x} (I^su(t')u(t')),I^su(t') \rangle_{L_{xy}^2} \ \mathrm{d}t' \right. \\ & + \left. \sum_{\substack{0 \leq \abs{\alpha_{1,2}} < s \\ \abs{\alpha_{1}} + \abs{\alpha_{2}} =s}} C(\alpha_{1},\alpha_{2}) \int_{0}^{t} \langle \frac{\partial}{\partial x} (D^{\alpha_{1}}u(t') D^{\alpha_{2}}u(t')), I^{s}u(t') \rangle_{L_{xy}^{2}} \ \mathrm{d}t'  \right).
\end{align*}
We observe that the last term is already of the desired form, and we therefore turn our attention to the first integral, in which all $s$ derivatives fall on the first factor of the product:
We readily see that
\begin{align*}
&\langle \frac{\partial}{\partial x} Pr_{\set{\abs{(\xi_1,q_1)} \gg \abs{(\xi_2,q_2)}}}(I^su(t')u(t')), I^su(t') \rangle_{L_{xy}^2}  = \\ & \langle Pr_{\set{\abs{(\xi_1,q_1)} \gg \abs{(\xi_2,q_2)}}}(\frac{\partial}{\partial x}I^su(t') u(t')), I^su(t') \rangle_{L_{xy}^2} \\ & + \langle Pr_{\set{\abs{(\xi_1,q_1)} \gg \abs{(\xi_2,q_2)}}}(I^su(t') \frac{\partial u}{\partial x} (t')),I^su(t') \rangle_{L^2_{xy}},
\end{align*}
and for the first summand, we may employ Parseval's identity, followed by a rearrangement of the factors in Fourier space ($u$ is real-valued), to further deduce that
\begin{align*}
&\langle Pr_{\set{\abs{(\xi_1,q_1)} \gg \abs{(\xi_2,q_2)}}}(\frac{\partial}{\partial x}I^su(t') u(t')), I^su(t') \rangle_{L_{xy}^2} = \\ & - \langle \frac{\partial}{\partial x} Pr_{\set{\abs{(\xi,q)} \gg \abs{(\xi_2,q_2)}}}(I^su(t') u(t')), I^su(t') \rangle_{L_{xy}^2},
\end{align*}
where $Pr_{\set{\abs{(\xi,q)} \gg \abs{(\xi_2,q_2)}}}$ is defined analogously to $Pr_{\set{\abs{(\xi_1,q_1)} \gg \abs{(\xi_2,q_2)}}}$.
Furthermore, we have the pointwise identity
\begin{align*} &\chi_{\set{\abs{(\xi_1,q_1)} \geq 100 \abs{(\xi-\xi_1,q-q_1)}}} + \chi_{\set{\abs{(\xi,q)} \geq 100 \abs{(\xi-\xi_1,q-q_1)}}} \\ & = 2 \chi_{\set{\abs{(\xi_1,q_1)} \geq 100 \abs{(\xi-\xi_1,q-q_1)} \ \text{or} \ \abs{(\xi,q)} \geq 100 \abs{(\xi-\xi_1,q-q_1)}}} \\ & \ \ \ - \chi_{\set{ \text{either} \ \abs{(\xi_1,q_1)} < 100 \abs{(\xi-\xi_1,q-q_1)} \ \text{or} \ \abs{(\xi,q)} < 100 \abs{(\xi-\xi_1,q-q_1)}}},
\end{align*}
so that by defining
\begin{align*}
&\Fcal_{xy}(Pr_1(v_1v_2))(\xi,q) \coloneqq \\ & \frac{1}{(2\pi)^2} \int_{\R} \sum_{q_1 \in \Z} (2 \chi_{\set{\abs{(\xi_1,q_1)} \geq 100 \abs{(\xi-\xi_1,q-q_1)} \ \text{or} \ \abs{(\xi,q)} \geq 100 \abs{(\xi-\xi_1,q-q_1)}}} \\ & - \chi_{\set{ \text{either} \ \abs{(\xi_1,q_1)} < 100 \abs{(\xi-\xi_1,q-q_1)} \ \text{or} \ \abs{(\xi,q)} < 100 \abs{(\xi-\xi_1,q-q_1)}}}) \\ & \cdot \Fcal_{xy}v_1(\xi_1,q_1) \Fcal_{xy}v_2(\xi-\xi_1,q-q_1) \ \mathrm{d}\xi_1,
\end{align*}
we obtain that
\begin{align*}
& \langle \frac{\partial}{\partial x}Pr_1(I^su(t') u(t')), I^su(t') \rangle_{L_{xy}^2} = \\ & \langle Pr_{\set{\abs{(\xi_1,q_1)} \gg \abs{(\xi_2,q_2)}}}(I^su(t') \frac{\partial u}{\partial x} (t')),I^su(t') \rangle_{L^2_{xy}}.
\end{align*}
The desired identity now follows from the fact that
\begin{align*} 2(I^su(t') u(t')) &= (Pr_1+Pr_2+Pr_3)(I^su(t')u(t')) \\ &= Pr_1(I^su(t')u(t')) + Pr_2(I^su(t')u(t')) + Pr_3(I^su(t')u(t')),  
\end{align*}
and this concludes the proof.
\end{proof}

Now, we replace each $u$ in the integrals occuring in Lemma \ref{equalityfordifferencek>1} and Lemma \ref{equalityfordifferencek=1} with an extension $v \in X_{s,\frac{1}{2}+}$ such that
\[ v|_{\intcc{-\frac{\delta_k}{2},\frac{\delta_k}{2}} \times \R \times \T} = u, \quad \text{supp}(v) \subseteq \intcc{-\delta_k,\delta_k} \times \R \times \T, \]
and 
\[ \norm{v}_{X_{\sigma,\frac{1}{2}+}} \les \norm{u}_{X_{\sigma, \frac{1}{2}+}^{\frac{\delta_k}{2}}} \quad \forall \sigma \leq s \]
hold, and we show

\begin{lemma} \label{trivialintestimatek>1}
Let $k \in \N^{\geq 2}$, $t \in \intcc{0,\frac{\delta_k}{2}}$, and let $\alpha_1,...,\alpha_{k+1}$ be multi-indices as in Lemma \ref{equalityfordifferencek>1}. Then
\[ \left\lvert \int_{0}^{t} \langle \prod_{i=1}^{k+1} D^{\alpha_{i}}v(t'), I^{s} v(t') \rangle_{L^{2}_{xy}} \ \mathrm{d}t' \right\rvert \les_{\epsilon', \delta_{k}} \norm{\prod_{i=1}^{k+1}D^{\alpha_{i}}v}_{X_{0,-\frac{1}{2}+\epsilon'}} \norm{u(0)}_{H^s} \]
holds true for every $\epsilon' > 0$.
\end{lemma}

\begin{proof}
By proceeding as described in the proof of \eqref{cubictermZK} (using the Cauchy-Schwarz inequality and the continuity of the multiplication operator $M$), we obtain
\[ \left\lvert \int_{0}^{t} \langle \prod_{i=1}^{k+1} D^{\alpha_{i}}v(t'), I^{s} v(t') \rangle_{L^{2}_{xy}} \ \mathrm{d}t' \right\rvert \les_{\epsilon', \delta_{k}} \norm{\prod_{i=1}^{k+1}D^{\alpha_{i}}v}_{X_{0,-\frac{1}{2}+\epsilon'}} \norm{v}_{X_{s,\frac{1}{2}-}}. \]
Moreover, by the preliminary remark, we may further conclude that
\[ \norm{v}_{X_{s,\frac{1}{2}-}} \leq \norm{v}_{X_{s,\frac{1}{2}+}} \les \norm{u}_{X_{s,\frac{1}{2}+}^\frac{\delta_k}{2}} \les \norm{u(0)}_{H^s},  \]
where in the final step we made use of the continuous dependence of the solution on the initial data. This completes the proof.
\end{proof}

\begin{rem} \label{trivialintestimatesk=1}
The same argument can be applied in a completely analogous way to further estimate the integral expressions arising in Lemma \ref{equalityfordifferencek=1}. In this way, we obtain
\begin{align*} 
&\left\lvert \int_{0}^{t} \langle Pr_{\set{\abs{(\xi_1,q_1)} \gg \abs{(\xi_2,q_2)}}}(I^sv(t') \frac{\partial v}{\partial x} (t')), I^sv(t') \rangle_{L_{xy}^2} \ \mathrm{d}t' \right\rvert \les_{\epsilon',\delta_1}  \\ & \norm{Pr_{\set{\abs{(\xi_1,q_1)} \gg \abs{(\xi_2,q_2)}}}(I^sv \frac{\partial v}{\partial x})}_{X_{0,-\frac{1}{2}+\epsilon'}} \norm{u(0)}_{H^s},
\end{align*}

\begin{align*}
&\left\lvert \int_{0}^{t} \langle  \frac{\partial}{\partial x}Pr_{2}(I^sv(t') v(t')), I^sv(t') \rangle_{L_{xy}^2} \ \mathrm{d}t' \right\rvert \les_{\epsilon',\delta_1} \\ & \norm{\frac{\partial}{\partial x}Pr_{2}(I^sv v)}_{X_{0,-\frac{1}{2}+\epsilon'}} \norm{u(0)}_{H^s},
\end{align*}

\begin{align*}
&\left\lvert \int_{0}^{t} \langle  \frac{\partial}{\partial x}Pr_{3}(I^sv(t') v(t')), I^sv(t') \rangle_{L_{xy}^2} \ \mathrm{d}t' \right\rvert \les_{\epsilon',\delta_1} \\ & \norm{\frac{\partial}{\partial x}Pr_{3}(I^sv v)}_{X_{0,-\frac{1}{2}+\epsilon'}} \norm{u(0)}_{H^s},
\end{align*}

and

\begin{align*}
&\left\lvert \int_{0}^{t} \langle \frac{\partial}{\partial x} (D^{\alpha_{1}}v(t') D^{\alpha_{2}}v(t')), I^{s}v(t') \rangle_{L_{xy}^{2}} \ \mathrm{d}t'  \right\rvert \les_{\epsilon', \delta_1} \\ & \norm{\frac{\partial}{\partial x} (D^{\alpha_{1}}v D^{\alpha_{2}}v)}_{X_{0,-\frac{1}{2}+\epsilon'}} \norm{u(0)}_{H^s}
\end{align*}

for every $\epsilon' > 0$.
\end{rem}

As an immediate consequence of Lemma \ref{equalityfordifferencek>1} and Lemma \ref{trivialintestimatek>1}, we obtain the following result

\begin{cor} \label{corogrowthk>1}
Let $k \in \N^{\geq2}$ and $t \in \intcc{0,\frac{\delta_{k}}{2}}$. Furthermore, let $\epsilon' > 0$ be arbitrary. Then, there exists a constant $C=C(k,\epsilon',\delta_{k}) > 0$ such that
\[ \vertiii{u(t)}_{H^{s}} \leq \vertiii{u(0)}_{H^{s}} + C \sum_{\substack{\abs{\alpha_{1}} +...+\abs{\alpha_{k+1}} = s+1 \\ \abs{\alpha_{1}} \geq ... \geq \abs{\alpha_{k+1}} \\ \abs{ \alpha_{2}} \geq 1}} \norm{\prod_{i=1}^{k+1}D^{\alpha_{i}} v}_{X_{0,-\frac{1}{2}+\epsilon'}}. \]
Here, $\vertiii{\cdot}_{H^s}$ denotes the norm defined by $\vertiii{f}_{H^s}^2 \coloneqq \norm{f}_{L^2_{xy}}^2 + \norm{f}_{\dot{H}^s}^2$.
\end{cor}

\begin{proof}
The combination of Lemma \ref{equalityfordifferencek>1} and Lemma \ref{trivialintestimatek>1} immediately leads us to 
\[ \norm{u(t)}_{\dot{H}^{s}}^{2} - \norm{ u(0)}_{\dot{H}^{s}}^{2} \leq \widetilde{C}(k,\epsilon',\delta_{k}) \norm{u(0)}_{H^{s}} \sum_{\substack{ 0 \leq \lvert \alpha_{i} \rvert \leq s \\ \lvert \alpha_{1} \rvert + ...+ \lvert \alpha_{k+1} \rvert = s+1}} \norm{\prod_{i=1}^{k+1} D^{\alpha_{i}} v }_{X_{0,-\frac{1}{2}+\epsilon'}}. \] 
Now, taking into account 
\[ \norm{ u(0) }_{H^{s}} \lesssim \vertiii{u(0)}_{H^{s}} \]
and the fact that the $L_{xy}^{2}$-norm is a conserved quantity, we further obtain
\begin{align*} 
&\underbrace{\norm{u(t)}_{\dot{H}^{s}}^{2} + \norm{u(t)}_{L_{xy}^{2}}^{2}}_{= \ \vertiii{u(t)}_{H^{s}}^{2}} - \underbrace{\left( \norm{u(0)}_{L_{xy}^{2}}^{2} + \norm{u(0)}_{\dot{H}^{s}}^{2} \right)}_{= \ \vertiii{u(0)}_{H^{s}}^{2}} \leq \\ & C(k,\epsilon',\delta_k) \vertiii{u(0)}_{H^{s}} \sum_{\substack{ 0 \leq \lvert \alpha_{i} \rvert \leq s \\ \lvert \alpha_{1} \rvert + ...+ \lvert \alpha_{k+1} \rvert = s+1}} \norm{\prod_{i=1}^{k+1}D^{\alpha_{i}} v }_{X_{0,-\frac{1}{2}+\epsilon'}},
\end{align*} 
and the desired estimate follows after rearranging the sum and dividing by \\ $\vertiii{u(t)}_{H^s} + \vertiii{u(0)}_{H^s}$.
\end{proof}

\begin{rem} \label{corogrowthk=1}
Using Lemma \ref{equalityfordifferencek=1} together with Remark \ref{trivialintestimatesk=1}, one arrives analogously at
\begin{align*}
&\vertiii{u(t)}_{H^s} \leq \vertiii{u(0)}_{H^s} + C(\epsilon',\delta_1) \left(\norm{Pr_{\set{\abs{(\xi_1,q_1)} \gg \abs{(\xi_2,q_2)}}}(I^sv \frac{\partial v}{\partial x})}_{X_{0,-\frac{1}{2}+\epsilon'}} \right. \\ &+ \left. \sum_{j=2}^{3} \norm{\frac{\partial}{\partial x}Pr_{j}(I^sv v)}_{X_{0,-\frac{1}{2}+\epsilon'}} + \sum_{\substack{\abs{\alpha_1}+\abs{\alpha_2}=s \\ \abs{\alpha_1} \geq \abs{\alpha_2} \geq 1}} \norm{\frac{\partial}{\partial x} (D^{\alpha_{1}}v D^{\alpha_{2}}v)}_{X_{0,-\frac{1}{2}+\epsilon'}} \right)
\end{align*}
when $k=1$.
\end{rem}

It remains to control the $X_{0,-\frac{1}{2}+}$-norms appearing in Corollary \ref{corogrowthk>1} and Remark \ref{corogrowthk=1} appropriately by powers of $\norm{u(0)}_{H^s}$. Since the arguments for the case $k=1$ and the higher-order nonlinearities differ substantially, we consider them separately.

\subsection{Key estimate in the case $k=1$}

\begin{prop} \label{keypropgrowthk=1}
For every $0<\epsilon \ll 1$, there exist $0<\epsilon' \ll 1$ and a constant $C=C(\epsilon,\sup_{t\in \R} \norm{u(t)}_{H^1}) > 0$ such that the inequality
\begin{align*}
&\norm{Pr_{\set{\abs{(\xi_1,q_1)} \gg \abs{(\xi_2,q_2)}}}(I^sv \frac{\partial v}{\partial x})}_{X_{0,-\frac{1}{2}+\epsilon'}} + \sum_{j=2}^{3} \norm{\frac{\partial}{\partial x}Pr_{j}(I^sv v)}_{X_{0,-\frac{1}{2}+\epsilon'}} \\ & + \sum_{\substack{\abs{\alpha_1}+\abs{\alpha_2}=s \\ \abs{\alpha_1} \geq \abs{\alpha_2} \geq 1}} \norm{\frac{\partial}{\partial x} (D^{\alpha_{1}}v D^{\alpha_{2}}v)}_{X_{0,-\frac{1}{2}+\epsilon'}} \leq C\norm{u(0)}_{H^s}^{1-\frac{1}{4(s-1)}+\epsilon}
\end{align*}
holds.
\end{prop}

\begin{proof}
In the following computations, we may assume without loss of generality that $\widehat{v},\widehat{f} \geq 0$ ($f\in X_{0,\frac{1}{2}-}$ with $\norm{f}_{X_{0,\frac{1}{2}-}} \leq 1$ for subsequent estimates by duality). Moreover, throughout the sequel, the symbol $\ast$ will refer to the convolution constraint $(\tau_0,\xi_0,q_0) = (\tau_1+\tau_2,\xi_1+\xi_2,q_1+q_2)$, and we proceed term by term. \\
For the first term, duality and Parseval's identity yield
\begin{align*}
&\norm{Pr_{\set{\abs{(\xi_1,q_1)} \gg \abs{(\xi_2,q_2)}}}(I^sv \frac{\partial v}{\partial x})}_{X_{0,-\frac{1}{2}+}} \les \\ & \sup_{\norm{f}_{X_{0,\frac{1}{2}-}} \leq 1} \int_{\R^2} \sum_{q_0 \in \Z} \int_{\R^2} \sum_{\substack{q_1 \in \Z \\ \ast}} \chi_{\set{\abs{(\xi_1,q_1)} \geq 100\abs{(\xi_2,q_2)}}} \abs{(\xi_1,q_1)}^s \abs{\xi_2} \\ & \cdot \widehat{v}(\tau_1,\xi_1,q_1) \widehat{v}(\tau_2,\xi_2,q_2) \widehat{f}(\tau_0,\xi_0,q_0) \ \mathrm{d}(\tau_1,\xi_1) \mathrm{d}(\tau_0,\xi_0) \eqqcolon \sup_{\norm{f}_{X_{0,\frac{1}{2}-}} \leq 1} I_{f,1},
\end{align*}
and since $\abs{(\xi_1,q_1)} \gg \abs{(\xi_2,q_2)}$, we may pass pointwise to
\[ \abs{(\xi_1,q_1)}^s \abs{\xi_2} \les \abs{\abs{(\xi_1,q_1)}^2 - \abs{(\xi_2,q_2)}^2}^\frac{1}{2} \langle (\xi_1,q_1) \rangle^{s-\frac{1}{2}+} \langle (\xi_2,q_2) \rangle^{\frac{1}{2}-}. \]
Undoing Plancherel, followed by Hölder's inequality then leads us to
\begin{align*}
I_{f,1} & \les \norm{MP_1(J^{s-\frac{1}{2}+}v,J^{\frac{1}{2}-}v)}_{L_{txy}^2} \norm{f}_{L_{txy}^2} \\ & \les \norm{v}_{X_{s-\frac{1}{2}+,\frac{1}{2}+}} \norm{v}_{X_{1,\frac{1}{2}+}} \norm{f}_{X_{0,\frac{1}{2}-}},
\end{align*}
where in the last step we made use of \eqref{MPlambda}.
By the choice of $v$ and the continuous dependence of the solution on the initial data, it further follows that
\begin{align*}
...&\les \norm{u}_{X_{s-\frac{1}{2}+,\frac{1}{2}+}^\frac{\delta_1}{2}} \norm{u}_{X_{1,\frac{1}{2}+}^\frac{\delta_1}{2}} \\ & \les \norm{u(0)}_{H^{s-\frac{1}{2}+}} \norm{u(0)}_{H^1},
\end{align*}
and since $\sup_{t \in \R} \norm{u(t)}_{H^1} < \infty$, we conclude that, overall
\[ \norm{Pr_{\set{\abs{(\xi_1,q_1)} \gg \abs{(\xi_2,q_2)}}}(I^sv \frac{\partial v}{\partial x})}_{X_{0,-\frac{1}{2}+}} \les \norm{u(0)}_{H^{s-\frac{1}{2}+}}.  \]
For the second term, we again use Parseval's identity and duality to write
\begin{align*}
&\norm{\frac{\partial}{\partial x}Pr_{2}(I^sv v)}_{X_{0,-\frac{1}{2}+}} \les \\ & \sup_{\norm{f}_{X_{0,\frac{1}{2}-}} \leq 1} \int_{\R^2} \sum_{q_0 \in \Z} \int_{\R^2} \sum_{\substack{q_1 \in \Z \\ \ast}} \chi_{\set{\abs{(\xi_1,q_1)} < 100 \abs{(\xi_2,q_2)} \ \text{and} \ \abs{(\xi_0,q_0)} < 100 \abs{(\xi_2, q_2)}}} \abs{\xi_0} \abs{(\xi_1,q_1)}^s \\ & \cdot \widehat{v}(\tau_1,\xi_1,q_1) \widehat{v}(\tau_2,\xi_2,q_2) \widehat{f}(\tau_0,\xi_0,q_0) \ \mathrm{d}(\tau_1,\xi_1)  \mathrm{d}(\tau_0,\xi_0) \eqqcolon \sup_{\norm{f}_{X_{0,\frac{1}{2}-}} \leq 1} I_{f,2},
\end{align*}
and we now distinguish two cases: \\
(i) \underline{$\abs{(\xi_1,q_1)} \gg \abs{(\xi_0,q_0)}$}: \\
In this situation, we may infer the pointwise bound
\[ \abs{\xi_0} \abs{(\xi_1,q_1)}^s \les \abs{\abs{(\xi_1,q_1)}^2 - \abs{(\xi_0,q_0)}^2}^\frac{1}{2} \langle (\xi_1,q_1) \rangle^{s-\frac{1}{2}+} \langle (\xi_0,q_0) \rangle^{-\frac{1}{2}-} \langle (\xi_2,q_2) \rangle \]
due to the active constraint $\abs{(\xi_0,q_0)} \les \abs{(\xi_2,q_2)}$.
An application of Parseval's identity, followed by Hölder's inequality then yields
\begin{align*}
I_{f,2} & \les \norm{MP_1(J^{s-\frac{1}{2}+}\widetilde{v},J^{-\frac{1}{2}-}f)}_{L_{txy}^2} \norm{J^1v}_{L_{txy}^2} \\ & \les \norm{v}_{X_{s-\frac{1}{2}+,\frac{1}{2}+}} \norm{f}_{X_{0,\frac{1}{2}-}} \norm{v}_{X_{1,\frac{1}{2}+}},
\end{align*}
where in the final step we used \eqref{MPlambdadual}. Now, taking into account the choice of $v$, the continuous dependence of the solution on the initial data, and the fact that $\sup_{t \in \R} \norm{u(t)}_{H^1} < \infty$, we can continue the calculation to obtain
\[...\les \norm{u(0)}_{H^{s-\frac{1}{2}+}}, \]
and this concludes the analysis of this subcase. \\
(ii) \underline{$\abs{(\xi_1,q_1)} \les \abs{(\xi_0,q_0)}$}: \\
In this case, taking into account the active assumption $\abs{(\xi_1,q_1)} \les \abs{(\xi_2,q_2)}$, we obtain
\[ \abs{\xi_0} \abs{(\xi_1,q_1)}^s \les \abs{\xi_0} \langle (\xi_0,q_0) \rangle^{\frac{3}{4}+} \langle (\xi_1,q_1) \rangle^{s-1} \langle (\xi_2,q_2) \rangle^{\frac{1}{4}-}. \]
Reversing the duality argument then gives
\[ \sup_{\norm{f}_{X_{0,\frac{1}{2}-}} \leq 1} I_{f,2} \les \norm{I_x^1(J^{s-1}v J^{\frac{1}{4}-}v)}_{X_{\frac{3}{4}+,-\frac{1}{2}+}}, \]
and this quantity can now be estimated by means of the bilinear estimate \eqref{lwpZKinH^slambda} to further obtain
\[... \les \norm{v}_{X_{s-\frac{1}{4}+,\frac{1}{2}+}} \norm{v}_{X_{1,\frac{1}{2}+}}. \]
Proceeding as before, we finally arrive at
\[... \les \norm{u(0)}_{H^{s-\frac{1}{4}+}}, \]
and this completes the analysis of this subcase as well. \\
Hence, for the second term, we have shown that
\[ \norm{\frac{\partial}{\partial x}Pr_{2}(I^sv v)}_{X_{0,-\frac{1}{2}+}} \les \norm{u(0)}_{H^{s-\frac{1}{4}+}} \]
overall. \\
We now turn to the third term. By duality and Plancherel, we again obtain
\begin{align*}
&\norm{\frac{\partial}{\partial x}Pr_{3}(I^sv v)}_{X_{0,-\frac{1}{2}+}} \les \\ & \sup_{\norm{f}_{X_{0,\frac{1}{2}-}} \leq 1} \int_{\R^2} \sum_{q_0 \in \Z} \int_{\R^2} \sum_{\substack{q_1 \in \Z \\ \ast}} \chi_{\set{\text{either} \ \abs{(\xi_1,q_1)} < 100 \abs{(\xi_2,q_2)} \ \text{or} \ \abs{(\xi_0,q_0)} < 100 \abs{(\xi_2, q_2)}}}  \abs{\xi_0} \\ & \cdot \abs{(\xi_1,q_1)}^s \widehat{v}(\tau_1,\xi_1,q_1) \widehat{v}(\tau_2,\xi_2,q_2) \widehat{f}(\tau_0,\xi_0,q_0) \ \mathrm{d}(\tau_1,\xi_1)  \mathrm{d}(\tau_0,\xi_0) \eqqcolon \sup_{\norm{f}_{X_{0,\frac{1}{2}-}} \leq 1} I_{f,3},
\end{align*}
and the frequency constraint imposed by the projector requires us to consider two cases: \\
(i) \underline{$\abs{(\xi_1,q_1)} < 100 \abs{(\xi_2,q_2)}$ and $\abs{(\xi_0,q_0)} \geq 100 \abs{(\xi_2,q_2)}$}: \\
In this setting, we have $\abs{(\xi_1,q_1)} \les \abs{(\xi_0,q_0)}$, which implies pointwise that
\[ \abs{\xi_0} \abs{(\xi_1,q_1)}^s \les \abs{\xi_0} \langle (\xi_0,q_0) \rangle^{\frac{3}{4}+} \langle (\xi_1,q_1) \rangle^{s-1} \langle (\xi_2,q_2) \rangle^{\frac{1}{4}-}.  \]
Using \eqref{lwpZKinH^slambda}, exactly as in the discussion of case (ii) for the second term, we then obtain
\begin{align*}
\sup_{\norm{f}_{X_{0,\frac{1}{2}-}} \leq 1} I_{f,3} &\les \norm{I_x^1(J^{s-1}vJ^{\frac{1}{4}-}v)}_{X_{\frac{3}{4}+,-\frac{1}{2}+}} \\ & \les \norm{v}_{X_{s-\frac{1}{4}+,\frac{1}{2}+}} \norm{v}_{X_{1,\frac{1}{2}+}},
\end{align*}
and the choice of $v$, in conjunction with the continuous dependence of the solution on the initial data, allows us to further conclude
\[ ... \les \norm{u(0)}_{H^{s-\frac{1}{4}+}}. \]
(ii) \underline{$\abs{(\xi_1,q_1)} \geq 100 \abs{(\xi_2,q_2)}$ and $\abs{(\xi_0,q_0)} < 100 \abs{(\xi_2,q_2)}$}: \\
In this case, we also have $\abs{(\xi_0,q_0)} \les \abs{(\xi_1,q_1)}$, and hence we may pass to
\[ \abs{\xi_0} \abs{(\xi_1,q_1)}^s \les \abs{\abs{(\xi_1,q_1)}^2-\abs{(\xi_2,q_2)}^2}^\frac{1}{2} \langle (\xi_1,q_1) \rangle^{s-\frac{1}{2}+} \langle (\xi_2,q_2) \rangle^{\frac{1}{2}-}. \]
We are therefore in the same situation as in the discussion of the first term and obtain, in exactly the same way that
\begin{align*}
I_{f,3} & \les \norm{MP_1(J^{s-\frac{1}{2}+}v,J^{\frac{1}{2}-}v)}_{L_{txy}^2} \norm{f}_{L_{txy}^2} \\ & \les \norm{v}_{X_{s-\frac{1}{2}+,\frac{1}{2}+}} \norm{v}_{X_{1,\frac{1}{2}+}} \norm{f}_{X_{0,\frac{1}{2}-}} \\ & \les \norm{u(0)}_{H^{s-\frac{1}{2}+}}.
\end{align*}
Altogether, we have thus shown for the third term that
\[ \norm{\frac{\partial}{\partial x}Pr_{3}(I^sv v)}_{X_{0,-\frac{1}{2}+}} \les \norm{u(0)}_{H^{s-\frac{1}{4}+}}. \]
It remains to analyze the terms appearing in the sum over $\alpha_1,\alpha_2$. To this end, we may assume without loss of generality that $\abs{(\xi_1,q_1)} \geq \abs{(\xi_2,q_2)}$ and restrict our attention to the contribution of
\begin{align*}
&\norm{I_x^1(I^{s-1}vI^1v)}_{X_{0,-\frac{1}{2}+}} \sim \sup_{\norm{f}_{X_{0,\frac{1}{2}-}} \leq 1} \int_{\R^2} \sum_{q_0 \in \Z} \int_{\R^2} \sum_{\substack{q_1 \in \Z \\ \ast}}  \abs{\xi_0} \abs{(\xi_1,q_1)}^{s-1} \abs{(\xi_2,q_2)} \\ & \cdot \widehat{v}(\tau_1,\xi_1,q_1) \widehat{v}(\tau_2,\xi_2,q_2) \widehat{f}(\tau_0,\xi_0,q_0) \ \mathrm{d}(\tau_1,\xi_1)  \mathrm{d}(\tau_0,\xi_0) \eqqcolon \sup_{\norm{f}_{X_{0,\frac{1}{2}-}} \leq 1} I_{f}, 
\end{align*}
as all remaining terms can be reduced to this case. \\
(i) \underline{$\abs{(\xi_1,q_1)} \gg \abs{(\xi_2,q_2)}$}: \\
In this case, it holds that
\[ \abs{\xi_0} \abs{(\xi_1,q_1)}^{s-1} \abs{(\xi_2,q_2)} \les \abs{\abs{(\xi_1,q_1)}^2 - \abs{(\xi_2,q_2)}^2}^\frac{1}{2} \langle (\xi_1,q_1) \rangle^{s-\frac{1}{2}+} \langle (\xi_2,q_2) \rangle^{\frac{1}{2}-}, \]
so that we can employ \eqref{MPlambda} exactly as in the discussion for the first term to obtain
\[ I_{f} \les \norm{u(0)}_{H^{s-\frac{1}{2}+}}. \]
(ii) \underline{$\abs{(\xi_1,q_1)} \sim \abs{(\xi_2,q_2)} \gg \abs{(\xi_0,q_0)}$}: \\
In this situation, we have
\begin{align*}
 \abs{\xi_0} \abs{(\xi_1,q_1)}^{s-1} \abs{(\xi_2,q_2)} &\les \abs{\abs{(\xi_1,q_1)}^2-\abs{(\xi_0,q_0)}^2}^\frac{1}{2} \langle (\xi_1,q_1) \rangle^{s-\frac{1}{2}+} \langle (\xi_0,q_0) \rangle^{-\frac{1}{2}-} \\ & \ \ \ \cdot \langle (\xi_2,q_2) \rangle, \end{align*}
which allows us to argue exactly as in the discussion of case (i) for the second term. By means of \eqref{MPlambdadual}, we thus obtain
\[ I_f \les \norm{u(0)}_{H^{s-\frac{1}{2}+}}, \]
and this concludes the treatment of this subcase. \\
(iii) \underline{$\abs{(\xi_1,q_1)} \sim \abs{(\xi_2,q_2)} \sim \abs{(\xi_0,q_0)}$}: \\
In this final subcase, the derivatives can be distributed arbitrarily among the factors, yielding
\[ \abs{\xi_0} \abs{(\xi_1,q_1)}^{s-1} \abs{(\xi_2,q_2)} \les \abs{\xi_0} \langle (\xi_0,q_0) \rangle^{\frac{3}{4}+} \langle (\xi_1,q_1) \rangle^{s-1} \langle (\xi_2,q_2) \rangle^{\frac{1}{4}-}. \]
We therefore find ourselves precisely in the setting of case (ii) (concerning the treatment of $Pr_2$), and, using \eqref{lwpZKinH^slambda}, we analogously obtain
\[ I_f \les \norm{u(0)}_{H^{s-\frac{1}{4}+}}. \] 
Hence, all terms to be examined have been bounded by $\norm{u(0)}_{H^{s-\frac{1}{4}+}}$,
and since
\[ \norm{u(0)}_{H^{s-\frac{1}{4}+}} \les \norm{u(0)}_{H^s}^{1-\theta} \norm{u(0)}_{H^1}^{\theta} \les \norm{u(0)}_{H^s}^{1-\theta} \]
holds with
\[s(1-\theta) + \theta = s-\frac{1}{4}+ \Leftrightarrow \theta = \frac{1}{4(s-1)}-, \]
the proof is complete.
\end{proof}

\subsection{Key estimate in the case $k\geq2$}

\begin{prop} \label{keypropgrowthk>1}
Let $k \in \N^{\geq2}$ and let $\alpha_1,...,\alpha_{k+1} \in \N^2_0$ be multi-indices as in Corollary \ref{corogrowthk>1}. Then, for every $0<\epsilon \ll 1$, there exist $0<\epsilon' \ll 1$ and a constant $C=C(k, \epsilon, \sup_{t \in \R} \norm{u(t)}_{H^1}) > 0$ such that the estimate
\[ \norm{\prod_{i=1}^{k+1}D^{\alpha_{i}} v}_{X_{0,-\frac{1}{2}+\epsilon'}} \leq C \norm{u(0)}_{H^s}^{1-\frac{1}{s-1}+\epsilon} \]
holds.
\end{prop}

\begin{proof}
In the following, we denote by $\ast$ the convolution constraint $(\tau_0,\xi_0,q_0) = (\tau_1+...+\tau_{k+1},\xi_1+...+\xi_{k+1},q_1+...+q_{k+1})$, and we may assume without loss of generality that all time-space Fourier transforms of the functions involved are nonnegative. By symmetry, we may further restrict our attention to the frequency configuration $\abs{(\xi_1,q_1)} \geq \abs{(\xi_2,q_2)} \geq...\geq\abs{(\xi_{k+1},q_{k+1})}$, which moreover allows us to identify the terms with $\abs{\alpha_1} = s$, $\abs{\alpha_2} = 1$, and $\abs{\alpha_i} = 0$ for all $ i \in \set{3,4,5,...}$ as the worst cases, to which all other terms can be reduced. \\
By duality and an application of Parseval's identity, we now write
\begin{align*}
&\norm{I^svI^1vv^{k-1}}_{X_{0,-\frac{1}{2}+}} \sim \sup_{\norm{f}_{X_{0,\frac{1}{2}-}} \leq 1} \int_{\R^2} \sum_{q_0 \in \Z} \int_{\R^{2k}} \sum_{\substack{q_1,...,q_k \in \Z \\ \ast}} \abs{(\xi_1,q_1)}^s \abs{(\xi_2,q_2)} \\ & \cdot \widehat{f}(\tau_0,\xi_0,q_0) \prod_{i=1}^{k+1} \widehat{v}(\tau_i,\xi_i,q_i) \ \mathrm{d}(\tau_1,...,\tau_k,\xi_1,...,\xi_k) \mathrm{d}(\tau_0,\xi_0) \eqqcolon \sup_{\norm{f}_{X_{0,\frac{1}{2}-}} \leq 1} I_f,
\end{align*}
and proceed with a case-by-case analysis in order to bound this expression: \\ \\
\underline{$k=2$}: \\
(i) \underline{$\abs{(\xi_1,q_1)} \gg \abs{(\xi_3,q_3)}$}: \\
In this case, we have that
\begin{align*}
\abs{(\xi_1,q_1)}^s \abs{(\xi_2,q_2)} &\les \abs{\abs{(\xi_1,q_1)}^2-\abs{(\xi_3,q_3)}^2}^\frac{1}{2} \langle (\xi_1,q_1) \rangle^{s-1+} \langle (\xi_2,q_2) \rangle^{1-} \langle (\xi_0,q_0) \rangle^{0-},
\end{align*}
and undoing Plancherel, followed by an application of Hölder's inequality leads us to
\begin{align*}
I_f & \les \norm{MP_1(J^{s-1+}v,v)}_{L_{txy}^2} \norm{J^{1-} \widetilde{v} J^{0-}f}_{L_{txy}^2} \\ & \leq \norm{MP_1(J^{s-1+}v,v)}_{L_{txy}^2} \norm{J^{1-}\widetilde{v}}_{L_{txy}^{4+}} \norm{J^{0-}f}_{L_{txy}^{4-}}.
\end{align*}
By applying \eqref{MPlambda} to the first factor and \eqref{L^4+lambda} and \eqref{L^4-lambda} to the second and third factors, respectively, it then follows that
\begin{align*}
...&\les \norm{v}_{X_{s-1+,\frac{1}{2}+}} \norm{v}_{X_{\frac{1}{2}+,\frac{1}{2}+}} \norm{v}_{X_{1,\frac{1}{2}+}} \norm{f}_{X_{0,\frac{1}{2}-}} \\ & \les \norm{v}_{X_{s-1+,\frac{1}{2}+}} \norm{v}_{X_{1,\frac{1}{2}+}}^2,
\end{align*}
and the choice of $v$, together with the continuous dependence of the solution on the initial data and the fact that $\sup_{t \in \R} \norm{u(t)}_{H^1} < \infty$, allows us to further conclude
\[... \les \norm{u(0)}_{H^{s-1+}}. \]
This completes the discussion of this subcase. \\
(ii) \underline{$\abs{(\xi_1,q_1)} \sim \abs{(\xi_2,q_2)} \sim \abs{(\xi_3,q_3)}$}: \\
In this situation, we may pass pointwise to
\[ \abs{(\xi_1,q_1)}^s \abs{(\xi_2,q_2)} \les \langle (\xi_1,q_1) \rangle^{s-1+} \langle (\xi_2,q_2) \rangle^{1-} \langle (\xi_3,q_3) \rangle^{1-} \langle (\xi_0,q_0) \rangle^{0-}, \]
and, upon applying Parseval's identity to return to physical space, we employ Hölder's inequality to obtain
\begin{align*}
I_f &\les \norm{J^{s-1+}vJ^{1-}v}_{L_{txy}^2} \norm{J^{1-}\widetilde{v}J^{0-}f}_{L_{txy}^2} \\ & \leq \norm{J^{s-1+}v}_{L_{txy}^4} \norm{J^{1-}v}_{L_{txy}^4} \norm{J^{1-}\widetilde{v}}_{L_{txy}^{4+}} \norm{J^{0-}f}_{L_{txy}^{4-}}.
\end{align*}
We now apply \eqref{L^4lambda} to the first two factors and \eqref{L^4+lambda} and \eqref{L^4-lambda} to the third and fourth factors, respectively. It then follows that
\begin{align*}
...&\les \norm{v}_{X_{s-1+,\frac{1}{2}+}} \norm{v}_{X_{1,\frac{1}{2}+}}^2 \norm{f}_{X_{0,\frac{1}{2}-}} \\ & \leq \norm{v}_{X_{s-1+,\frac{1}{2}+}} \norm{v}_{X_{1,\frac{1}{2}+}}^2,
\end{align*}
and, as in case (i), we can further conclude
\[...\les \norm{u(0)}_{H^{s-1+}}. \]
With this, the discussion for the case $k=2$ is completed. \\ \\
\underline{$k=3$}: \\
At this point, we recall that each factor $v$ has compact support in time. This allows us to use \eqref{optimizedL^plambda}, \eqref{optimizedL^6+}, and \eqref{optimizedL^6-} in the computations that follow. \\
(i) \underline{$\abs{(\xi_1,q_1)} \gg \abs{(\xi_4,q_4)}$}: \\
The active assumption leads us to the pointwise estimate
\begin{align*}
\abs{(\xi_1,q_1)}^s \abs{(\xi_2,q_2)} & \les \abs{\abs{(\xi_1,q_1)}^2-\abs{(\xi_4,q_4)}^2}^\frac{1}{2} \langle (\xi_1,q_1) \rangle^{s-1+} \langle (\xi_2,q_2) \rangle^{1-} \langle (\xi_0,q_0) \rangle^{0-} \\ & \ \ \ \cdot \langle (\xi_3,q_3) \rangle^{0-},
\end{align*}
and returning to physical space via Plancherel, followed by Hölder's inequality, gives
\begin{align*}
I_f &\les \norm{MP_1(J^{s-1+}v,v)}_{L_{txy}^2} \norm{J^{1-}\widetilde{v}J^{0-}fJ^{0-}\widetilde{v}}_{L_{txy}^2} \\ & \les \norm{MP_1(J^{s-1+}v,v)}_{L_{txy}^2} \norm{J^{1-}\widetilde{v}}_{L_{txy}^{4+}} \norm{J^{0-}f}_{L_{txy}^{4-}} \norm{J^{0-}\widetilde{v}}_{L_t^\infty L_{xy}^\infty}.
\end{align*}
We now apply \eqref{MPlambda} to the first factor and \eqref{L^4+lambda} and \eqref{L^4-lambda} to the second and third factors, respectively, while the fourth factor is handled by the Sobolev embedding theorem. This leads us to
\begin{align*}
...&\les \norm{v}_{X_{s-1+,\frac{1}{2}+}} \norm{v}_{X_{\frac{1}{2}+,\frac{1}{2}+}} \norm{v}_{X_{1,\frac{1}{2}+}}^2 \norm{f}_{X_{0,\frac{1}{2}-}} \\ & \les \norm{v}_{X_{s-1+,\frac{1}{2}+}} \norm{v}_{X_{1,\frac{1}{2}+}}^3,
\end{align*}
and the argument then proceeds as discussed in the case $k=2$. Hence, we ultimately obtain
\[...\les \norm{u(0)}_{H^{s-1+}}, \]
and this concludes the treatment of this subcase. \\
(ii) \underline{$\abs{(\xi_1,q_1)} \sim \abs{(\xi_2,q_2)} \sim \abs{(\xi_3,q_3)} \sim \abs{(\xi_4,q_4)}$}: \\
In this situation, we may now distribute the factors arbitrarily among the factors, which yields
\[ \abs{(\xi_1,q_1)}^s \abs{(\xi_2,q_2)} \les \langle (\xi_1,q_1) \rangle^{s-\frac{4}{3}+} \langle (\xi_0,q_0) \rangle^{0-} \langle (\xi_2,q_2) \rangle^{\frac{7}{9}-} \langle (\xi_3,q_3) \rangle^{\frac{7}{9}-} \langle (\xi_4,q_4) \rangle^{\frac{7}{9}-}.  \]
Applying Parseval's identity and Hölder's inequality then gives
\begin{align*}
I_f &\les \norm{J^{s-\frac{4}{3}+}\widetilde{v}J^{0-}f}_{L_{txy}^2} \norm{J^{\frac{7}{9}-}v J^{\frac{7}{9}-}v J^{\frac{7}{9}-}v}_{L_{txy}^2} \\ & \leq \norm{J^{s-\frac{4}{3}+}\widetilde{v}}_{L_{txy}^{4+}} \norm{J^{0-}f}_{L_{txy}^{4-}} \norm{J^{\frac{7}{9}-}v}_{L_{Txy}^6} \norm{J^{\frac{7}{9}-}v}_{L_{Txy}^6} \norm{J^{\frac{7}{9}-}v}_{L_{Txy}^6},
\end{align*}
and for the first two factors, we use \eqref{L^4+lambda} and \eqref{L^4-lambda}, respectively, while for the remaining three factors, we rely on \eqref{optimizedL^plambda} ($p=6$).
It thus follows that
\begin{align*}
...&\les \norm{v}_{X_{s-\frac{4}{3}+,\frac{1}{2}+}} \norm{v}_{X_{1,\frac{1}{2}+}}^3 \norm{f}_{X_{0,\frac{1}{2}-}} \\ & \les \norm{v}_{X_{s-\frac{4}{3}+,\frac{1}{2}+}} \norm{v}_{X_{1,\frac{1}{2}+}}^3,
\end{align*}
and finally
\[...\les \norm{u(0)}_{H^{s-\frac{4}{3}+}}, \]
thereby completing the discussion of the intermediate case $k=3$. \\ \\
\underline{$k\geq 4$}: \\
(i) \underline{$\abs{(\xi_1,q_1)} \gg \abs{(\xi_5,q_5)}$}: \\
For the frequency configuration under consideration, we obtain
\begin{align*}
 \abs{(\xi_1,q_1)}^s \abs{(\xi_2,q_2)} &\les \abs{\abs{(\xi_1,q_1)}^2-\abs{(\xi_5,q_5)}^2}^\frac{1}{2} \langle (\xi_1,q_1) \rangle^{s-1+} \langle (\xi_2,q_2) \rangle^{1-} \langle (\xi_0,q_0) \rangle^{0-} \\ & \ \ \ \cdot \prod_{\substack{i=3 \\ i \neq 5}}^{k+1} \langle (\xi_i,q_i) \rangle^{0-}, 
\end{align*}
which, after Plancherel and Hölder's inequality, gives
\begin{align*}
I_f & \les \norm{MP_1(J^{s-1+}v,v)}_{L_{txy}^2} \norm{J^{1-}\widetilde{v} J^{0-}f \prod_{\substack{i=3 \\ i \neq 5}}^{k+1} J^{0-}\widetilde{v}}_{L_{txy}^2} \\ & \leq \norm{MP_1(J^{s-1+}v,v)}_{L_{txy}^2} \norm{J^{1-}\widetilde{v}}_{L_{txy}^{4+}} \norm{J^{0-}f}_{L_{txy}^{4-}} \prod_{\substack{i=3 \\ i \neq 5}}^{k+1} \norm{J^{0-}\widetilde{v}}_{L_t^\infty L_{xy}^\infty}.
\end{align*}
We now apply \eqref{MPlambda} to the first factor, \eqref{L^4+lambda} and \eqref{L^4-lambda} to the second and third factors, respectively, and the Sobolev embedding theorem to all remaining factors, yielding
\begin{align*}
...&\les \norm{v}_{X_{s-1+,\frac{1}{2}+}} \norm{v}_{X_{\frac{1}{2}+,\frac{1}{2}+}} \norm{v}_{X_{1,\frac{1}{2}+}}^{k-1} \norm{f}_{X_{0,\frac{1}{2}-}} \\ & \les \norm{v}_{X_{s-1+,\frac{1}{2}+}} \norm{v}_{X_{1,\frac{1}{2}+}}^{k}. 
\end{align*}
Now, taking into account the choice of $v$, the continuous dependence of the solution on the initial data, and the uniform boundedness of $\norm{u(t)}_{H^1}$, we finally obtain
\[...\les \norm{u(0)}_{H^{s-1+}}, \]
and this completes the discussion of this subcase. \\
(ii) \underline{$\abs{(\xi_1,q_1)} \sim \abs{(\xi_2,q_2)} \sim \abs{(\xi_3,q_3)} \sim \abs{(\xi_4,q_4)} \sim \abs{(\xi_5,q_5)}$}: \\
In this final subcase, we may infer the pointwise bound
\[ \abs{(\xi_1,q_1)}^s \abs{(\xi_2,q_2)} \les \langle (\xi_1,q_1) \rangle^{s-\frac{17}{9}+} \langle (\xi_0,q_0) \rangle^{-\frac{2}{9}-} \left( \prod_{i=2}^{5} \langle (\xi_i,q_i) \rangle^{\frac{7}{9}-} \right) \prod_{j=6}^{k+1} \langle (\xi_j,q_j) \rangle^{0-}, \]
which, after undoing Plancherel and an application of Hölder's inequality, leads us to
\begin{align*}
I_f &\les \norm{J^{s-\frac{17}{9}+}v}_{L_{Txy}^{6+}} \norm{J^{-\frac{2}{9}-}f}_{L_{Txy}^{6-}} \left( \prod_{i=2}^{5} \norm{J^{\frac{7}{9}-}v}_{L_{Txy}^6} \right) \prod_{j=6}^{k+1} \norm{J^{0-}v}_{L_t^\infty L_{xy}^\infty}.
\end{align*}
The first and second factors are now treated using \eqref{optimizedL^6+} and \eqref{optimizedL^6-}, respectively, while the third, fourth, fifth, and sixth factors can be estimated using \eqref{optimizedL^plambda} ($p=6$). All remaining factors are again handled by means of the Sobolev embedding theorem, yielding
\begin{align*}
...&\les \norm{v}_{X_{s-\frac{5}{3}+,\frac{1}{2}+}} \norm{v}_{X_{1,\frac{1}{2}+}}^k \norm{f}_{X_{0,\frac{1}{2}-}} \\ & \leq \norm{v}_{X_{s-\frac{5}{3}+,\frac{1}{2}+}} \norm{v}_{X_{1,\frac{1}{2}+}}^k
\end{align*}
and, subsequently,
\[...\les \norm{u(0)}_{H^{s-\frac{5}{3}+}}. \]
Gathering all the intermediate results, we see that for each $k \in \N^{\geq 2}$, the corresponding worst-case term is bounded by $\norm{u(0)}_{H^{s-1+}}$. Since
\[ \norm{u(0)}_{H^{s-1+}} \les \norm{u(0)}_{H^s}^{1-\theta} \norm{u(0)}_{H^1}^{\theta} \les \norm{u(0)}_{H^s}^{1-\theta}  \]
with
\[ s(1-\theta) + \theta = s-1+ \Leftrightarrow \theta = \frac{1}{s-1}-, \]
the assertion follows.
\end{proof}

We are now prepared to give the

\begin{proof}[Proof of Theorem \ref{GrowthgZK}]
We begin with the case $k=1$. To this end, let $\alpha = 4(s-1)+\widetilde{\epsilon}$ for an arbitrarily prescribed $0<\widetilde{\epsilon} \ll 1$. By combining Remark \ref{corogrowthk=1} and Proposition \ref{keypropgrowthk=1}, we obtain, for any $0<\epsilon \ll 1$, that
\[ \vertiii{u(t)}_{H^{s}} \leq \vertiii{u(0)}_{H^{s}} + \widetilde{C}(\eps, \sup_{t \in \mathbb{R}} \lVert u(t) \rVert_{H^{1}}) \vertiii{u(0)}_{H^{s}}^{1-\frac{1}{4(s-1)}+\epsilon} \]
holds for all $t \in \intcc{0,\frac{\delta_1(\norm{u_0}_{H^1})}{2}}$. Similarly, working through Lemma \ref{equalityfordifferencek=1}, Remark \ref{trivialintestimatesk=1}, Remark \ref{corogrowthk=1}, and Proposition \ref{keypropgrowthk=1} for an arbitrary $t_0 \in \R$ and fixed $t \in \intcc{t_0-\frac{\delta_1(\sup_{t\in \R} \norm{u(t)}_{H^1})}{2},t_0+\frac{\delta_1(\sup_{t\in \R} \norm{u(t)}_{H^1})}{2}}$ then leads us to
\[ \vertiii{u(t)}_{H^{s}} \leq \vertiii{u(t_{0})}_{H^{s}} + C(\eps, \sup_{t \in \mathbb{R}} \lVert u(t) \rVert_{H^{1}}) \vertiii{u(t_{0})}_{H^{s}}^{1-\frac{1}{4(s-1)}+\epsilon}, \]
with the constant $C$ being independent of $t_0$.
For $m \in \N_0$, we now set
\[ I_m \coloneqq \intcc{m \frac{\delta_{1}(\sup_{t \in \mathbb{R}} \lVert u(t) \rVert_{H^{1}})}{2}, (m+1) \frac{\delta_{1}(\sup_{t \in \mathbb{R}} \lVert u(t) \rVert_{H^{1}})}{2}} \]
and
\[ \widetilde{I}_m \coloneqq \intcc{-(m+1)\frac{\delta_{1}(\sup_{t \in \mathbb{R}} \lVert u(t) \rVert_{H^{1}})}{2}, -m \frac{\delta_{1}(\sup_{t \in \mathbb{R}} \lVert u(t) \rVert_{H^{1}})}{2}}, \]
and an application of the Grönwall-type Lemma \ref{Gronwall} to the sequences
\[ a_{m} \coloneqq \sup_{t \in I_m} \vertiii{u(t)}_{H^{s}}, \quad m \in \mathbb{N}_{0} \]
and
\[ \widetilde{a}_{m} \coloneqq \sup_{t \in \widetilde{I}_m} \vertiii{u(t)}_{H^{s}}, \quad m \in \mathbb{N}_{0} \]
finally yields
\[ \vertiii{u(t)}_{H^{s}} \leq C(\beta, \sup_{t \in \mathbb{R}} \lVert u(t) \rVert_{H^{1}}) (1+\lvert t \rvert)^{\beta} (1+ \vertiii{u_{0}}_{H^{s}}) \]
for all $t \in \R$, and in fact for every
\[ \beta > \frac{1}{\frac{1}{4(s-1)}-\eps}. \]
By choosing $0<\epsilon \ll 1$ sufficiently small, one can enforce
\[ \alpha > \frac{1}{\frac{1}{4(s-1)}-\eps}, \]
and taking into account the equivalence of the norms $\vertiii{\cdot}_{H^s}$ and $\norm{\cdot}_{H^s}$, this completes the argument in the case $k=1$. \\
For the cases $k\geq 2$, one proceeds in a completely analogous manner using the corresponding auxiliary results Lemma \ref{equalityfordifferencek>1}, Lemma \ref{trivialintestimatek>1}, Corollary \ref{corogrowthk>1}, and Proposition \ref{keypropgrowthk>1}, and ultimately obtains
\[ \norm{u(t)}_{H^{s}} \leq C(k, \beta,  \sup_{t \in \mathbb{R}} \lVert u(t) \rVert_{H^{1}}) (1+\lvert t \rvert)^{\beta} (1+ \norm{u_0}_{H^{s}})  \]
for all $t \in \R$ and every
\[ \beta > \frac{1}{\frac{1}{s-1}-\epsilon}. \]
By choosing $0<\epsilon \ll 1$ sufficiently small in this setting, any $\beta > s-1$ is admissible, and thus the proof is complete.
\end{proof}

By interpolation, the restriction $s \in 2\N$ in Theorem \ref{GrowthgZK} can be relaxed, provided the initial data are sufficiently regular. This simple argument is the content of the following

\begin{cor}
Let $k \in \N$,
\[ \alpha (k,s) \coloneqq \begin{cases} 4(s-1), & \text{if} \quad k=1, \\ s-1, & \text{if} \quad k\geq 2, \end{cases} \]
$s \in \intoo{1,2m}$ for some $m \in \mathbb{N}$, $u_{0} \in H^{2m}(\R \times \T)$, and let $u$ be the associated global solution of $\mathrm{(CP}_{k, \R \times \T} \mathrm{)}$, as stated in Theorem \ref{GrowthgZK}. Then, for every $\alpha > \alpha(k,s)$, there exists a constant $C=C(k,\alpha,\sup_{t\in \R} \norm{u(t)}_{H^1})>0$ such that
\[ \norm{u(t)}_{H^{s}(\mathbb{R} \times \T)} \lesssim C (1+\lvert t \rvert)^{\alpha} (1+ \norm{ u_{0} }_{H^{2m}(\mathbb{R} \times \T)}) \]
holds true for all $t \in \R$.
\end{cor}

\begin{proof}
We consider the case $k=1$. The cases $k\geq 2$ can be handled analogously.
Given $s \in \intoo{1,2m}$, there exists $\theta \in (0,1)$ such that 
\[ (1-\theta) + \theta 2m =s \Leftrightarrow \theta = \frac{s-1}{2m-1}. \]
Now, applying Theorem \ref{GrowthgZK} at the regularity level $2m$ and taking into account that $\sup_{t \in \R} \norm{u(t)}_{H^1} < \infty$ and $\frac{s-1}{2m-1} \leq 1$, we obtain 
\begin{align*} 
\norm{ u(t)}_{H^{s}} &\les \norm{ u(t)}_{H^{1}}^{1-\theta} \norm{ u(t)}_{H^{2m}}^{\theta} \\ &\les \norm{ u(t)}_{H^{1}}^{1-\theta} (1+\lvert t \rvert)^{\theta(4(2m-1) + \epsilon)} (1+ \norm{ u_{0}}_{H^{2m}})^{\theta} \\ &\les (1+ \lvert t \rvert)^{\theta(4(2m-1) + \epsilon)} (1+ \lVert u_{0} \rVert_{H^{2m}}) \\ &= (1+\lvert t \rvert)^{4(s-1) + \frac{s-1}{2m-1} \epsilon} (1+ \lVert u_{0} \rVert_{H^{2m}}) \\ &\leq (1+\lvert t \rvert)^{4(s-1) + \epsilon} (1+ \lVert u_{0} \rVert_{H^{2m}}) 
\end{align*}
for all $t \in \R$ and every $\epsilon > 0$, as required.
\end{proof}

\ew

\bibliographystyle{amsplain}

\bibliography{referencesGWP}

\end{document}